\documentclass[twoside,11pt]{amsart}
\usepackage[all]{xy}
\CompileMatrices
\usepackage{mathdots}
\usepackage{amsfonts}
\usepackage{amssymb}
\usepackage{amsthm}
\usepackage{amsmath}
\usepackage{ifthen}
\usepackage{enumitem}
\usepackage[usenames]{color}
\usepackage{comment}

\usepackage{xcolor}
\usepackage[hyperfootnotes=false, colorlinks, linkcolor={blue}, citecolor={magenta}, filecolor={blue}, urlcolor={blue}, plainpages=false, pdfpagelabels]{hyperref}

 \newtheorem{thm}{Theorem}[section]
 \newtheorem{prop}[thm]{Proposition}
 \newtheorem{cor}[thm]{Corollary}
 \newtheorem{lem}[thm]{Lemma}

\theoremstyle{definition}
\newtheorem{defn}[thm]{Definition}

\theoremstyle{remark}
\newtheorem{rem}[thm]{Remark}

\newcommand{\Z}{\mathbb{Z}}
\newcommand{\Q}{\mathbb{Q}}
\newcommand{\R}{\mathbb{R}}
\newcommand{\C}{\mathbb{C}}

  1

\newcommand{\transpose}[1]{\text{$^t\!#1$}}
\newcommand{\G}{\ifmmode {\mathcal{G}}\else${\mathcal{G}}$\ \fi}

\def\sectionnam{\@empty}
\def\subsectionnam{\@empty}

\numberwithin{equation}{section}

\begin{document}

\title[Algebraicity and the $p$-adic Interpolation of Special $L$-values] 
{Algebraicity and the $p$-adic Interpolation of Special $L$-values for certain Classical Groups}

\author{Yubo Jin}
\address{Institute for Advanced Study in Mathematics\\ Zhejiang University\\ Hangzhou, 310058, China}
\email{yubo.jin@zju.edu.cn}

\date{\today}

\maketitle

\begin{abstract}
In this paper, we calculate the ramified local integrals in the doubling method and present an integral representation of standard $L$-functions for classical groups. We explicitly construct local sections of Eisenstein series such that the local ramified integrals represent certain ramified $L$-factors. As an application, we prove algebraicity of special $L$-values and construct $p$-adic $L$-functions for symplectic, unitary, quaternionic unitary and quaternionic orthogonal groups.
\end{abstract}

\tableofcontents

\section{Introduction}

One of the central problems in number theory is the study of special values of $L$-functions. The main object studied in this paper is the standard $L$-function for automorphic forms on classical groups. We present an integral representation for these $L$-functions using the doubling method. Utilizing the integral representation, we prove the algebraicity of certain special $L$-values and construct $p$-adic $L$-functions interpolating these values. We introduce our results and compare with works in the literature in the following three subsections.

\subsection{Integral representations}

Let $G$ be a classical group over a number field $F$ defined as in \eqref{2.1.2} or \eqref{2.1.4}. The first main theme of this paper is an integral representation for standard $L$-functions of classical groups. One way to obtain such an integral representation is the doubling method originated in \cite{G84b, Piatetski-ShapiroRallis1987Doubling}. Take a cuspidal representation $\pi$ of $G(\mathbb{A})$ and a cusp form $\phi\in\pi$. We consider a doubling embedding \eqref{2.1.7} $G\times G\to H$ into a bigger classical group $H$ defined as in \eqref{2.1.6}. The main strategy of the doubling method is to pullback an Eisenstein series $E(h;f_s)$ on $H(\mathbb{A})$ induced from a character $\chi$ along such an embedding and to pair it with the cusp form on $G(\mathbb{A})$. That is, we consider the global integral of the form
\begin{equation}
\label{111}
\begin{aligned}
&\mathcal{Z}(s;\phi,f_s)\\
=&\int_{G(F)\times G(F)\backslash G(\mathbb{A})\times G(\mathbb{A})}E((g_1,g_2);f_s)\overline{\phi_1(g_1)}\phi_2(g_2)\chi(\nu(g_2))^{-1}dg_1dg_2,
\end{aligned}
\end{equation}
where $(g_1,g_2)$ is the image of $g_1,g_2\in G(\mathbb{A})$ in $H(\mathbb{A})$ and $\phi_1,\phi_2$ are certain translate of $\phi$.

It is shown in \cite{Piatetski-ShapiroRallis1987Doubling} that this global integral has an Euler product expression \eqref{2.2.11} and thus reduces the study of \eqref{111} to the study of local integrals place by place. It is also well known that one can make a proper choice of the section $f_s$ such that $\mathcal{Z}(s;\phi,f_s)$ represents the partial $L$-function of $\phi$, i.e. all the ramified local $L$-factors are set to $1$. For the study of arithmetic problems of special $L$-values and especially the construction of $p$-adic $L$-functions, the information at ramified places is indispensable.

The definition of local $L$-factors is indeed a fundamental problem in the study of automorphic representations. In \cite{L70}, Langlands conjectured that one can associate to any cuspidal representation $\pi=\otimes\pi_v$ a local $L$-factor $L_v(s,\pi_v)$ and an epsilon factor such that the global $L$-function $L(s,\pi)$ satisfies a functional equation. In \cite{Y14}, Yamana defined these local factors and proved the functional equation using the doubling method. His approach works for all irreducible automorphic representations of classical groups and is used in proving some analytic properties of $L$-functions. However, he did not construct the local sections of Eisenstein series and did not compute the local integrals explicitly so it is not clear how his computations can be used to study the algebraicity of special $L$-values or to construct $p$-adic $L$-functions.

We study the ramified local integrals in a different way which is inspired by \cite{Sh95}. Assume $\phi$ is fixed by some open compact subgroup $K(\mathfrak{n})$ \eqref{3.4.2} and is an eigenform for a certain Hecke algebra $\mathcal{H}(K(\mathfrak{n}),\mathfrak{X})$. For a Hecke character $\chi$ whose conductor divides $\mathfrak{n}$, we define the $L$-function $L(s,\phi\times\chi)$ to be a Dirichlet series of the Hecke eigenvalues of $\phi$. This extends the definition of the $L$-function for symplectic groups in \cite{Sh95} and is an analogue of the $L$-functions for classical (elliptic) modular forms defined by Dirichlet series of Fourier coefficients. In particular, the $L$-function $L(s,\phi\times\chi)$ has all bad Euler factors outside the conductor of $\chi$. In \cite{Sh95}, Shimura constructs local sections of Eisenstein series explicitly at all places such that $\mathcal{Z}(s;\phi,f_s)$ represents $L(s,\phi\times\chi)$. Our main theorem on integral representation, which extends his result to all classical groups, is stated as follow.

\begin{thm}
\label{theorem 1.1}
(Theorem \ref{theorem 2.2}, \ref{theorem 6.1}) There is a choice of $f_s$ such that
\begin{equation}
\mathcal{Z}(s;\phi,f_s)=C\cdot L\left(s+\frac{1}{2},\phi\times\chi\right)\cdot\mathcal{Z}_{\infty}(s;\phi_{\infty},f_s^{\infty})\cdot\langle\phi',\phi\rangle,
\end{equation}
where $C$ is some nonzero constant depending on $s$, $\phi'$ is some simple translate of $\phi$ and $\mathcal{Z}_{\infty}(s;\phi_{\infty},f_s^{\infty})$ a nonzero constant depending on the choice of the archimedean  section $f_s^{\infty}$. When the underlying symmetric space of $G$ is hermitian, and $\phi$ is a holomorphic cusp form as in Definition \ref{definition 5.5}, we can further make a choice of $f_s^{\infty}$ such that $\mathcal{Z}_{\infty}(s;\phi_{\infty},f_s^{\infty})$ is the constant given in Proposition \ref{proposition 5.4}.
\end{thm}

If $G$ is a unitary group, we assume all places $v|\mathfrak{n}$ are nonsplit in the quadratic extension defining the group $G$. This is only for simplicity and also because the split case is well studied in \cite{HLS} and \cite{EHLS}. The main difficulty for extending the result of \cite{Sh95} is to deal with the classical groups which are not totally isotropic (i.e. $r>0$ in \eqref{2.1.3}). In this case, the doubling map and the image of the doubling embedding \eqref{2.1.9} are much more involved which complicates the computations. For the purpose of constructing $p$-adic $L$-functions, our local sections are also properly chosen such that the Eisenstein series has a nice Fourier expansion. This is indeed the core technical issue of this work. We do not study the archimedean integral in general in this work. For the special cases we are considering, the archimedean computations follow from \cite{Sh97, Sh00} and \cite{JYB1}.

Recently, Cai, Friedberg, Ginzburg and Kaplan \cite{CaiFriedbergGinzburgKaplan2019} presented an integral representation for $\mathrm{Sp}_{2n}\times\mathrm{GL}_k$ by the twisted doubling method generalizing the classical doubling method studied in this paper. In \cite{Cai2021}, the unfolding of the global integral are also worked out for $G\times\mathrm{GL}_k$ with $G$ any classical group. It will also be important to study the ramified integrals derived from the twisted doubling method.  For example, one should expect that one can define the local $L$-factors and prove the functional equations for standard $L$-functions of $G\times\mathrm{GL}_k$ as in \cite{Y14}. It is also an interesting question whether one can construct local sections and compute the ramified integral for $G\times\mathrm{GL}_k$ explicitly as we have done here for $G\times\mathrm{GL}_1$.

\subsection{The algebraic result}

The celebrated Deligne's conjecture \cite{D79} claims that the critical values of motivic $L$-functions, up to certain periods, are algebraic numbers. In this paper, we study the automorphic counterpart of this conjecture. As the approach here relies heavily on the theory of Shimura varieties, we restrict ourselves to the classical group $G$ whose underlying symmetric space is hermitian. Such groups (except some orthogonal groups) are listed at the beginning of Section \ref{section 5}. In these cases, the notion of the algebraic modular forms is well defined. We refer the reader to the beginning of Section \ref{section 5.3} for a summary of various characterizations of algebraic modular forms in the literature. All of them rely on the fact that the symmetric space of $G$ is hermitian so that one can associate $G$ to a Shimura variety.

We fix the following setup. Assume $F$ is a totally real number field. Let $\boldsymbol{l}=(l,...,l)$ be a parallel weight satisfying
\begin{equation}
\label{1.3}
\begin{aligned}
l&\geq\left\{\begin{array}{cc}
m+1 & \text{ Case II,}\\
n+1 & \text{ Case III, IV, V.}\\
\end{array}\right.&&\text{ when }F\neq\Q,\\
l&\geq\left\{\begin{array}{cc}
m+1 & \text{ Case II,}\\
n+r+1 & \text{ Case III, IV, V,}\\
\end{array}\right.&&\text{ when }F=\Q,
\end{aligned}
\end{equation}
with $m,n,r$ as in \eqref{2.1.3}.
Fix a specific prime $\boldsymbol{p}$ of $F$ and an integral ideal $\mathfrak{n}=\mathfrak{n}_1\mathfrak{n}_2=\prod_{v}\mathfrak{p}_v^{\mathfrak{c}_v}$ with $\mathfrak{n}_1,\mathfrak{n}_2,\boldsymbol{p}$ coprime. We make the following assumptions:\\
(1) $2\in\mathcal{O}_v^{\times}$ and $\theta\in\mathrm{GL}_r(\mathcal{O}_v)$ for all $v|\mathfrak{n}\boldsymbol{p}$.\\
(2) $\boldsymbol{f}\in\mathcal{S}_{\boldsymbol{l}}(K(\mathfrak{n}\boldsymbol{p}),\overline{\Q})$ is an algebraic eigenform for the Hecke algebra $\mathcal{H}(K(\mathfrak{n}\boldsymbol{p}),\mathfrak{X})$ as in Section \ref{section 3.4}.\\
(3) $\boldsymbol{f}$ is an eigenform for the $U(\boldsymbol{p})$ operator defined in \eqref{3.2.8} with eigenvalue $\alpha(\boldsymbol{p})\neq 0$.\\
(4) $\chi=\chi_1\boldsymbol{\chi}$ with $\chi_1$ has conductor $\mathfrak{n}_2$ and $\boldsymbol{\chi}$ has conductor $\boldsymbol{p}^{\boldsymbol{c}}$ for some integer $\boldsymbol{c}\geq 0$. We assume $\chi$ has infinity type $\boldsymbol{l}$. That is, $\chi_v(x)=x^{l}|x|^{-l}$ for all $v|\infty$.\\
(5) When $G$ is the unitary group, we assume all places $v|\mathfrak{n}\boldsymbol{p}$ are nonsplit in the imaginary quadratic extension $E/F$ defining the unitary group.

We are interested in the special value of the $L$-function $L\left(s+\frac{1}{2},\boldsymbol{f}\times\chi\right)$ at \begin{equation}
\label{1.4}
s=s_0:=\left\{\begin{array}{cc}
l-\kappa & \text{ Case II, III, IV,}\\
\frac{l}{2}-\kappa & \text{ Case V}.
\end{array}\right.
\end{equation}
with $\kappa$ a constant depending on $n$ given in \eqref{2.2.5}. Our main theorem on algebraicity is stated as follow.
\begin{thm}
\label{theorem 1.2}
(Theorem \ref{theorem 8.1}) For $l,s_0$ as above,
\begin{equation}
\begin{aligned}
&\frac{L\left(s_0+\frac{1}{2},\boldsymbol{f}\times\chi\right)}{\pi^{d(F)\boldsymbol{d}(\pi)}\Omega\cdot\langle\boldsymbol{f},\boldsymbol{f}\rangle}\in\overline{\Q},\qquad & \text{ if }\boldsymbol{c}>0,\\
&\frac{L\left(s_0+\frac{1}{2},\boldsymbol{f}\times\chi\right)M\left(s_0+\frac{1}{2},\boldsymbol{f}\times\chi\right)}{\pi^{d(F)\boldsymbol{d}(\pi)}\Omega\cdot\langle\boldsymbol{f},\boldsymbol{f}\rangle}\in\overline{\Q},\qquad & \text{ if }\boldsymbol{c}=0,
\end{aligned}
\end{equation}
where $d(F)=[F:\Q]$, $\boldsymbol{d}(\pi)$ is the constant given in \eqref{8.4}, $\langle\cdot,\cdot\rangle$ is the Petersson inner product and $M\left(s_0+\frac{1}{2},\boldsymbol{f}\times\chi\right)$ is the modification factor listed in Proposition \ref{proposition 4.6}. Here $\Omega=1$ in Case II, III, IV and in Case V, $\Omega$ is the CM period \eqref{cmperiod} depending only on the group $G$.
\end{thm}

When the group is totally isotropic (i.e. $r=0$ in \eqref{2.1.3}), we also obtain a refined version of the above theorem. That is we describe the action of $\mathrm{Gal}(\overline{\Q}/F)$ on these special $L$-values in Theorem \ref{theorem 8.2}. The proof of this theorem uses the standard strategy in \cite{BS} and \cite{Sh00}. That is, we derive the algebraicity from the integral representation \eqref{6.16}, \eqref{6.17} reformulated from Theorem \ref{theorem 1.1}, \ref{theorem 2.2}, \ref{theorem 6.1} and the algebraic properties of the Fourier coefficients of Eisenstein series in Corollary \ref{corollary 7.12}. 

This kind of result is also obtained in \cite{BS, Sh95, Sh00} for symplectic and unitary groups and in \cite{JYB1} for quaternionic unitary groups. We explain what is new in this work. First of all, all these works, except \cite{Sh95} for symplectic groups, only consider partial $L$-functions while the $L$-function considered here includes those ramified $L$-factors. Of course, there are only finitely many missing $L$-factors in the partial $L$-function and if these ramified $L$-factors are known to be algebraic one may manually add these factors to the algebraicity result of partial $L$-functions. But this only make sense at the special values beyond the absolutely convergence bound, i.e. $s=s_0>\kappa$. Secondly, inspired by \cite{BS}, our local sections of the Eisenstein series are chosen such that we only need information about the Fourier coefficients of rank greater or equal to $2m$ (where $m$ is the Witt index of $G$). This allows us to get a better bound on $l$ and discuss the special values below the absolutely convergence bound.

The orthogonal groups are not studied here for two reasons. Firstly, when the symmetric space of $G$ is hermitian (i.e. $G(F_v)$ has Witt index $2$ for any archimedean places $v$), the symmetric space of $H$ is no longer hermitian so that one need to carefully define the meaning of algebraic modular forms on $H(\mathbb{A})$. Secondly, the archimedean computations for the Fourier coefficients of Eisenstein series will involve certain generalized Bessel functions studied in \cite{Sh99b}. The analytic properties are studied there but there are no explicit formulas as for the confluent hypergeometric function in \cite{Sh82} so that we do not know the algebraic properties of these functions so far.

\subsection{The $p$-adic $L$-function}

We keep the setup in the previous subsection. In particular, we fix a prime ideal $\boldsymbol{p}$ of $F$ above an odd prime number $p$.

Once the algebraicity of special $L$-values is known, one can ask about the $p$-adic interpolation of these values. Keeping the setup as in last subsection, our main theorem on $p$-adic $L$-functions (Theorem \ref{theorem 8.5}, Theorem \ref{theorem 8.6}, \eqref{8.12}, \eqref{8.13}, \eqref{8.18}, \eqref{8.19}) is stated as follows.

\begin{thm}
\label{theorem 1.3}
Assume $\boldsymbol{f}$ is $\boldsymbol{p}$-ordinary in the sense that $\alpha(\boldsymbol{p})\in\mathcal{O}_{\C_p}^{\times}$. Fix $\chi_1$ be a Hecke character of conductor $\mathfrak{n}_2$ and infinity type $\boldsymbol{l}$. That is $\chi_{1,v}(x)=x^{l}|x|^{-l}$ for any $v|\infty$.
 \\
(1) For unitary and quaternionic unitary groups, there exists a $p$-adic measure $\mu(\boldsymbol{f})$ on $\mathrm{Cl}_E^+(\boldsymbol{p}^{\infty})$ such that for any finite order Hecke character $\boldsymbol{\chi}$ of conductor $\boldsymbol{p^c}$,
\begin{equation}
\begin{aligned}
\int_{\mathrm{Cl}_E^+(\boldsymbol{p}^{\infty})}\boldsymbol{\chi}d\mu(\boldsymbol{f})&=|\boldsymbol{\varpi}|^{\boldsymbol{c}\mathbf{d}_1\frac{m(m-1)}{2}}\pi^{d(F)d(\pi)}\left(c_l(s_0)\prod_{i=0}^{n-1}\Gamma(\mathbf{d}_1(l-i))\right)^{d(F)}\\
&\times G^D(\chi)^{-m}M\left(s_0+\frac{1}{2},\boldsymbol{f}\times\chi\right)\cdot\frac{L\left(s_0+\frac{1}{2},\boldsymbol{f}\times\chi\right)}{\Omega\cdot\langle\boldsymbol{f},\boldsymbol{f}\rangle}.
\end{aligned}
\end{equation}
(2) For symplectic groups we assume the Witt index $m$ has the same parity with the weight $l$ of $\boldsymbol{f}$, i.e. $l\equiv m\text{ mod }2$. For quaternionic orthogonal groups, we assume $\boldsymbol{p}$ splits in the quaternion algebra $D$ if the group is not totally isotropic (i.e. $r>0$ in \eqref{2.1.3}). Then in these two cases, there exists a $p$-adic measure $\mu(\boldsymbol{f})$ on $\mathrm{Cl}_F^{+}(\boldsymbol{p}^{\infty})$ such that for any finite order Hecke character $\boldsymbol{\chi}$ of conductor $\boldsymbol{p^c}$,
\begin{equation}
\begin{aligned}
&\int_{\mathrm{Cl}_F^+(\boldsymbol{p}^{\infty})}\boldsymbol{\chi}d\mu(\boldsymbol{f})\\
=&|\boldsymbol{\varpi}|^{\boldsymbol{c}\mathbf{d}_1\frac{m(m-1)}{2}}N_{F/\Q}(\boldsymbol{p})^{\boldsymbol{c}\left(s_0-\frac{1}{2}\right)}G^D(\chi)^{-m}G^F(\chi)^{-1}\pi^{d(F)d(\pi)}\\
\times&\left(c_l(s_0)\Gamma\left(s_0+\frac{1}{2}\right)\prod_{i=0}^{n\mathbf{d}_1-1}\Gamma\left(l-\frac{i}{2}\right)\right)^{d(F)}\cdot\frac{L_{\boldsymbol{p}}\left(s_0+\frac{1}{2},\chi\right)}{L_{\boldsymbol{p}}\left(\frac{1}{2}-s_0,\chi^{-1}\right)}\\
\times&M\left(s_0+\frac{1}{2},\boldsymbol{f}\times\chi\right)\cdot\frac{L\left(s_0+\frac{1}{2},\boldsymbol{f}\times\chi\right)}{\langle\boldsymbol{f},\boldsymbol{f}\rangle}.
\end{aligned}
\end{equation}

Here we are again denoting $\chi=\boldsymbol{\chi}\chi_1$ when $\boldsymbol{\chi}$ varying and:\\
(a) $E/F$ is an imaginary quadratic extension defining the unitary group and in other cases $E=F$,\\
(b) $\Omega=1$ for quaternionic unitary groups and $\Omega$ is the CM period \eqref{cmperiod} for unitary groups,\\
(c) $\mathrm{Cl}_E^+(\boldsymbol{p}^{\infty})$ is the $p$-adic analytic group defined in \eqref{8.2.2},\\
(d) $d(F)=[F:\Q]$, $\mathbf{d}_1=2$ for two quaternionic cases and $\mathbf{d}_1=1$ for symplectic and unitary groups,\\
(e) $G^D(\chi)$ is the Gauss sum of $\chi$ defined on $D$,\\
(f) $c_l(s_0)$ is given by Proposition \ref{proposition 5.4} and $d(\pi)$ is given in \eqref{7.4.8},\\
(g) $M(s_0+\frac{1}{2},\boldsymbol{f}\times\chi)$ is given in Proposition \ref{proposition 4.6} if $\boldsymbol{c}=0$ and is understood as $1$ if $\boldsymbol{c}>0$,\\
(h) $L_{\boldsymbol{p}}\left(s_0+\frac{1}{2},\chi\right)^{-1}=1-\chi(\boldsymbol{\varpi})|\boldsymbol{\varpi}|^{s_0+\frac{1}{2}}$ if $\boldsymbol{c}=0$ and is understood as $1$ if $\boldsymbol{c}>0$.
\end{thm}

The assumption for symplectic groups is only for simplicity while the assumption for quaternionic orthogonal groups is technical and is necessary in our proof of Theorem \ref{theorem 8.6}. Again, the construction of $p$-adic $L$-functions relies on properly choosing the local sections of the Eisenstein series such that its Fourier coefficients have $p$-adic interpolations. We refer the reader to \cite{LZ20} where it is carefully explained how these local sections should be chosen.

We compare our result with other works in the following.

For symplectic groups, $p$-adic $L$-functions have been constructed in \cite{BS} using the doubling method and in \cite{CP96} using the Rankin-Selberg method. The $p$-adic $L$-functions for ordinary families are constructed in \cite{LZ20}. We admit that the approach in our work is highly inspired by \cite{BS} and \cite{LZ20}. Although all these works are concerning the base field $F=\Q$, there is no difficulty to generalize their work to any totally real field $F$ as we have done here.

For unitary groups, $p$-adic $L$-functions are studied in \cite{E21, EHLS, HLS, SU, WX15b} for ordinary families. All these works assume that $\boldsymbol{p}$ is split in the imaginary quadratic extension $E/F$ so that the local group $G(F_{\boldsymbol{p}})$ at $\boldsymbol{p}$ is a general linear group. Their local sections are always chosen as certain Godement-Jacquet sections, and we do not discuss this case in our work. When $\boldsymbol{p}$ is inert, the $p$-adic $L$-function is constructed in \cite{B16} for totally isotropic groups (i.e. $r=0$ in \eqref{2.1.3}). Our result for the general unitary group with $\boldsymbol{p}$ nonsplit is new.

The $L$-functions for the two quaternionic cases are less studied than symplectic and unitary groups. In our previous work \cite{JYB2}, we have constructed $p$-adic $L$-functions for these totally isotropic groups and restricted to the case when $\boldsymbol{p}$ splits in the quaternion algebra. This case is much simpler as the local group $G(F_{\boldsymbol{p}})$ will be either an orthogonal group or symplectic group and both are totally isotropic. We have removed these restrictions in this paper.

We admit that the construction of $p$-adic $L$-functions for ordinary families is beyond the scope of this work. For $p$-adic families, one also needs to understand more about the geometry of Shimura varieties and $p$-adic modular forms. For example, in \cite{E21, EHLS, HLS} the split assumption on $\boldsymbol{p}$ is also used to guarantee the nonvanishing of a certain ordinary locus in defining the $p$-adic modular forms (see also \cite[5.3(2)]{E21}). In the two quaternionic cases, the geometry of Shimura varieties becomes more challenging as these Shimura varieties are not of PEL type.

We end up the introduction by giving an overview of the rest of the paper. In Section \ref{section 2}, we review the global integral in the doubling method and state our main theorems on integral representation of standard $L$-functions. The definition of local $L$-factors and computations of non-archimedean local integrals are carried out in Section \ref{section 3} and Section \ref{section 4} respectively. The rest of the paper will devoted to study properties of special $L$-values and we restrict to certain classical groups starting in Section \ref{section 5} in which we also review the definition of algebraic modular forms and calculate the archimedean local integrals. The integral representation is then reformulated in Section \ref{section 6} and we calculate the Fourier expansion of Eisenstein series in Section \ref{section 7}. With all the results obtained in these two sections we apply our computations to prove the algebraicity of special $L$-values and construct the $p$-adic $L$-functions in the final Section \ref{section 8}.

\section{The Integral Representation of the Standard $L$-function}
\label{section 2}

We present a global integral representation of the standard $L$-function for classical groups in this section. The definition of the local $L$-factors and computation of the local integrals will be carried out in next two sections.

\subsection{The doubling embedding}
\label{section 2.1}

We start by fixing some general notations. For an associative ring $R$ with identity, denote by $\mathrm{Mat}_{m,n}(R)$ the $R$-module of all $m\times n$ matrices with entries in $R$. Set $\mathrm{Mat}_n(R)=\mathrm{Mat}_{n,n}(R)$ and $\mathrm{GL}_n(R)=\mathrm{Mat}_n(R)^{\times}$. For $x\in\mathrm{Mat}_{m,n}(R)$, denote $\transpose{x}$ for its transpose. Denote by $1_n$ and $0_n$, or even simply $1$ and $0$ if their sizes are clear from the context, for the identity matrix and zero matrix in $\mathrm{Mat}_n(R)$, respectively. We follow the setup for classical groups as in \cite{Y14}.

Let $F$ be a local or global field and $D$ an $F$-algebra with involution $\rho$ whose center $E$ contains $F$. The couple $(D,\rho)$ considered in this paper will belong to the following five types:\\
(a) $D=E=F$ and $\rho$ is the identity map,\\
(b) $D$ is a division quaternion algebra over $E=F$ and $\rho$ is the main involution of $D$,\\
(c) $D$ is a division algebra central over a quadratic extension $E$ of $F$ and $\rho$ generates $\mathrm{Gal}(E/F)$,\\
(d) $D=\mathrm{Mat}_2(E),E=F$ and $\rho$ is given by $\left[\begin{array}{cc}
a & b\\
c & d
\end{array}\right]^{\rho}=\left[\begin{array}{cc}
d & -b\\
-c & a
\end{array}\right]$,\\
(e) $D=\mathbf{D}\oplus\mathbf{D}^{\mathrm{op}},E=F\oplus F$ and $\rho$ is given by $(x,y)^{\rho}=(y,x)$, where $\mathbf{D}$ is a division algebra central over $F$ and $\mathbf{D}^{\mathrm{op}}$ is its opposite algebra.

For $x=(x_{ij})\in\mathrm{Mat}_{mn}(D)$, set $x^{\rho}=(x_{ij}^{\rho})$ and $x^{\ast}=\transpose{x}^{\rho},\hat{x}=(x^{\ast})^{-1}$. For $x\in\mathrm{Mat}_n(D)$, $\nu(x)\in E,\tau(x)\in E$ stand for its reduced norm and reduced trace to the center $E$.

Fix a triple $(D,\rho,\epsilon)$ with $\epsilon=\pm 1$. Let $W$ be a free left $D$-module of rank $n$. By an $\epsilon$-hermitian space we mean a structure $\mathcal{W}=(W,\langle\cdot,\cdot\rangle)$ where $\langle\cdot,\cdot\rangle$ is an $\epsilon$-hermitian form on $W$, that is, an $F$-bilinear map $\langle\cdot,\cdot\rangle:W\times W\to D$ such that
\begin{equation}
\label{2.1.1}
\langle x,y\rangle^{\rho}=\epsilon\langle y,x\rangle,\quad\langle ax,by\rangle=a\langle x,y\rangle b^{\rho}, \quad(a,b\in D;x,y\in W).
\end{equation}
We always assume such a form to be non-degenerate, i.e. $\langle x,W\rangle=0$ implies $x=0$. Denote the ring of all $D$-linear endomorphisms of $W$ by $\mathrm{End}_D(W)$ and $\mathrm{GL}_D(W)=\mathrm{End}_D(W)^{\times}$. If we view elements of $W$ as row vectors, then $\mathrm{GL}_D(W)$ acts on $W$ from the right. The classical group of $\mathcal{W}$ is defined as 
\begin{equation}
\label{2.1.2}
G:=G(\mathcal{W}):=\{g\in\mathrm{GL}_D(W):\langle xg,yg\rangle=\langle x,y\rangle\text{ for all }x,y\in W\},
\end{equation}
which is a (possibly disconnected) reductive algebraic group over $F$. By fixing a basis of $W$, we can identify $\mathrm{End}_D(W)$ with $\mathrm{Mat}_n(D)$ and $\mathrm{GL}_D(W)$ with $\mathrm{GL}_n(D)$. Then $\langle\cdot,\cdot\rangle$ can be expressed as a matrix of the form
\begin{equation}
\label{2.1.3}
\Phi=\left[\begin{array}{ccc}
0 & 0 & 1_m\\
0 & \theta & 0\\
\epsilon\cdot1_m & 0 & 0
\end{array}\right]\text{ with }n=2m+r,\theta^{\ast}=\epsilon\theta\in\mathrm{GL}_{r}(D)
\end{equation}
and thus the classical group $G$ can be realized as
\begin{equation}
\label{2.1.4}
G:=G(W,\Phi)=\{g\in\mathrm{GL}_n(D):g\Phi g^{\ast}=\Phi\}.
\end{equation}
We assume $m\geq 1$ throughout the paper to avoid the discussion of definite classical groups. Doubling the underlying $\epsilon$-hermitian space we consider $\mathcal{V}=(W\oplus W,\langle\langle\cdot,\cdot\rangle\rangle)$ where
\begin{equation}
\label{2.1.5}
\langle\langle(x_1,x_2),(y_1,y_2)\rangle\rangle:=\langle x_1,y_1\rangle-\langle x_2,y_2\rangle\text{ for }(x_1,x_2),(y_1,y_2)\in W\oplus W.
\end{equation}
By fixing a basis of $\mathcal{V}$, the classical group $G(\mathcal{V})$ is isomorphic to
\begin{equation}
\label{2.1.6}
H=\{h\in\mathrm{GL}_{2n}(D):gJ_ng^{\ast}=J_n\},\quad J_n=\left[\begin{array}{cc}
0 & 1_n\\
\epsilon\cdot 1_n & 0
\end{array}\right].
\end{equation}
Note that
\[
R\left[\begin{array}{cc}
\Phi & 0\\
0 & -\Phi
\end{array}\right]R^{\ast}=J_n,
\]
with
\[
R=\left[\begin{array}{cccccc}
0 & \frac{\epsilon}{2}\cdot 1_r & 0 & 0 & \frac{\epsilon}{2}\cdot 1_r & 0\\
0 & 0 & 0 & 0 & 0 & -\epsilon\cdot 1_m\\
1_m & 0 & 0 & 0 & 0 & 0\\
0 & \theta^{-1} & 0 & 0 & -\theta^{-1} & 0\\
0 & 0 & 0 & 1_m & 0 & 0\\
0 & 0 & 1_m & 0 & 0 & 0
\end{array}\right].
\]
Then we define a doubling map
\begin{equation}
\label{2.1.7}
\begin{aligned}
G\times G&\to H\\
(g_1,g_2)&\mapsto R\left[\begin{array}{cc}
g_1 & 0\\
0 & g_2
\end{array}\right]R^{-1}.
\end{aligned}
\end{equation}
We thus view $G\times G$ as a subgroup of $H$ and identify $(g_1,g_2)$ with its image in $H$. More explicitly, if we write
\begin{equation}
\label{2.1.8}
g_1=\left[\begin{array}{ccc}
a_1 & f_1 & b_1\\
h_1 & e_1 & j_1\\
c_1 & k_1 & d_1
\end{array}\right],\qquad g_2=\left[\begin{array}{ccc}
a_2 & f_2 & b_2\\
h_2 & e_2 & j_2\\
c_2 & k_2 & d_2
\end{array}\right],
\end{equation}
with $a_1,a_2,d_1,d_2$ of size $m\times m$, $e_1,e_2$ of size $r\times r$, then
\begin{equation}
\label{2.1.9}
(g_1,g_2)=\left[\begin{array}{cccccc}
\frac{e_1+e_2}{2} & -\frac{j_2}{2} & \frac{\epsilon h_1}{2} & \frac{\epsilon(e_1-e_2)\theta}{4} & \frac{\epsilon h_2}{2} & \frac{\epsilon j_1}{2}\\
-k_2 & d_2 & 0 & \frac{\epsilon k_2\theta}{2} & -\epsilon c_2 & 0\\
\epsilon f_1 & 0 & a_1 & \frac{f_1\theta}{2} & 0 & b_1\\
\epsilon\theta^{-1}(e_1-e_2) & \epsilon\theta^{-1}j_2 & \theta^{-1}h_1 & \theta^{-1}\frac{e_1+e_2}{2}\theta & -\theta^{-1}h_2 & \theta^{-1}j_1\\
\epsilon f_2 & -\epsilon b_2 & 0 & -\frac{f_2\theta}{2} & a_2 & 0\\
\epsilon k_1 & 0 & c_1 & \frac{k_1\theta}{2} & 0 & d_1
\end{array}\right].
\end{equation}

\subsection{The global zeta integral}
\label{section 2.2}

Let $F$ be a number field with adele ring $\mathbb{A}$. We consider tuples $(D,\rho,\epsilon)$ of following five cases:

\begin{tabular}{ll}
(Case I, Orthogonal) & $(D,\rho)$ of type (a) with $\epsilon=1$,\\
(Case II, Symplectic) & $(D,\rho)$ of type (a) with $\epsilon=-1$,\\
(Case III, Quaternionic Orthogonal) & $(D,\rho)$ of type (b) with $\epsilon=1$,\\
(Case IV, Quaternionic Unitary) & $(D,\rho)$ of type (b) with $\epsilon=-1$,\\
(Case V, Unitary) & $(D,\rho)$ of type (c) with $D=E$ and $\epsilon=-1$.
\end{tabular}

\begin{rem}
\label{remark 2.1} 
In this paper, we label our groups as Case I-V for simplicity but we may also call the name of the groups (i.e. orthogonal, symplectic, ...) so that one can easily compare the groups here with the one in other papers. The notion of orthogonal, symplectic and unitary groups are well known and appear frequently in the literature while the groups of Case III, IV do not have a standard name. Here we call them quaternionic orthogonal or unitary depending on whether the form defining the group is hermitian (like the orthogonal group) or skew-hermitian (like the unitary).  But indeed, both groups have been called `quaternionic unitary' in the literature so the reader should be careful about their meaning. For example, the groups studied in \cite{G77} and \cite{Sh99} are the quaternionic orthogonal group here. 
\end{rem}

Denote $\mathbb{A}_{D}=D\otimes_{F}\mathbb{A}$ for the adelization and $D_v=D\otimes_FF_v$ for the localization at a place $v$ of $F$. Fix an $\epsilon$-hermitian matrix $\Phi\in\mathrm{GL}_n(D)$ of the form in \eqref{2.1.3} with $\theta$ an anisotropic matrix (so the Witt index of $\Phi$ is $m$ and $n=2m+r$) and define group $G,H$ as in \eqref{2.1.4},\eqref{2.1.6} with a doubling embedding $G\times G\to H$ given by \eqref{2.1.7}. For global groups, we will write $G(\mathbb{A}),H(\mathbb{A}),G(F_v),H(F_v)$ for its adelization and localization but simply write $G=G(F),H=H(F)$ for the rational points if its meaning is clear from the context. 

\begin{rem}
In this paper, we shall use the term `totally isotropic group' to indicate the group whose associated $\epsilon$-hermitian form is totally isotropic (i.e. $r=0$ in \eqref{2.1.3}) instead of using the term split group or quasi-split group. In the sequel, when discussing the local groups $G(F_v)$, we will distinguish the `split' and `nonsplit' case according to whether $v$ is split in $D$ or not.
\end{rem}

Let $P\subset H$ be the Siegel parabolic subgroup whose Levi component is $\mathrm{GL}_n(D)$. More explicitly, $P=M\ltimes N$ with
\begin{equation}
\label{2.2.1}
M=\left\{\left[\begin{array}{cc}
a & 0\\
0 & \hat{a}
\end{array}\right]:a\in\mathrm{GL}_n(D)\right\},\quad N=\left\{\left[\begin{array}{cc}
1 & b\\
0 & 1
\end{array}\right]:b\in S_n(F)\right\}.
\end{equation}
Here $S_n(F)$ is an additive algebraic group with
\begin{equation}
\label{2.2.2}
S_n(F)=\left\{b\in\mathrm{Mat}_n(D):\epsilon b+b^{\ast}=0\right\}.
\end{equation}
Let $\chi:E^{\times}\backslash\mathbb{A}_E^{\times}\to\C^{\times}$ be a Hecke character and extend it to a character on $\mathrm{GL}_n(\mathbb{A}_D)$ (still denoted by $\chi$) by taking the composite with the reduced norm $\nu:\mathrm{GL}_n(\mathbb{A}_D)\to \mathbb{A}_E^{\times}$. Consider the induced representation
\begin{equation}
\label{2.2.3}
\mathrm{Ind}_{P(\mathbb{A})}^{H(\mathbb{A})}(\chi|\nu(\cdot)|^{s})
\end{equation}
consisting functions $f_s:H(\mathbb{A})\to\C$ such that
\begin{equation}
\label{2.2.4}
f_s(pg)=\chi(\nu(a))|\mathrm{N}_{E/F}(\nu(a))|^{s+\kappa}f_s(g),\text{ for }p=\left[\begin{array}{cc}
a & b\\
0 & \hat{a}
\end{array}\right]\in P(\mathbb{A}),
\end{equation}
where 
\begin{equation}
\label{2.2.5}
\kappa=\left\{\begin{array}{cc}
\frac{n-1}{2} & \text{ Case I, }\\
\frac{n+1}{2} & \text{ Case II, }\\
\frac{2n+1}{2} & \text{ Case III, }\\
\frac{2n-1}{2} & \text{ Case IV, }\\
\frac{n}{2} & \text{ Case V.}
\end{array}\right.
\end{equation}
We then form the Eisenstein series 
\begin{equation}
\label{2.2.6}
E(h;f_s)=\sum_{\gamma\in P(F)\backslash H(F)}f_s(\gamma h),\qquad h\in H(\mathbb{A})
\end{equation}
on $H(\mathbb{A})$ associated to a standard section $f_s$. 

Let $\pi$ be a cuspidal automorphic representation of $G(\mathbb{A})$ with trivial central character and $\phi_1,\phi_2\in\pi$ be two cusp forms. The global integral we consider is
\begin{equation}
\label{2.2.7}
\begin{aligned}
&\mathcal{Z}(s;\phi_1,\phi_2,f_s)\\
=&\int_{G(F)\times G(F)\backslash G(\mathbb{A})\times G(\mathbb{A})}E((g_1,g_2);f_s)\overline{\phi_1(g_1)}\phi_2(g_2)\chi(\nu(g_2))^{-1}dg_1dg_2.
\end{aligned}
\end{equation}
Unfolding the Eisenstein series, we get
\begin{equation}
\label{2.2.8}
\mathcal{Z}(s;\phi_1,\phi_2,f_s)=\int_{G(\mathbb{A})}f_s(\delta(g,1))\langle\pi(g)\phi_1,\phi_2\rangle dg,
\end{equation}
where
\begin{equation}
\label{2.2.9}
\langle\phi_1,\phi_2\rangle=\int_{G(F)\backslash G(\mathbb{A})}\phi_1(g)\overline{\phi_2(g)}dg,
\end{equation}
is the standard inner product on $G(\mathbb{A})$ and
\begin{equation}
\label{2.2.10}
\delta=\left[\begin{array}{cccccc}
1_r & 0 & 0 & 0 & 0 & 0\\
0 & 1_m & 0 & 0 & 0 & 0\\
0 & 0 & 1_{m} & 0 & 0 & 0\\
0 & 0 & 0 & 1_{r} & 0 & 0\\
0 & 0 & -1_m & 0 & 1_m & 0\\
0 & \epsilon\cdot 1_m & 0 & 0 & 0 & 1_m
\end{array}\right].
\end{equation}
Here $\delta$ is chosen such that $\delta(g,g)\delta^{-1}\in P$. Also note that we must take $\phi_1,\phi_2\in\pi$ in a same representation space otherwise the integral will be identically zero. 

Write $\pi=\otimes_v'\pi_{v}$ and assume $\phi_1=\otimes_v\phi_{1,v},\phi_2=\otimes_v\phi_{2,v}$ with $\phi_{1,v},\phi_{2,v}\in\pi_v$. Also choose the section $f_s$ such that $f_s=\prod_vf_{s,v}$ is factorizable with local sections $f_{s,v}\in\mathrm{Ind}_{P(F_v)}^{H(F_v)}(\chi|\nu(\cdot)|^s)$. Due to the uniqueness of the pairing, $\langle\cdot,\cdot\rangle$ is factorizable in the sense that $\langle\phi_1,\phi_2\rangle=\prod_v\langle\phi_{1,v},\phi_{2,v}\rangle$, where
\begin{equation}
\langle\phi_{1,v},\phi_{2,v}\rangle=\int_{G(F_v)}\overline{\phi_{1,v}(g)}\phi_{2,v}(g)dg
\end{equation}
is the local pairing. Then $\mathcal{Z}(s;\phi_1,\phi_2,f_s)$ has an Euler product expression
\begin{equation}
\label{2.2.11}
\begin{aligned}
\mathcal{Z}(s;\phi_1,\phi_2,f_s)&=\prod_v\mathcal{Z}_v(s;\phi_{1,v},\phi_{2,v},f_{s,v}),\\
\mathcal{Z}_v(s;\phi_{1,v},\phi_{2,v},f_{s,v})&=\int_{G(F_v)}f_{s,v}(\delta(g,1))\langle\pi(g)\phi_{1,v},\phi_{2,v}\rangle dg.
\end{aligned}
\end{equation}
Hence, the global integral $\mathcal{Z}(s;\phi_1,\phi_2,f_s)$ can be studied locally place by place.

In some works of the doubling method (e.g. \cite{BS,G84b,Sh97,Sh00}), the integral of the following form is considered
\begin{equation}
\label{global2}
\mathcal{Z}'(g_2;\phi_1,f_s)=\int_{G(F)\backslash G(\mathbb{A})}E((g_1,g_2);f_s)\overline{\phi_1(g_1)}dg_1.
\end{equation}
The computation of \eqref{global2} is same as the one for  \eqref{2.2.7}. In particular, we have
\begin{equation}
\begin{aligned}
\mathcal{Z}'(g_2;\phi_1,f_s)&=\chi(\nu(g_2))\int_{G(\mathbb{A})}f_s(\delta(g_1,1))\phi_1(g_2g_1)dg_1\\
&=\chi(\nu(g_2))\prod_v\mathcal{Z}'_v(s;\phi_{1,v},f_{s,v}),\\
\mathcal{Z}'_v(g_2;\phi_{1,v},f_{s,v})&=\int_{G(F_v)}f_{s,v}(\delta(g_1,1))\phi_{1,v}(g_2g_1)dg_1.
\end{aligned}
\end{equation}

\subsection{Main result on integral representations}
\label{section 2.3}

The first main result of this paper is an integral representation of standard $L$-functions. That is, we make the choice of $f_s$ such that the global integral $\mathcal{Z}$ in \eqref{2.2.7} represents the $L$-function defined in Section \ref{section 3}. We summarize our result here.

Let $\mathfrak{o}$ be the ring of integers of $F$ and $\mathcal{O}$ a maximal order of $D$. Denote $\mathfrak{o}_v,\mathcal{O}_v$ for their localizations and assume $D_v=\mathcal{O}_v\otimes_{\mathfrak{o}_v}F_v$. For a finite place $v$ corresponds to a prime ideal $\mathfrak{p}_v$ of $\mathfrak{o}$, denote $\varpi_v$ for the uniformizer of $\mathfrak{p}_v$ and set $q_v=|\varpi_v|_v^{-1}$. Fix an $\mathfrak{o}$-ideal $\mathfrak{n}=\mathfrak{n}_1\mathfrak{n}_2$ with $\mathfrak{n}_1,\mathfrak{n}_2$ coprime and write $\mathfrak{n}=\prod_v\mathfrak{p}_v^{\mathfrak{c}_v}$. Define the following open compact subgroup of $G(\mathfrak{o})$:
\begin{equation}
\label{2.3.1}
K(\mathfrak{n})=G(\mathfrak{o})\cap\left[\begin{array}{ccc}
\mathrm{Mat}_{m}(\mathcal{O}) & \mathrm{Mat}_{m,r}(\mathcal{O}) & \mathrm{Mat}_{m}(\mathcal{O})\\
\mathrm{Mat}_{r,m}(\mathfrak{n}\mathcal{O}) &1+\mathrm{Mat}_{r}(\mathfrak{n}'\mathcal{O}) & \mathrm{Mat}_{r,m}(\mathcal{O})\\
\mathrm{Mat}_{m}(\mathfrak{n}\mathcal{O}) & \mathrm{Mat}_{m,r}(\mathfrak{n}\mathcal{O}) & \mathrm{Mat}_{m}(\mathcal{O})
\end{array}\right],
\end{equation}
where $\mathfrak{n}'=\prod_{v,\mathfrak{c}_v\geq 1}\mathfrak{p}_v$ is the support of $\mathfrak{n}$.

Let $\phi\in\pi$ be a cusp form. Set $S_{\infty}$ be the set of all archimedean places of $F$. Denote $S_1$ be the set consisting of places dividing $\mathfrak{n}_1$ and $S_2$ the set consisting of places dividing $\mathfrak{n}_2$. We make the following assumptions:\\
(1) $2\in\mathcal{O}_v^{\times}$ and $\theta\in\mathrm{GL}_r(\mathcal{O}_v)$ for all $v\in S_1\cup S_2$,\\
(2) $\phi$ is fixed by $K(\mathfrak{n})$ and is unramified outside $S_1,S_2$, i.e. fixed by $G(\mathfrak{o}_v)$ for $v\notin S_1\cup S_2\cup S_{\infty}$,\\
(3) $\phi$ is an eigenfunction for the Hecke algebra $\mathcal{H}(K(\mathfrak{n}),\mathfrak{X})$ as in Section \ref{section 3.4},\\
(4) $\chi$ has conductor $\mathfrak{n}_2$,\\
(5) In Case V, all places $v\in S_1\cup S_2$ are nonsplit in $\mathcal{O}$.

The standard $L$-function $L(s,\phi\times\chi)$ of $\phi$ twisted by $\chi$ is defined in Section \ref{section 3.4}. There is an Euler product expression
\begin{equation}
\label{2.3.2}
L(s,\phi\times\chi)=\prod_vL_v(s,\phi_v\times\chi_v).
\end{equation}
When $v\notin S_1\cup S_2$, $L_v(s,\phi_v\times\chi_v)$ is the unramified local $L$-factors defined with $m$ in Section \ref{section 3.1} replaced by the Witt index of $G(F_v)$. When $v\in S_1\cup S_2$, $L_v(s,\phi_v\times\chi_v)$ are the ramified local $L$-factors, and in particular $L_v(s,\phi_v\times\chi_v)=1$ if $v\in S_2$. The integral representation for the partial $L$-function
\begin{equation}
\label{2.3.3}
L^{S_1\cup S_2}(s,\phi\times\chi):=\prod_{v\notin S_1\cup S_2}L_v(s,\phi\times\chi)
\end{equation}
is well known. We make the choice of local sections $f_{s,v}$ properly for $v\in S_1\cup S_2$ such that the global integral $\mathcal{Z}$ represent the complete $L$-function. Let 
\begin{equation}
\label{2.3.4}
w=\left[\begin{array}{ccc}
0 & 0 & 1_m\\
0 & 1_r & 0\\
\epsilon\cdot 1_m & 0 & 0
\end{array}\right]
\end{equation}
be a Weyl element. Define $\eta_1\in G(\mathbb{A})$ to be an element such that $(\eta_1)_v=w$ for $v\in S_1$ and $(\eta_1)_v=1$ for $v\notin S_1$. Similarly set $\eta_2\in G(\mathbb{A})$ to be an element such that $(\eta_2)_v=w$ for $v\in S_2$ and $(\eta_2)_v=1$ for $v\notin S_2$. 

Take $\phi_1=\pi(\eta_1)\phi$, $\phi_2=\pi(\eta_2)\phi$ and write
\begin{equation}
\label{2.3.5}
\mathcal{Z}(s;\phi,f_s):=\mathcal{Z}(s;\phi_1,\phi_2,f_2).
\end{equation}

\begin{thm}
\label{theorem 2.2}
Keep the assumptions of $\phi,\chi$ as above. Take the section $f_s$ to be
\begin{equation}
\label{2.3.6}
f_s=\prod_{v\notin S_1\cup S_2\cup S_{\infty}}f_{s,v}^0\cdot\prod_{v\in S_1}f_{s,v}^{\dagger,\mathfrak{c}_v}\cdot\prod_{v\in S_2}f_{s,v}^{\ddagger,\mathfrak{c}_v}\cdot\prod_{v\in S_{\infty}}f_{s,v}^{\infty}.
\end{equation}
Then
\begin{equation}
\label{2.3.7}
\mathcal{Z}(s;\phi,f_s)=C\cdot L\left(s+\frac{1}{2},\phi\times\chi\right)\cdot\mathcal{Z}_{\infty}(s;\phi_{\infty},f_s^{\infty})\cdot\prod_{v\nmid\infty}\langle\pi(\eta)\phi_v|U'(\mathfrak{n_1}),\phi_v\rangle.
\end{equation}
Here:\\
(a) $f_{s,v}^0,f_{s,v}^{\dagger,\mathfrak{c}_v},f_{s,v}^{\ddagger,\mathfrak{c}_v}$ are local sections defined  by \eqref{4.2.1}, \eqref{4.3.2}, \eqref{4.4.1} and $f_{s,v}^{\infty}$ can be chosen such that
\begin{equation}
\label{2.3.8}
\mathcal{Z}_{\infty}(s;\phi_{\infty},f_s^{\infty}):=\prod_v\mathcal{Z}_{v|\infty}(s;\phi_v,f^{\infty}_{s,v})\neq 0,
\end{equation}
(b) $U'(\mathfrak{n}_1)=\prod_{v|\mathfrak{n_1}}U'(\mathfrak{p}_v^{\mathfrak{c}_v})$ is the Hecke operator defined by \eqref{3.2.9} and 
\begin{equation}
\label{2.3.9}
\eta=\prod_{v\in S_2}\left[\begin{array}{ccc}
0 & 0 & \varpi_v^{-\mathfrak{c}_v}\cdot 1_m\\
0 & 1_r & 0\\
\varpi_v^{\mathfrak{c}_v}\cdot 1_m & 0 & 0
\end{array}\right],
\end{equation}
(c) $C$ is a constant given by
\begin{equation}
\label{2.3.10}
C=\chi(\mathfrak{n}_1)^{m\mathbf{d}_1}|\mathfrak{n}_1|^{m\mathbf{d}_2(s+\kappa)}\mathrm{vol}(\mathrm{GL}_m(\mathcal{O})/\mathrm{GL}_m(\mathfrak{n}_2\mathcal{O})),
\end{equation}
with
\begin{equation}
\label{2.3.11}
\mathbf{d}_1=\left\{\begin{array}{cc}
1 & \text{ Case I, II, V},\\
2 & \text{ Case III, IV},
\end{array}\right.\qquad\mathbf{d}_2=\left\{\begin{array}{cc}
1 & \text{ Case I, II},\\
2 & \text{ Case III, IV, V}.
\end{array}\right.
\end{equation}
\end{thm}

\begin{rem}
\label{remark 2.3}
This is proved by combining the local computations of Proposition \ref{proposition 4.1}, \ref{proposition 4.2}, \ref{proposition 4.4}. In Case III, IV, if $D_v$ splits then the local computations follow from the one for Case I, II as the local group $G(F_v)$ is a symplectic group in Case III or an orthogonal group in Case IV (see also Section \ref{section 3.3}). For Case V, we do not cover the split case in this paper for simplicity and also because the split is well studied in \cite{HLS} and \cite{EHLS}. Hence throughout the paper we will assume all $v|\mathfrak{n}$ are nonsplit in $\mathcal{O}$ for Case V.
\end{rem}

\section{Hecke Operators and local $L$-factors}

\label{section 3}
In this and the next section, we fix the following local setup. Let $F$ be a non-archimedean local field and $\mathfrak{o}$ its ring of integers with the maximal ideal $\mathfrak{p}$. Fix a uniformizer $\varpi$ and the absolute value $|\cdot|$ on $F$ normalized so that $|\varpi|=q^{-1}$ with $q$ the cardinality of the residue field. We consider tuples $(D,\rho,\epsilon)$ of following eight cases:

\begin{tabular}{ll}
(Case I, Orthogonal) & $(D,\rho)$ of type (a) with $\epsilon=1$,\\
(Case II, Symplectic)  &$(D,\rho)$ of type (a) with $\epsilon=-1$,\\
(Case III, Quaternionic Orthogonal Nonsplit) &$(D,\rho)$ of type (b) with $\epsilon=1$,\\
(Case III', Quaternionic Orthogonal Split) &$(D,\rho)$ of type (d) with $\epsilon=1$,\\
(Case IV, Quaternionic Unitary Nonsplit) &$(D,\rho)$ of type (b) with $\epsilon=-1$,\\
(Case IV', Quaternionic Unitary Split) &$(D,\rho)$ of type (d) with $\epsilon=-1$,\\
(Case V, Unitary Nonsplit) &$(D,\rho)$ of type (c) with $D=E$, $\epsilon=-1$,\\
(Case V', Unitary Split) &$(D,\rho)$ of type (e) with $D=E$, $\epsilon=-1$.
\end{tabular}

We fix a maximal order $\mathcal{O}$ of $D$ such that $D=\mathcal{O}\otimes_{\mathfrak{o}}F$. Let $\mathfrak{q}$ be a prime in $\mathcal{O}$ above $\mathfrak{p}$ and fix $\widetilde{\varpi}$ a uniformizer of $\mathfrak{q}$.

\subsection{Unramified local $L$-factors}
\label{section 3.1}

In this and the next subsection, we do not consider three split cases (i.e Case III', IV', V'). Let
\begin{equation}
\label{3.1.1}
G:=G(F):=\{g\in\mathrm{GL}_n(D):g\Phi g^{\ast}=\Phi\},\,\Phi=\left[\begin{array}{ccc}
0 & 0 & 1_m\\
0 & \theta & 0\\
\epsilon\cdot 1_m & 0 & 0
\end{array}\right],
\end{equation}
with $n=2m+r$ and $\theta^{\ast}=\epsilon\theta\in\mathrm{GL}_r(D)$ is anisotropic. Assume $\pi$ is an unramified admissible representation of $G(F)$ and $\phi\in\pi$ a spherical vector. Also assume $\chi$ is an unramified character of $E^{\times}$.

Recall the Cartan decomposition
\begin{equation}
\label{3.1.2}
\begin{aligned}
G(F)&=\coprod_{\substack{e_1,...,e_m\in\Z\\0\leq e_1\leq ...\leq e_m}}K_{e_1,...,e_m},\\
K_{e_1,...,e_m}&=G(\mathfrak{o})\mathrm{diag}[\widetilde{\varpi}^{e_1},...,\widetilde{\varpi}^{e_m},1_r,\widetilde{\varpi}^{-e_1},...,\widetilde{\varpi}^{-e_m}]G(\mathfrak{o}).
\end{aligned}
\end{equation}

The local spherical Hecke algebra $\mathcal{H}$ is generated by all such double cosets $K_{e_1,...e_m}$. The action of the Hecke operator associated to $[K_{e_1,...,e_{m}}]$ on $\phi$ is given by
\begin{equation}
\label{3.1.3}
\phi|[K_{e_1,...,e_{m}}]=\int_{K_{e_1,...,e_{m}}}\pi(g)\phi dg.
\end{equation}
Here the measure $dg$ is normalized such that $G(\mathfrak{o})$ has volume $1$. Since the space of spherical vectors has dimension one, $\phi$ is an eigenvector under the action of Hecke operators, that is
\begin{equation}
\label{3.1.4}
\phi|[K_{e_1,...,e_{m}}]=\lambda_{e_1,...,e_{m}}(\phi)\phi,
\end{equation}
for some scalar $\lambda_{e_1,...,e_{m}}(\phi)$. We define the unramified local $L$-factors as
\begin{equation}
\label{3.1.5}
\begin{aligned}
&L\left(s+\frac{1}{2},\phi\times\chi\right)\\
=&b(s,\chi)\sum_{\substack{e_1,...,e_{m}\in\Z\\0\leq e_1\leq ...\leq e_{m}}}\lambda_{e_1,...,e_{m}}(\phi)\left(\chi(\nu(\widetilde{\varpi}))|\mathrm{N}_{E/F}(\nu(\widetilde{\varpi}))|^{s+\kappa}\right)^{e_1+...+e_{m}}.
\end{aligned}
\end{equation}
Here $b(s,\chi)$ is the normalizing factor given in the following list (taken from \cite[p.667]{Y14} but Case III, IV should be calculated from \cite[Proposition 3.5]{Sh99}). \\
(Case I, Orthogonal) 
\[
\begin{aligned}
b(s,\chi)&=\prod_{i=1}^{\lfloor\frac{n}{2}\rfloor}L(2s+n+1-2i,\chi^2),
\end{aligned}
\]
(Case II, Symplectic) 
\[
\begin{aligned}
b(s,\chi)&=L\left(s+\frac{n+1}{2},\chi\right)\prod_{i=1}^{\frac{n}{2}}L(2s-1+2i,\chi^2),
\end{aligned}
\]
(Case III, Quaternionic Orthogonal Nonsplit)
\[
b(s,\chi)=L\left(s+\frac{2n+1}{2},\chi\right)\prod_{i=1}^nL(2s+2n+1-4i,\chi^2),
\]
(Case IV, Quaternionic Unitary Nonsplit)
\[
b(s,\chi)=\prod_{i=1}^nL(2s+2n+3-4i,\chi^2),
\]
(Case V, Unitary) Set $\chi^0=\chi|_{F^{\times}}$ and let $\chi_{E/F}$ be the quadratic character associated to $E/F$, then
\[
\begin{aligned}
b(s,\chi)&=\prod_{i=1}^{n}L(2s+i,\chi^0\chi_{E/F}^{n+i}).
\end{aligned}
\]
Here $L(s,\chi)$ means the local $L$-factor of Hecke $L$-functions.

\begin{prop}
\label{proposition 3.1}
Let $\alpha_i,1\leq i\leq m$ be the Satake parameters of $\phi$. Then $L(s,\phi\times\chi)$ has an Euler product expansion with $L(s,\phi\times\chi)^{-1}$ given by the following list. \\
(Case I, Orthogonal)
\[
\prod_{i=1}^{\lfloor\frac{r}{2}\rfloor}\left(1-\chi(\varpi)^2q^{2i-r-2s}\right)\times\prod_{i=1}^{m}\left(1-\chi(\varpi)\alpha_iq^{-1+\frac{r}{2}-s}\right)\left(1-\chi(\varpi)\alpha_i^{-1}q^{1-\frac{r}{2}-s}\right),
\]
(Case II, Symplectic)
\[
\left(1-\chi(\varpi)q^{-s}\right)\times\prod_{i=1}^{m}\left(1-\chi(\varpi)\alpha_iq^{-s}\right)\left(1-\chi(\varpi)\alpha^{-1}_iq^{-s}\right),
\]
(Case III, Quaternionic Orthogonal Nonsplit)
\[
\begin{aligned}
&\left(1-\chi(\varpi)q^{-r-s}\right)\times\prod_{i=1}^{m+r}\left(1-\chi(\varpi)^2q^{4i-2r-2s}\right)\\
\times&\prod_{i=1}^m\left(1-\chi(\varpi)\alpha_iq^{-1+r-s}\right)\left(1-\chi(\varpi)\alpha_i^{-1}q^{-r-s}\right),
\end{aligned}
\]
(Case IV, Quaternionic Unitary Nonsplit)
\[
\begin{aligned}
&\prod_{i=1}^{m+r}\left(1-\chi(\varpi)^2q^{4i-2-2r-2s}\right)\\
\times&\prod_{i=1}^m\left(1-\chi(\varpi)\alpha_iq^{-2+r-s}\right)\left(1-\chi(\varpi)\alpha_i^{-1}q^{1-r-s}\right),
\end{aligned}
\]
(Case V, Unitary Inert) $E/F$ is inert,
\[
\prod_{i=1}^{m}\left(1-\chi(\varpi)\alpha_i q^{-1+r-s}\right)\left(1-\chi(\varpi)\alpha_i^{-1}q^{1-r-s}\right),
\]
(Case V, Unitary Ramified) $E/F$ is ramified, 
\[
\prod_{i=1}^{m}\left(1-\chi(\widetilde{\varpi})\alpha_i q^{\frac{r-1}{2}-s}\right)\left(1-\chi(\widetilde{\varpi})\alpha_i^{-1}q^{-\frac{r-1}{2}-s}\right),
\]
\end{prop}

\begin{proof}
The symplectic and unitary cases are given in \cite[Theorem 19.8]{Sh00}. The orthogonal case is given by \cite[Proposition 17.14]{Sh04} and the quaternionic orthogonal group is studied in \cite[Theorem 3.12]{Sh99}. All can be computed using the method in \cite[Section 16]{Sh97}. For the quaternionic unitary groups, by the same manner, we calculate the following Dirichlet series,
\[
\begin{aligned}
&\sum_{\substack{e_1,...,e_m\in\Z\\0\leq e_1\leq ...\leq e_m}}\lambda_{e_1,...,e_m}(q^{-s})^{e_1+...+e_m}=\alpha(s)\beta(2s-2m+1)\mathcal{A}(s-n+1-r,s),
\end{aligned}
\]
where
\[
\begin{aligned}
\alpha(s)&=\prod_{i=1}^m\frac{1-q^{4i-4-2s}}{1-q^{2m+2i-3-2s}},\\
\beta(s)&=\prod_{i=1}^m\frac{1-q^{2i-2-s}}{1-q^{2r+2i-2-s}},\\
\mathcal{A}(s',s)&=\prod_{i=1}^m\frac{1-q^{2i-2-s-s'}}{(1-q^{-2-s'}\alpha_i)(1-q^{2m-s}\alpha_i^{-1})}.
\end{aligned}
\]
Then
\[
\begin{aligned}
&\sum_{\substack{e_1,...,e_m\in\Z\\0\leq e_1\leq ...\leq e_m}}\lambda_{e_1,...,e_m}(q^{-s})^{e_1+...+e_m}=\prod_{i=1}^m\frac{1-q^{4i-4-2s}}{(1-q^{n+r-3-s}\alpha_i)(1-q^{2m-s}\alpha_i^{-1})}.
\end{aligned}
\]
Multiplying the normalizing factor $b(s,\chi)$ we obtain the result in the above list.
\end{proof}

\subsection{Ramified local $L$-factors}
\label{section 3.2}

Let
\begin{equation}
\label{3.2.1}
G:=G(F):=\{g\in\mathrm{GL}_n(D):g\Phi g^{\ast}=\Phi\},\,\Phi=\left[\begin{array}{ccc}
0 & 0 & 1_m\\
0 & \theta & 0\\
\epsilon\cdot 1_m & 0 & 0
\end{array}\right],
\end{equation}
with $n=2m+r$ and $\theta^{\ast}=\epsilon\theta\in\mathrm{GL}_r(D)$ not necessarily anisotropic. In the ramified cases, we will always assume that $2$ and $\theta$ are unramified, i.e. $2\in\mathcal{O}^{\times},\theta\in\mathrm{GL}_r(\mathcal{O})$. For an integer $\mathfrak{c}\geq 1$, we consider the following two open compact subgroups of $G(\mathfrak{o})$:
\begin{equation}
\label{3.2.2}
\begin{aligned}
K(\mathfrak{p}^{\mathfrak{c}})&=G(\mathfrak{o})\cap\left[\begin{array}{ccc}
\mathrm{Mat}_{m}(\mathcal{O}) & \mathrm{Mat}_{m,r}(\mathcal{O}) & \mathrm{Mat}_{m}(\mathcal{O})\\
\mathrm{Mat}_{r,m}(\mathfrak{p^c}\mathcal{O}) &1+\mathrm{Mat}_{r}(\mathfrak{p}\mathcal{O}) & \mathrm{Mat}_{r,m}(\mathcal{O})\\
\mathrm{Mat}_{m}(\mathfrak{p^c}\mathcal{O}) & \mathrm{Mat}_{m,r}(\mathfrak{p^c}\mathcal{O}) & \mathrm{Mat}_{m}(\mathcal{O})
\end{array}\right],\\
K'(\mathfrak{p}^{\mathfrak{c}})&=G(\mathfrak{o})\cap\left[\begin{array}{ccc}
\mathrm{Mat}_{m}(\mathcal{O}) & \mathrm{Mat}_{m,r}(\mathfrak{p^c}\mathcal{O}) & \mathrm{Mat}_{m}(\mathfrak{p^c}\mathcal{O})\\
\mathrm{Mat}_{r,m}(\mathcal{O}) &1+\mathrm{Mat}_{r}(\mathfrak{p}\mathcal{O}) & \mathrm{Mat}_{r,m}(\mathfrak{p^c}\mathcal{O})\\
\mathrm{Mat}_{m}(\mathcal{O}) & \mathrm{Mat}_{m,r}(\mathcal{O}) & \mathrm{Mat}_{m}(\mathcal{O})
\end{array}\right].
\end{aligned}
\end{equation}
Clearly, they are related by $K(\mathfrak{p^c})=wK'(\mathfrak{p^c})w$ with $w$ the Weyl element as \eqref{2.3.4}. Let
\begin{equation}
\label{3.2.3}
\mathfrak{M}=\mathrm{GL}_m(D)\cap\mathrm{Mat}_m(\mathcal{O}),\mathfrak{Q}=\{\mathrm{diag}[u,1_r,\hat{u}],u\in\mathfrak{M}\},\mathfrak{X}=K(\mathfrak{p^c})\mathfrak{Q}K(\mathfrak{p^c}).
\end{equation}
For $\xi=\mathrm{diag}[u,1_r,\hat{u}]\in\mathfrak{Q}$, we define $\mathfrak{d}(\xi)$ be the integer such that $\nu(u)=\widetilde{\varpi}^{\mathfrak{d}(\xi)}$. The local Hecke algebra $\mathcal{H}(K(\mathfrak{p^c}),\mathfrak{X})$ associated to $K(\mathfrak{p^c})$ and $\mathfrak{X}$ is generated by double cosets $[K(\mathfrak{p^c})\xi K(\mathfrak{p^c})]$ with $\xi\in\mathfrak{Q}$. This kind of Hecke algebra generalizes the one in \cite[Section 19]{Sh00}. Let $\pi$ be an admissible representation of $G(F)$. Assume $\phi\in\pi$ is a vector fixed by $K(\mathfrak{p^c})$, the Hecke operator $[K(\mathfrak{p^c})\xi K(\mathfrak{p^c})]$ acts on $\phi$ by
\begin{equation}
\label{3.2.4}
\phi|[K(\mathfrak{p^c})\xi K(\mathfrak{p^c})]=\int_{K(\mathfrak{p^c})\xi K(\mathfrak{p^c})}\pi(g)\phi dg.
\end{equation}
If we assume the measure $dg$ is normalized such that $K(\mathfrak{p^c})$ has volume $1$, then the action can be written as a sum
\begin{equation}
\label{3.2.5}
\phi|[K(\mathfrak{p^c})\xi K(\mathfrak{p^c})]=\sum_{K(\mathfrak{p^c})\xi K(\mathfrak{p^c})/K(\mathfrak{p^c})}\pi(g)\phi.
\end{equation}
The coset in the sum is characterized in the following lemma.

\begin{lem}
\label{lemma 3.2}
Let $\xi=\mathrm{diag}[u,1_r,\hat{u}]$ with $u\in\mathfrak{M}$ then
\begin{equation}
\label{3.2.6}
K(\mathfrak{p^c})\xi K(\mathfrak{p^c})=\coprod_{d,b,c} \left[\begin{array}{ccc}
d & -b^{\ast}\theta^{-1} & c\hat{d}\\
0 & 1 & b\hat{d}\\
0 & 0 & \hat{d}
\end{array}\right]K(\mathfrak{p^c}),
\end{equation}
where $d\in\mathrm{GL}_m(D)u\mathrm{GL}_m(D)/\mathrm{GL}_m(D)$, $b\in\mathrm{Mat}_{m,r}(\mathcal{O})/\mathrm{Mat}_{m,r}(\mathcal{O})d^{\ast}$ and\\ $c\in\mathrm{Mat}_{m}(\mathcal{O})/d\mathrm{Mat}_{m}(\mathcal{O})d^{\ast}$ satisfying $\epsilon c+b^{\ast}\hat{\theta}b+c^{\ast}=0$.
\end{lem}

\begin{proof}
This is an analogue of \cite[Lemma 19.2]{Sh00} and can be verified in a straightforward way.
\end{proof}

Assume $\phi\in\pi$ is an eigenvector for all $[K(\mathfrak{p^c})\xi K(\mathfrak{p^c})]$, that is there exists a scalar $\lambda_{\xi}$ such that
\begin{equation}
\label{3.2.7}
\phi|[K(\mathfrak{p^c})\xi K(\mathfrak{p^c})]=\lambda_{\xi}(\phi)\phi.
\end{equation}
For an integer $\mathfrak{n}\geq 1$, we consider a special Hecke operator
\begin{equation}
\label{3.2.8}
U(\mathfrak{p^n}):=[K(\mathfrak{p^c})\xi K(\mathfrak{p^c})]\text{ with }\xi=\mathrm{diag}[\varpi^{\mathfrak{n}}\cdot 1_m,1_r,\varpi^{-\mathfrak{n}}\cdot 1_m].
\end{equation}
Denote the Hecke eigenvalue for operator $U(\mathfrak{p^n})$ as $\alpha(\mathfrak{p^n})$. Clearly by Lemma \ref{lemma 3.2}, one has $U(\mathfrak{p^n})=U(\mathfrak{p})^{\mathfrak{n}}$ and $\alpha(\mathfrak{p^n})=\alpha(\mathfrak{p})^{\mathfrak{n}}$. In later computations, we will also use another kind of Hecke operator
\begin{equation}
\label{3.2.9}
U'(\mathfrak{p^n}):=[K(\mathfrak{p^c})\xi K(\mathfrak{p^c})]\text{ with }\xi=\left[\begin{array}{ccc}
0 & 0 & \varpi^{-\mathfrak{n}}\cdot 1_m\\
0 & 1_r & 0\\
\epsilon\varpi^{\mathfrak{n}}\cdot 1_m & 0 & 0
\end{array}\right].
\end{equation}
Its action on $\phi$ is defined similarly as above.

Assume $\chi$ is an unramified character, define the ramified local $L$-factors as
\begin{equation}
\label{3.2.10}
L\left(s+\frac{1}{2},\phi\times\chi\right)=\sum_{\xi\in K(\mathfrak{p^c})\backslash\mathfrak{X}/K(\mathfrak{p^c})}\lambda_{\xi}(\phi)\left(\chi(\nu(\widetilde{\varpi}))|\mathrm{N}_{E/F}(\nu(\widetilde{\varpi}))|^{s+\kappa}\right)^{\mathfrak{d}(\xi)}.
\end{equation}
If $\chi$ is ramified then we simply set
\begin{equation}
\label{3.2.11}
L\left(s+\frac{1}{2},\phi\times\chi\right)=1.
\end{equation}

\begin{prop}
\label{proposition 3.3}
Let $\beta_i,1\leq i\leq m$ be the Satake parameters of $\phi$ and assume $\chi$ is unramified. Then $L(s,\phi\times\chi)$ has an Euler product expansion with $L(s,\phi\times\chi)^{-1}$ given by the following list. \\
(Case I, Orthogonal)
\[
\prod_{i=1}^m\left(1-\chi(\varpi)\beta_iq^{-1+\frac{r}{2}-s}\right),
\]
(Case II, Symplectic)
\[
\prod_{i=1}^m\left(1-\chi(\varpi)\beta_iq^{-s}\right),
\]
(Case III, Quaternionic Orthogonal Nonsplit)
\[
\prod_{i=1}^m\left(1-\chi(\varpi)\beta_iq^{-1+r-s}\right),
\]
(Case IV, Quaternionic Unitary Nonsplit)
\[
\prod_{i=1}^m\left(1-\chi(\varpi)\beta_iq^{-2+r-s}\right),
\]
(Case V, Unitary Inert) $E/F$ is inert,
\[
\prod_{i=1}^m\left(1-\chi(\varpi)\beta_i q^{-1+r-s}\right),
\]
(Case V, Unitary Ramified) $E/F$ is ramified,
\[
\prod_{i=1}^m\left(1-\chi(\widetilde{\varpi})\beta_iq^{\frac{r-1}{2}-s}\right),
\]
\end{prop}

\begin{proof}
This is an analogue of \cite[Theorem 19.8]{Sh00} for symplectic and unitary cases. The proof for all the cases are the same so we only compute the orthogonal case as an example and omit the other cases. 

The Satake map $\omega:\mathbb{T}(K(\mathfrak{p^c}),\mathfrak{X})\to\Q[t_1,...,t_m]$ is defined as follows. Given a coset $d\mathrm{GL}_m(\mathfrak{o})$ for $d\in\mathrm{GL}_m(F)$, we can find a lower triangular matrix $g\in\mathrm{GL}_m(F)$ such that $d\mathrm{GL}_m(\mathfrak{o})=g\mathrm{GL}_m(\mathfrak{o})$. Assume the diagonal elements of $g$ are of the form $\varpi^{e_1},...,\varpi^{e_m}$ with $e_i\in\Z$ and set $\omega_0(d\mathrm{GL}_m(\mathfrak{o}))=\prod_{i=1}^m(q^{-i}t_i)^{e_i}$. For $K(\mathfrak{p^c})\xi K(\mathfrak{p^c})=\coprod_yyK(\mathfrak{p^c})$ with $y$ as in Lemma \ref{lemma 3.2}, we then define $\omega([K(\mathfrak{p^c})\xi K(\mathfrak{p^c})])=\sum_y\omega_0(y\mathrm{GL}_m(\mathfrak{o}))$.

Set $T=\chi(\varpi)q^{-s}$. By \cite[Lemma 19.9]{Sh00} we calculate the Dirichlet series
\[
\begin{aligned}
&\sum_{\xi\in K(\mathfrak{p^c})\backslash\mathfrak{X}/K(\mathfrak{p^c})}\lambda_{\xi}(\phi)T^{\nu(d)}=\sum_{d\in\mathrm{GL}_m(F)/\mathrm{GL}_m(\mathfrak{o})}\mathrm{vol}(b,c)\omega_0(d\mathrm{GL}_m(\mathfrak{o}))T^{\nu(d)}.
\end{aligned}
\]
Here $\mathrm{vol}(b,c)$ is the volume of $K(\mathfrak{p^c})\xi K(\mathfrak{p^c})/K(\mathfrak{p^c})$ with fixed $d$. Clearly 
\[
\mathrm{vol}\left(\mathrm{Mat}_{m,r}(\mathfrak{p^c}\mathcal{O})/\mathrm{Mat}_{m,r}(\mathfrak{p^c}\mathcal{O})d^{\ast}\right)=|\nu(d)|^{-r},
\]
and $c':=cJ+\frac{1}{2}b^{\ast}\hat{\theta} b$ satisfies $c+\epsilon c'^{\ast}=0$. Then by \cite[Lemma 13.2]{Sh97} we have
\[
\begin{aligned}
&\sum_{\xi\in K(\mathfrak{p^c})\backslash\mathfrak{X}/K(\mathfrak{p^c})}\lambda_{\xi}(\phi)T^{\nu(d)}\\
=&\sum_{d\in\mathrm{GL}_m(F)/\mathrm{GL}_m(\mathfrak{o})}|\nu(d)|^{-r-m+1}\omega_0(d\mathrm{GL}_m(\mathfrak{o}))T^{\nu(d)}\\
=&\prod_{i=1}^m(1-\chi(\varpi)\beta_iq^{m+r-2-s})^{-1}
\end{aligned}
\]
Changing $s\mapsto s+\frac{n-1}{2}-\frac{1}{2}$ we obtain the result in above list.
\end{proof}

\subsection{The split case}
\label{section 3.3}
Let
\begin{equation}
\label{3.3.1}
G:=G(F):=\{g\in\mathrm{GL}_n(D):g\Phi g^{\ast}=\Phi\},\,\Phi=\left[\begin{array}{ccc}
0 & 0 & 1_m\\
0 & \theta & 0\\
\epsilon\cdot 1_m & 0 &0
\end{array}\right],
\end{equation}
with $n=2m+r$ and $\theta^{\ast}=\epsilon\theta\in\mathrm{GL}_r(D)$ not necessarily anisotropic. In Case III', IV', this group is isomorphic to
\begin{equation}
\label{3.3.2}
\widetilde{G}:=\widetilde{G}(F):=\{g\in\mathrm{GL}_{2n}(F):g\widetilde{\Phi}\transpose{g}=\widetilde{\Phi}\},\,\widetilde{\Phi}=\left[\begin{array}{ccc}
0 & 0 & 1_{2m}\\
0 & \widetilde{\theta} & 0\\
-\epsilon\cdot 1_{2m} & 0 & 0
\end{array}\right],
\end{equation}
with $\transpose{\widetilde{\theta}}=-\epsilon\widetilde{\theta}\in\mathrm{GL}_{2r}(F)$. In Case V', $G\cong\mathrm{GL}_n(F)$ is simply the general linear group. We omit the discussion of Case V' for simplicity as it is well studied in \cite{Sh97,Sh00}. 

\subsubsection{Unramified local $L$-factors}

The group $\widetilde{G}$ is further isomorphic to
\begin{equation}
\label{3.3.3}
\widetilde{G}':=\widetilde{G}'(F):=\{g\in\mathrm{GL}_{2n}(F),g\widetilde{\Phi}'\transpose{g}=\widetilde{\Phi}'\},\,\widetilde{\Phi}'=\left[\begin{array}{ccc}
0 & 0 & 1_{m'}\\
0 & \theta' & 0\\
-\epsilon\cdot 1_{m'} & 0 & 0
\end{array}\right],
\end{equation}
with $2n=2m'+r'$ and $\transpose{\theta}'=-\epsilon\theta'\in\mathrm{GL}_{r'}(F)$ anisotropic. This is a group of Case I or II discussed in Section \ref{section 3.1}. Assume $\chi$ is an unramified character of $F^{\times}$, $\pi$ an unramified admissible representation of $G(F)$ and $\phi\in\pi$ a spherical vector. Let $\pi'$ be an unramified admissible representation of $\widetilde{G}'(F)$ and $\phi'\in\pi'$ a spherical vector obtained from $\pi,\phi$ under the isomorphism $G\cong\widetilde{G}'$. We thus define the local $L$-factor $L(s,\phi\times\chi):=L(s,\phi'\times\chi)$ as in Section \ref{section 3.1}. In particular, the normalizing factors are\\
(Case III', Quaternionic Orthogonal Split) 
\[
\begin{aligned}
b(s,\chi)&=L\left(s+\frac{2n+1}{2},\chi\right)\prod_{i=1}^{n}L(2s-1+2i,\chi^2),
\end{aligned}
\]
(Case IV', Quaternionic Unitary Split) 
\[
\begin{aligned}
b(s,\chi)&=\prod_{i=1}^{n}L(2s+2n+1-2i,\chi^2).
\end{aligned}
\]
Let $\alpha_i$ be the Satake parameters of $\phi$ then $L(s,\phi\times\chi)^{-1}$ are given by\\
(Case III', Quaternionic Orthogonal Split)
\[
\left(1-\chi(\varpi)q^{-s}\right)\times\prod_{i=1}^{n}\left(1-\chi(\varpi)\alpha_iq^{-s}\right)\left(1-\chi(\varpi)\alpha^{-1}_iq^{-s}\right),
\]
(Case IV', Quaternionic Unitary Split)
\[
\prod_{i=1}^{\lfloor\frac{r'}{2}\rfloor}\left(1-\chi(\varpi)^2q^{2i-r'-2s}\right)\times\prod_{i=1}^{m'}\left(1-\chi(\varpi)\alpha_iq^{-1+\frac{r'}{2}-s}\right)\left(1-\chi(\varpi)\alpha_i^{-1}q^{1-\frac{r'}{2}-s}\right).
\]

\subsubsection{Ramified local $L$-factors}

For an integer $\mathfrak{c}\geq 1$, we consider the open compact subgroup $K(\mathfrak{p^c})$ of $G(\mathfrak{o})$ as in Section \ref{section 3.2}. The isomorphism between $G$ and $\widetilde{G}$ can be chosen such that the image of $G(\mathfrak{o})$ is $\widetilde{G}(\mathfrak{o})$. We will fix such an isomorphism throughout the paper. In this case, the image of $K(\mathfrak{p^c})$ is
\begin{equation}
\label{3.3.4}
\widetilde{K}(\mathfrak{p^c})=\widetilde{G}(\mathfrak{o})\cap\left[\begin{array}{ccc}
\mathrm{Mat}_{2m}(\mathfrak{o}) & \mathrm{Mat}_{2m,2r}(\mathfrak{o}) & \mathrm{Mat}_{2m}(\mathfrak{o})\\
\mathrm{Mat}_{2r,2m}(\mathfrak{p^c}\mathfrak{o}) &1+\mathrm{Mat}_{2r}(\mathfrak{p}\mathfrak{o}) & \mathrm{Mat}_{2r,2m}(\mathfrak{o})\\
\mathrm{Mat}_{2m}(\mathfrak{p^c}\mathfrak{o}) & \mathrm{Mat}_{2m,2r}(\mathfrak{p^c}\mathfrak{o}) & \mathrm{Mat}_{2m}(\mathfrak{o})
\end{array}\right]
\end{equation}
We can define the Hecke algebras $\mathcal{H}(K(\mathfrak{p^c}),\mathfrak{X})$ and $\mathcal{H}(\widetilde{K}(\mathfrak{p^c}),\mathfrak{X})$ similarly as Section \ref{section 3.2}. Let $\pi$ be an admissible representation of $G(F)$. Assume $\phi\in\pi$ is a vector fixed by $K(\mathfrak{p^c})$ and is an eigenvector for the Hecke algebra $\mathcal{H}(K(\mathfrak{p^c}),\mathfrak{X})$. Let $\pi'$ be an admissible representation of $\widetilde{G}(F)$ and $\phi'\in\pi$ a vector obtained from $\pi,\phi$
under the isomorphism $G\cong\widetilde{G}$. We thus define the local $L$-factor $L(s,\phi\times\chi):=L(s,\phi'\times\chi)$ as in \eqref{3.2.10}, \eqref{3.2.11}. In particular, $L(s,\phi\times\chi)^{-1}$ are given by\\
(Case III', Quaternionic Orthogonal Split)
\[
\prod_{i=1}^{2m}\left(1-\chi(\varpi)\beta_iq^{-s}\right),
\]
(Case IV', Quaternionic Unitary Split)
\[
\prod_{i=1}^{2m}\left(1-\chi(\varpi)\beta_iq^{-1+r-s}\right).
\]
The operator $U'(\mathfrak{p^n})$ is defined as in \eqref{3.2.9} for orthogonal and symplectic groups. In Case IV', the $U(\mathfrak{p^n})$ is also the one defined for Case I in \eqref{3.2.8}. In Case III', we define $U(\mathfrak{p^n})$ as
\begin{equation}
\label{3.3.5}
U(\mathfrak{p^n}):=[\widetilde{K}(\mathfrak{p^c})\xi\widetilde{K}(\mathfrak{p^c})]\text{ with }\xi=\mathrm{diag}[\varpi^{\mathfrak{n}}\cdot 1_n,\varpi^{-\mathfrak{n}}\cdot 1_n].
\end{equation}

\begin{rem}
\label{remark 3.4}
Note that when defining local $L$-factors, we always assume the group is chosen such that $m$ is the global Witt index in unramified cases which is not applied for ramified cases. In other words, in ramified cases our open compact subgroup $K(\mathfrak{p^c})$ is not chosen to be maximal. For example in above Case III', clearly the local group $\widetilde{G}$ can be further isomorphic to $\widetilde{G}'$ with $m'=n, r'=0$ as in the unramified computations. But we are still considering the open compact subgroup $\widetilde{K}(\mathfrak{p^c})$ rather than the bigger one
\[
\widetilde{G}'(\mathfrak{o})\cap\left[\begin{array}{cc}
\mathrm{Mat}_{n}(\mathfrak{o}) & \mathrm{Mat}_n(\mathfrak{o})\\
\mathrm{Mat}_{n}(\mathfrak{p^c}\mathfrak{o}) & \mathrm{Mat}_n(\mathfrak{o})
\end{array}\right],
\]
which causes our local $L$-factors to be of degree $2m$ rather than the expected $n$ as in Case II. We make those restrictions because only these $L$-factors show up in our integral representations.
\end{rem}

\subsection{The global $L$-function}
\label{section 3.4}

We summarize our definition for the standard $L$-function. Let $F$ be a number field and
\begin{equation}
\label{3.4.1}
G:=G(F):=\{g\in\mathrm{GL}_n(D):g\Phi g^{\ast}=\Phi\},\,\Phi=\left[\begin{array}{ccc}
0 & 0 & 1_m\\
0 & \theta & 0\\
\epsilon\cdot 1_m & 0 & 0
\end{array}\right],
\end{equation}
be the global group as in Section \ref{section 2}. Let $\mathfrak{o}$ be the ring of integers of $F$ and $\mathcal{O}$ a fixed maximal order of $D$ such that $D=F\otimes_{\mathfrak{o}}\mathcal{O}$. Let $\mathfrak{n}=\mathfrak{n}_1\mathfrak{n}_2$ be an integral ideal of $\mathfrak{o}$ with $\mathfrak{n}_1,\mathfrak{n}_2$ coprime and $\chi:E^{\times}\backslash\mathbb{A}_E^{\times}\to\C^{\times}$ a Hecke character of conductor $\mathfrak{n}_2$. Consider the open compact subgroup
\begin{equation}
\label{3.4.2}
K(\mathfrak{n})=G(\mathfrak{o})\cap\left[\begin{array}{ccc}
\mathrm{Mat}_{m}(\mathcal{O}) & \mathrm{Mat}_{m,r}(\mathcal{O}) & \mathrm{Mat}_{m}(\mathcal{O})\\
\mathrm{Mat}_{r,m}(\mathfrak{n}\mathcal{O}) &1+\mathrm{Mat}_{r}(\mathfrak{n'}\mathcal{O}) & \mathrm{Mat}_{r,m}(\mathcal{O})\\
\mathrm{Mat}_{m}(\mathfrak{n}\mathcal{O}) & \mathrm{Mat}_{m,r}(\mathfrak{n}\mathcal{O}) & \mathrm{Mat}_{m}(\mathcal{O})
\end{array}\right],
\end{equation}
with $\mathfrak{n}'$ the support of $\mathfrak{n}$. For $v$ a finite place corresponds to a prime ideal $\mathfrak{p}$ denote $\mathfrak{M}_v=\mathrm{GL}_m(D_v)\cap\mathrm{Mat}_m(\mathcal{O}_v)$ and set
\begin{equation}
\label{3.4.3}
\mathfrak{Q}_v=\left\{\begin{array}{cc}
G(\mathfrak{o}_v) & \mathfrak{p}\nmid\mathfrak{n},\\
\mathrm{diag}[u,1_r,\hat{u}]\text{ with }u\in\mathfrak{M}_v& \mathfrak{p}|\mathfrak{n_1}\\
1 & \mathfrak{p}|\mathfrak{n_2}.
\end{array}\right.
\end{equation}
Let $\mathfrak{Q}=\prod_v\mathfrak{Q}_v$ and $\mathfrak{X}=K(\mathfrak{n})\mathfrak{Q}K(\mathfrak{n})$. Define the global Hecke algebra $\mathcal{H}:=\mathcal{H}(K(\mathfrak{n}),\mathfrak{X})$ associated to $K(\mathfrak{n})$ and $\mathfrak{X}$ be the one generated by double cosets $[K(\mathfrak{n})\xi K(\mathfrak{n})]$ with $\xi\in\mathfrak{Q}$. Its action on a cusp form $\phi\in\pi$ can be similarly defined as in the local case treated above. Assume $\phi\in\pi$ is fixed by $K(\mathfrak{n})$ and is an eigenfunction for $\mathcal{H}$. That is
\begin{equation}
\label{3.4.4}
\phi|[K(\mathfrak{n})\xi K(\mathfrak{n})]=\lambda_{\xi}(\phi)\phi.
\end{equation}
The standard $L$-function of $\phi$ twisted by $\chi$ is defined as
\begin{equation}
\label{3.4.5}
L\left(s+\frac{1}{2},\phi\times\chi\right)=b(s,\chi)\sum_{\substack{\xi\in K(\mathfrak{n})\backslash\mathfrak{X}/K(\mathfrak{n})\\\xi=\mathrm{diag}[u,1,\hat{u}]}}\lambda_{\xi}(\phi)\chi(\nu(u))|\mathrm{N}_{E/F}(\nu(u))|^{s+\kappa}.
\end{equation}
Clearly, it has an Euler product expression
\begin{equation}
\label{3.4.6}
L\left(s+\frac{1}{2},\phi\times\chi\right)=\prod_vL_v\left(s+\frac{1}{2},\phi_v\times\chi_v\right),
\end{equation}
with $L_v\left(s+\frac{1}{2},\phi_v\times\chi_v\right)$ the local $L$-factors defined in the previous two subsections. Note that when $v\nmid\mathfrak{n}$ is unramified, $G(\mathfrak{o}_v)$ may have Witt index $m'\geq m$ and the unramified local $L$-factors is defined by replacing $m$ with $m'$ in Section \ref{section 3.1}.

\begin{rem}
\label{remark 3.5}
We give several remarks on our $L$-functions.\\
(1) Here we define the $L$-function by a Dirichlet series associated to certain Hecke eigenvalues which can be viewed as an analogue of the $L$-function for classical ($\mathrm{GL}_2$) modular forms. This kind of $L$-function is also studied in \cite{Sh97,Sh00} for symplectic and unitary groups.\\
(2) As the reader may notice, we are writing $L(s,\phi\times\chi)$ to indicate its dependence on the cusp form $\phi$. Unlike the $\mathrm{GL}_2$ case, we do not have a clear correspondence between eigenforms and cuspidal representations. This is because, the subspace of $\pi$ fixed by $K(\mathfrak{n})$ (take $\mathfrak{n}$ to be the conductor of $\pi$) may not be of dimension one. However, we remark that newform theory is recently known for unitary groups in \cite{AOY24, Ato25}.\\
(3) The unramified local $L$-factors defined here are really the Langlands $L$-function associated to the natural embedding of the $L$-group ${^L}G$ into a general linear group.\\
(4) We make no claim that our definition of ramified $L$-factors is `correct'. Indeed, it is a conjecture of Langlands \cite{L70} that for any cuspidal representation $\pi$ one can associate to any place is a local $L$-factor $L_v(s,\pi_v)$ and a local root number $\epsilon_v(s,\pi_v)$ such that the global $L$-function satisfies a functional equation of the form
\[
L(s,\pi)=\epsilon(s,\pi)L(1-s,\pi).
\]
Using the doubling method, Yamana \cite{Y14} gives a definition of local $L$-factors and proves the functional equation for classical groups. However, he does not define these factors explicitly as in our Proposition \ref{proposition 3.1}, \ref{proposition 3.3} and it is not clear how his approach can be used to study algebraicity or $p$-adic properties which is interested in this paper. We have not compared our $L$-factors with his and we also do not know whether the $L$-function defined here can be proved to satisfy a functional equation.
\end{rem}

\section{The Non-archimedean Local Integrals}
\label{section 4}

We carried out the computations of the non-archimedean local integrals in this section. We keep the setting for non-archimedean local fields and tuples $(D,\rho,\epsilon)$ as in the beginning of Section \ref{section 3}. The unramified local integrals are well known but we will also review the computations for completeness. For ramified local integrals, it is also well known that one can choose a local section $f_{s,v}$ such that $\mathcal{Z}_v(s;\phi_{1,v},\phi_{2,v},f_{s,v})\neq 0$. For our purpose, we explicitly construct certain local sections $f_s^{\dagger},f_s^{\ddagger},f_s^p$ such that $\mathcal{Z}$ represents the local $L$-factors defined in the last section or the $p$-adic modifications. These local sections will also be chosen such that the Eisenstein series has a nice Fourier expansion. To make our notations and computations consistent, we do not consider the split cases (i.e. Case III', IV', V'). For Case III', IV', the local groups are isomorphic to the groups in Case I, II and our arguments can be directly extended to these two cases. The Case V' should be treated separately and we omit it for simplicity. We will assume in this paper that Case V' does not occur in the ramified setting.

\subsection{Setup for non-archimedean local integrals}
\label{section 4.1}
Let 
\begin{equation}
\label{4.1.1}
G:=G(F):=\{g\in\mathrm{GL}_n(D):g\Phi g^{\ast}=\Phi\},\,\Phi=\left[\begin{array}{ccc}
0 & 0 & 1_m\\
0 & \theta & 0\\
\epsilon\cdot 1_m & 0 &0
\end{array}\right],
\end{equation}
with $n=2m+r$ and $\theta^{\ast}=\epsilon\theta\in\mathrm{GL}_r(D)$ not necessarily anisotropic. Let
\begin{equation}
\label{4.1.2}
H:=H(F):=\{h\in\mathrm{GL}_{2n}(D):hJ_n h^{\ast}=J_n\},\,J_n=\left[\begin{array}{cc}
0 & 1_n\\
\epsilon\cdot 1_n & 0
\end{array}\right],
\end{equation}
and define an embedding
\begin{equation}
\label{4.1.3}
\begin{aligned}
G\times G&\to H,\\
(g_1,g_2)&\mapsto R\left[\begin{array}{cc}
g_1 & 0 \\
0 & g_2
\end{array}\right]R^{-1},
\end{aligned}
\end{equation}
with
\[
R=\left[\begin{array}{cccccc}
0 & \frac{\epsilon}{2}\cdot 1_r & 0 & 0 & \frac{\epsilon}{2}\cdot 1_r & 0\\
0 & 0 & 0 & 0 & 0 & -\epsilon\cdot 1_m\\
1_m & 0 & 0 & 0 & 0 & 0\\
0 & \theta^{-1} & 0 & 0 & -\theta^{-1} & 0\\
0 & 0 & 0 & 1_m & 0 & 0\\
0 & 0 & 1_m & 0 & 0 & 0
\end{array}\right].
\]
We identify $(g_1,g_2)$ with its image in $H$. Let $P\subset H$ be the Siegel parabolic subgroup consisting elements of the form $\left[\begin{array}{cc}
\ast & \ast\\
0_n & \ast
\end{array}\right]$, with Levi decomposition $P=M\ltimes N$. Consider the induced representation $\mathrm{Ind}_{P(F)}^{H(F)}(\chi|\nu(\cdot)|^{s})$ for a character $\chi:E^{\times}\to\C^{\times}$. Let $\pi$ be an admissible representation of $G(F)$ and $\phi_1,\phi_2\in\pi$. For a section $f_s\in\mathrm{Ind}_{P(F)}^{H(F)}(\chi|\nu(\cdot)|^s)$ we consider the local integral
\begin{equation}
\label{4.1.4}
\mathcal{Z}(s;\phi_1,\phi_2,f_s)=\int_{G(F)}f_s(\delta(g,1))\langle\pi(g)\phi_1,\phi_2\rangle dg,
\end{equation}
where
\[
\delta=\left[\begin{array}{cccccc}
1_r & 0 & 0 & 0 & 0 & 0\\
0 & 1_m & 0 & 0 & 0 & 0\\
0 & 0 & 1_{m} & 0 & 0 & 0\\
0 & 0 & 0 & 1_{r} & 0 & 0\\
0 & 0 & -1_m & 0 & 1_m & 0\\
0 & \epsilon\cdot 1_m & 0 & 0 & 0 & 1_m
\end{array}\right].
\]

\subsection{The unramified local integrals}
\label{section 4.2}

Assume $\chi$ is an unramified character of $E^{\times}$. Let $\pi$ be an unramified admissible representation of $G(F)$ and $\phi_1=\phi_2=\phi\in\pi$ a spherical vector.

Take the local section $f_s^0\in\mathrm{Ind}_{P(F)}^{H(F)}(\chi|\nu\cdot|^s)$ to be the spherical section normalized such that
\begin{equation}
\label{4.2.1}
f_s^0(pk)=\chi(\nu(a))|\mathrm{N}_{E/F}(\nu(a))|^{s+\kappa}b(s,\chi)\text{ with }p=\left[\begin{array}{cc}
a & b\\
0 & \hat{a}
\end{array}\right],k\in H(\mathfrak{o}).
\end{equation}
Here $b(s,\chi)$ is the normalizing factor given in Section \ref{section 3.1}. Note that in \eqref{4.1.1}, $\theta$ is not necessarily anisotropic. Let $m'$ be the Witt index of $G$. That is $G$ is isomorphic to the following $F$-group
\begin{equation}
\label{4.2.2}
G':=G'(\Phi'):=\{g\in\mathrm{GL}_n(D):g\Phi'g^{\ast}=\Phi'\}
\end{equation}
with
\[
\Phi'=\left[\begin{array}{ccc}
0 & 0 & 1_{m'}\\
0 & \theta' & 0\\
\epsilon 1_{m'} & 0 & 0
\end{array}\right],n=2m'+r',\theta'^{\ast}=\epsilon\theta'\in\mathrm{GL}_{r'}(D)\text{ anisotropic }.
\]
Then there exists a matrix $S\in\mathrm{GL}_r(D)$ with $S\Phi' S^{\ast}=\Phi$ and the isomorphism between $G$ and $G'$ are given by
\begin{equation}
\label{4.2.3}
\begin{aligned}
G'&\stackrel{\sim}\longrightarrow G,\\
g&\mapsto SgS^{-1}.
\end{aligned}
\end{equation}
Denote by $\pi',\phi'$ for the admissible representation and cusp form of $G'(F)$ obtained from $\pi,\phi$ under isomorphism \eqref{4.2.3}. Then the local $L$-factor $L(s,\phi\times\chi)=L(s,\phi'\times\chi)$ is defined in Section \ref{section 3.1} with $m$ replaced by $m'$. 

\begin{prop}
\label{proposition 4.1}
For $f_s^0$ chosen as above and $\phi_1=\phi_2=\phi\in\pi$ a spherical vector, we have
\begin{equation}
\label{4.2.4}
\mathcal{Z}(s;\phi,\phi,f_s^0)=L\left(s+\frac{1}{2},\phi\times\chi\right)\langle\phi,\phi\rangle.
\end{equation}
\end{prop}

\begin{proof}
The computation is similar to the one in \cite{Pi21} for symplectic groups. We have another doubling embedding
\[
\begin{aligned}
G'\times G'&\to H,\\
(g_1,g_2)&\mapsto R'\left[\begin{array}{cc}
g_1 & 0\\
0 & g_2
\end{array}\right]R'^{-1},\qquad R'=R\left[\begin{array}{cc}
S & 0\\
0 & S
\end{array}\right]
\end{aligned}
\]
which is compatible with \eqref{4.1.3},\eqref{4.2.3}. Then the integral \eqref{4.1.4} is equivalent to the integral 
\[
\mathcal{Z}(s;\phi',\phi',f_s)=\int_{G'(F)}f_s(\delta(g,1))\langle\pi'(g)\phi',\phi'\rangle dg.
\]
By Cartan decomposition, we have
\[
\begin{aligned}
&\sum_{\substack{e_1,...,e_{m'}\in\Z\\0\leq e_1\leq ...\leq e_{m'}}}f^0_s(\delta\cdot(\mathrm{diag}[\widetilde{\varpi}^{e_1},...,\widetilde{\varpi}^{e_{m'}},1_{r'},\widetilde{\varpi}^{-e_{1}},...,\widetilde{\varpi}^{-e_{m'}}],1))\int_{K_{e_1,...,e_{m'}}}\langle\pi'(g)\phi',\phi'\rangle dg\\
=&\sum_{\substack{e_1,...,e_{m'}\in\Z\\0\leq e_1\leq ...\leq e_{m'}}}\left(\chi(\nu(\widetilde{\varpi}))|\mathrm{N}_{E/F}(\nu(\widetilde{\varpi}))|^{s+\kappa}\right)^{e_1+...+e_{m'}}b(s,\chi)\lambda_{e_1,...,e_{m'}}(\phi')\langle\phi',\phi'\rangle\\
=&L\left(s+\frac{1}{2},\phi'\times\chi\right)\langle\phi',\phi'\rangle.
\end{aligned}
\]
\end{proof}

\subsection{The local section $f_s^{\dagger,\mathfrak{c}}$}
\label{section 4.3}

In this and the next two subsections, we consider the ramified local integrals. We will always assume $2\in\mathcal{O}^{\times}$ and $\theta\in\mathrm{GL}_r(\mathcal{O})$ in the ramified cases. For an integer $\mathfrak{c}\geq 1$ let $N'(\mathfrak{p^c})$ be the subgroup of $N(F)$ consisting of elements of the form
\begin{equation}
\label{4.3.1}
\left[\begin{array}{cccc}
1_r & 0 & x & y\\
0 & 1_{2m} & \epsilon y^{\ast} & z\\
0 & 0 & 1_r & 0\\
0 & 0 & 0 & 1_{2m}
\end{array}\right],\,x\in S_r(\mathfrak{p}\mathcal{O}),y\in\mathrm{Mat}_{r,2m}(\mathfrak{p^c}\mathcal{O}),z\in S_{2m}(\mathfrak{p^c}\mathcal{O}).
\end{equation}
Define $f_s^{\dagger,\mathfrak{c}}\in\mathrm{Ind}_{P(F)}^{H(F)}(\chi|\nu(\cdot)|^s)$ to be a local section supported on $P(F)J_nN'(\mathfrak{p^c})$ with
\begin{equation}
\label{4.3.2}
f_s^{\dagger,\mathfrak{c}}(pJ_nn)=\chi(\nu(a))|\mathrm{N}_{E/F}(\nu(a))|^{s+\kappa}\text{ for }p=\left[\begin{array}{cc}
a & b\\
0 & \hat{a}
\end{array}\right]\in P(F),n\in N'(\mathfrak{p^c}).
\end{equation}
Note that when pulling back along $G\times G\to H$, we have
\begin{equation}
\label{4.3.3}
f_s^{\dagger,\mathfrak{c}}((g_1k_1,g_2k_2))=\chi(\nu(d_{k_1}))\chi(\nu(a_{k_2}))f_s^{\dagger,\mathfrak{c}}((g_1,g_2))
\end{equation}
for $k_1=\left[\begin{array}{ccc}
a_{k_1} & f_{k_1} & b_{k_1}\\
h_{k_1} & e_{k_1} & j_{k_1}\\
c_{k_1} & k_{k_1} & d_{k_1}
\end{array}\right]\in K'(\mathfrak{p^c})$ and $k_2=\left[\begin{array}{ccc}
a_{k_2} & f_{k_2} & b_{k_2}\\
h_{k_2} & e_{k_2} & j_{k_2}\\
c_{k_2} & k_{k_2} & d_{k_2}
\end{array}\right]\in K(\mathfrak{p^c})$.

The following proposition is an analogue of the computations in \cite{Sh95} for symplectic groups. We extend the arguments there to all classical groups. Especially, the main difficulty in the computations is to deal with the group with $r\neq 0$.

\begin{prop}
\label{proposition 4.2}
Assume $\phi\in\pi$ is fixed by $K(\mathfrak{p^c})$ and is an eigenvector of the Hecke algebra $\mathcal{H}(K(\mathfrak{p^c}),\mathfrak{X})$. Set $\phi_1=\pi(w)\phi,\phi_2=\phi$. Assume $\chi$ is an unramified character and $2\in\mathcal{O}^{\times},\theta\in\mathrm{GL}_r(\mathcal{O})$, then
\begin{equation}
\label{4.3.4}
\begin{aligned}
&\mathcal{Z}(s;\pi(w)\phi,\phi,f_s^{\dagger,\mathfrak{c}})=\chi(\varpi)^{\mathfrak{c}m\mathbf{d}_1}q^{-\mathfrak{c}m\mathbf{d}_2(s+\kappa)}L\left(s+\frac{1}{2},\phi\times\chi\right)\cdot\left\langle\phi|U'(\mathfrak{p^c}),\phi\right\rangle.
\end{aligned}
\end{equation}
Here 
\begin{equation}
\label{4.3.5}
\mathbf{d}_1=\left\{\begin{array}{cc}
1 & \text{ Case I, II, V},\\
2 & \text{ Case III, IV},
\end{array}\right.\qquad\mathbf{d}_2=\left\{\begin{array}{cc}
1 & \text{ Case I, II},\\
2 & \text{ Case III, IV, V},
\end{array}\right.
\end{equation}
\end{prop}

\begin{proof}
Denote
\[
\mathfrak{M}(\mathfrak{p^c})=\mathrm{GL}_m(D)\cap\mathrm{Mat}_m(\mathfrak{p^c}\mathcal{O}),\,\mathfrak{Q}(\mathfrak{p^c})=\{\mathrm{diag}[\hat{u},-1_r,u],u\in\mathfrak{M}(\mathfrak{p^c})\}.
\]
We claim that
\[
f_s^{\dagger}(\delta(g,1))\neq 0\text{ if and only if }g\in K(\mathfrak{p^c})\mathfrak{Q}(\mathfrak{p^c})K'(\mathfrak{p^c}).
\]

Write $g=\left[\begin{array}{ccc}
a & f & b\\
h & e & j\\
c & k & d
\end{array}\right]$ with $a,d$ of size $m\times m$, $e$ of size $r\times r$ and compute
\[
\delta(g,1)=\left[\begin{array}{cccccc}
\frac{e+1}{2} & 0 & \frac{\epsilon h}{2} & \frac{\epsilon(e-1)\theta}{4} & 0 & \frac{\epsilon j}{2}\\
0 & 1 & 0 & 0 & 0 & 0\\
\epsilon f & 0 & a & \frac{f\theta}{2} & 0 & b\\
\epsilon\theta^{-1}(e-1) & 0 & \theta^{-1}h & \theta^{-1}\frac{e+1}{2}\theta & 0 & \theta^{-1}j\\
-\epsilon f & 0 & -a & -\frac{f\theta}{2} & 1 & -b\\
\epsilon k & \epsilon & c & \frac{k\theta}{2} & 0 & d
\end{array}\right].
\]
The elements in $P(F)J_nN'(\mathfrak{p^c})$ can be written as
\[
\begin{aligned}
&\left[\begin{array}{cccc}
\ast & \ast & \ast & \ast\\
\ast & \ast & \ast & \ast\\
0 & 0 & D_1 & D_2\\
0 & 0 & D_3 & D_4
\end{array}\right]J_n\left[\begin{array}{cccc}
1_r & 0 & x & y\\
0 & 1_{2m} & \epsilon y^{\ast} & z\\
0 & 0 & 1_r & 0\\
0 & 0 & 0 & 1_{2m}
\end{array}\right]\\
=&\left[\begin{array}{cccc}
\ast & \ast & \ast & \ast\\
\ast & \ast & \ast & \ast\\
\epsilon D_1 & \epsilon D_2 & \epsilon D_1x+D_2y^{\ast} & \epsilon D_1y+\epsilon D_2z\\
\epsilon D_3 & \epsilon D_4 & \epsilon D_3x+D_4y^{\ast} & \epsilon D_3y+\epsilon D_4z
\end{array}\right],
\end{aligned}
\]
with $x\in S_r(\mathfrak{p}\mathcal{O}),y\in\mathrm{Mat}_{r,2m}(\mathfrak{p^c}\mathcal{O}),z\in S_{2m}(\mathfrak{p^c}\mathcal{O})$. Assume $\delta(g,1)$ is of above form, then comparing two expressions we need
\[
D_1=\theta^{-1}(e-1),\,D_2=\left[\begin{array}{cc}
0 & \epsilon\theta^{-1}h\end{array}\right],\,D_3=\left[\begin{array}{c}
-f\\k
\end{array}\right],\,D_4=\left[\begin{array}{cc}
0 & -\epsilon a\\
1 & \epsilon c
\end{array}\right],
\]
and
\[
\begin{aligned}
\epsilon D_1x+D_2y^{\ast}&=\theta^{-1}\frac{e+1}{2}\theta,
&D_1y+D_2z&=\left[\begin{array}{cc}
0 & \epsilon\theta^{-1}j
\end{array}\right],\\
\epsilon D_3x+D_4y^{\ast}&=\left[\begin{array}{c}
-\frac{f\theta}{2}\\
\frac{k\theta}{2}
\end{array}\right],&D_3y+D_4z&=\left[\begin{array}{cc}
\epsilon & -\epsilon b\\
0 & \epsilon d
\end{array}\right].
\end{aligned}
\]
First of all, write $y=\left[\begin{array}{cc}
y_1 & y_2
\end{array}\right]$, then
\[
\left[\begin{array}{c}
-\frac{f\theta}{2}\\
\frac{k\theta}{2}
\end{array}\right]=\epsilon D_3x+D_4y^{\ast}=\left[\begin{array}{cc}
-\epsilon fx-\epsilon ay_2^{\ast}\\
\epsilon kx+y_1^{\ast}+\epsilon cy_2^{\ast}
\end{array}\right]
\]
implies
\[
ay_2^{\ast}=-\epsilon f(\epsilon x-\frac{\theta}{2}),\qquad y_1^{\ast}+\epsilon cy_2^{\ast}=-k(\epsilon x-\frac{\theta}{2}).
\]
Since by our assumption $\frac{\theta}{2}\in\mathrm{GL}_r(\mathcal{O})$ and the condition on $x$, the first equation forces $a$ to be invertible and thus 
\[
y_2^{\ast}=-\epsilon a^{-1}f(\epsilon x-\frac{\theta}{2}),\qquad y_1^{\ast}=-(k-ca^{-1}f)(\epsilon x-\frac{\theta}{2}).
\]
The condition on $y$ then forces $a^{-1}f,k-ca^{-1}f\in\mathrm{Mat}_{m,r}(\mathfrak{p^c}\mathcal{O})$. Secondly, from
\[
\theta^{-1}\frac{e+1}{2}\theta=\epsilon D_1x+D_2y^{\ast}=\epsilon\theta^{-1}(e-1)x+\epsilon\theta^{-1}hy_2^{\ast},
\]
we obtain
\[
\epsilon(e-1+ha^{-1}f)x=\frac{e+1-ha^{-1}f}{2}\theta.
\]
The condition on $x$ then forces $e-ha^{-1}f+1\in\mathrm{Mat}_r(\mathfrak{p}\mathcal{O})$. Finally, comparing $D_4$ and $D_3y+D_4z$, the condition on $z$ forces
\[
\left[\begin{array}{cc}
0 & -\epsilon a\\
1 & \epsilon c
\end{array}\right]^{-1}\left[\begin{array}{cc}
\epsilon & -\epsilon b\\
0 & \epsilon d
\end{array}\right]=\left[\begin{array}{cc}
\epsilon ca^{-1} & \epsilon d\\
-a^{-1} & a^{-1}b
\end{array}\right]\in\mathrm{Mat}_{2m}(\mathfrak{p^c}\mathcal{O})
\]
and hence $a^{-1},ca^{-1},a^{-1}b\in\mathrm{Mat}_{m}(\mathfrak{p^c}\mathcal{O})$. Since $a$ is invertible, we can write
\[
g=\left[\begin{array}{ccc}
1_m & 0 & 0\\
ha^{-1} & 1_r & 0\\
ca^{-1} & -\epsilon \hat{a}h^{\ast}\theta^{-1} & 1_m
\end{array}\right]\left[\begin{array}{ccc}
a & 0 & 0\\
0 & e-ha^{-1}f & 0\\
0 & 0 & \hat{a}
\end{array}\right]\left[\begin{array}{ccc}
1_m & a^{-1}f & a^{-1}b\\
0 & 1_r & -\epsilon\theta f^{\ast}\hat{a}\\
0 & 0 & 1_m
\end{array}\right]
\]
and our claim clearly follows. 

By straightforward computations and the fact that
\[
K(\mathfrak{p}^{\mathfrak{c}})\mathfrak{Q}(\mathfrak{p^c})K'(\mathfrak{p^c})=K(\mathfrak{p}^{\mathfrak{c}})\mathfrak{Q}K(\mathfrak{p^c})\cdot K(\mathfrak{p}^{\mathfrak{c}})\mathrm{diag}[\varpi^{-\mathfrak{c}}\cdot 1_m,1_r,\varpi^{\mathfrak{c}}\cdot 1_m]K'(\mathfrak{p^c}),
\]
we have
\[
\begin{aligned}
&\mathcal{Z}(s;\pi(w)\phi,\phi,f^{\dagger,\mathfrak{c}}_s)\\
=&\chi(\varpi)^{\mathfrak{c}m\mathbf{d}_1}q^{-\mathfrak{c}m\mathbf{d}_2(s+\kappa)}\sum_{\xi\in K(\mathfrak{p^c})\backslash\mathfrak{X}/K(\mathfrak{p^c})}\lambda_{\xi}(\phi)(\chi(\nu(\widetilde{\varpi}))|\mathrm{N}_{E/F}(\nu(\widetilde{\varpi}))|^{s+\kappa})^{\mathfrak{d}(\xi)}\\
\times&\int_{K(\mathfrak{p}^{\mathfrak{c}})\mathrm{diag}[\varpi^{-\mathfrak{c}}\cdot 1_m,-1_r,\varpi^{\mathfrak{c}}\cdot 1_m]K'(\mathfrak{p^c})}\left\langle\pi(gw)\phi,\phi\right\rangle dg\\
=&\chi(\varpi)^{\mathfrak{c}m\mathbf{d}_1}q^{-\mathfrak{c}m\mathbf{d}_2(s+\kappa)}L\left(s+\frac{1}{2},\phi\times\chi\right)\left\langle\phi|U'(\mathfrak{p^c}),\phi\right\rangle.
\end{aligned}
\]
\end{proof}

\subsection{The local section $f_s^{\ddagger,\mathfrak{c}}$}
\label{section 4.4}

If $\chi$ is a ramified character then $\mathcal{Z}(s;\phi_1,\phi_2,f_s^{\dagger,\mathfrak{c}})$ will be identically zero. In this subsection, we define a section $f_s^{\ddagger,\mathfrak{c}}$ as a twist of $f_s^{\dagger,0}$ such that $\mathcal{Z}(s;\phi_1,\phi_2,f_s^{\ddagger,\mathfrak{c}})$ is a non-zero constant. Assume $\chi$ has conductor $\mathfrak{p^c}$, we define
\begin{equation}
\label{4.4.1}
\begin{aligned}
f_s^{\ddagger,\mathfrak{c}}(h)&=\sum_{u\in\mathrm{GL}_m(\mathcal{O})/\varpi^{\mathfrak{c}}\mathrm{GL}_m(\mathcal{O})}\chi^{-1}(\nu(u))\\
&\times f_s^{\dagger,0}\left(h\left[\begin{array}{cccccc}
1_r & 0 & 0 & 0 & 0 & 0\\
0 & 1_m & 0 & 0 & 0 & \frac{u}{\varpi^{\mathfrak{c}}}\\
0 & 0 & 1_m & 0 & -\frac{\epsilon u^{\ast}}{\varpi^{\mathfrak{c}}} & 0\\
0 & 0 & 0 & 1_r & 0 & 0\\
0 & 0 & 0 & 0 & 1_m & 0\\
0 & 0 & 0 & 0 & 0 & 1_m
\end{array}\right]\right).
\end{aligned}
\end{equation}

The following lemma shows the reason for the twist.

\begin{lem}
\label{lemma 4.3}
When pulling back along $G\times G\to H$, we have
\begin{equation}
\label{4.4.2}
f_s^{\ddagger,\mathfrak{c}}((g_1k_1,g_2k_2))=\chi(\nu(k_2))f_s^{\ddagger,\mathfrak{c}}((g_1,g_2))
\end{equation}
for $k_1\in K(\mathfrak{p}^{2\mathfrak{c}}), k_2\in K'(\mathfrak{p}^{2\mathfrak{c}})$.
\end{lem}

\begin{proof}
For the notations of $k_1,k_2$ as before, we have
\[
\begin{aligned}
f_s^{\ddagger,\mathfrak{c}}((g_1k_1,g_2k_2))&=\chi(\nu(d_{k_1}))\chi(\nu(a_{k_2}))\sum_{u\in\mathrm{GL}_m(\mathcal{O})/\varpi^{\mathfrak{c}}\mathrm{GL}_m(\mathcal{O})}\chi^{-1}(\nu(u))\\
&\times f_s^{\dagger,0}\left((g_1,g_2)\left[\begin{array}{cccccc}
1_r & 0 & 0 & 0 & 0 & 0\\
0 & 1_m & 0 & 0 & 0 & \frac{u'}{\varpi^{\mathfrak{c}}}\\
0 & 0 & 1_m & 0 & -\frac{\epsilon u'^{\ast}}{\varpi^{\mathfrak{c}}} & 0\\
0 & 0 & 0 & 1_r & 0 & 0\\
0 & 0 & 0 & 0 & 1_m & 0\\
0 & 0 & 0 & 0 & 0 & 1_m
\end{array}\right]\right)
\end{aligned}
\]
with $u'=d_{k_2}ud_{k_1}^{-1}$. Then changing variables $u\to d_{k_2}^{-1}ud_{k_1}$ gives the desired result.
\end{proof}

The following proposition is an analogue of the computations in \cite{BS} and \cite[Proposition 11.16]{SU}. Again the main difficulty is to deal with the group with $r\neq 0$ and the arguments is similar to the proof of Proposition \ref{proposition 4.2}.

\begin{prop}
\label{proposition 4.4}
Assume $\phi\in\pi$ is fixed by $K(\mathfrak{p^c})$, $\chi$ is a character of conductor $\mathfrak{p^c}$ and $2\in\mathcal{O}^{\times},\theta\in\mathrm{GL}_r(\mathcal{O})$. Denote
\begin{equation}
\label{4.4.4}
K_1(\mathfrak{p^c})=G(\mathfrak{o})\cap\left[\begin{array}{ccc}
1+\mathrm{Mat}_{2m}(\mathfrak{p}^{\mathfrak{c}}\mathcal{O}) & \mathrm{Mat}_{2m,r}(\mathfrak{p^c}\mathcal{O}) & \mathrm{Mat}_{2m}(\mathcal{O})\\
\mathrm{Mat}_{r,2m}(\mathfrak{p}^{\mathfrak{c}}\mathcal{O}) &1+\mathrm{Mat}_{r}(\mathfrak{p}\mathcal{O}) & \mathrm{Mat}_{r,2m}(\mathfrak{p^c}\mathcal{O})\\
\mathrm{Mat}_{2m}(\mathfrak{p}^{2\mathfrak{c}}\mathcal{O}) & \mathrm{Mat}_{2m,r}(\mathfrak{p}^{\mathfrak{c}}\mathcal{O}) & 1+\mathrm{Mat}_{2m}(\mathfrak{p}^{\mathfrak{c}}\mathcal{O})
\end{array}\right].
\end{equation}
Then
\begin{equation}
\begin{aligned}
&\mathcal{Z}(s;\phi,\pi(w)\phi,f_s^{\ddagger,\mathfrak{c}})\\
=&\mathrm{vol}(\mathrm{GL}_m(\mathcal{O})/\varpi^{\mathfrak{c}}\mathrm{GL}_m(\mathcal{O}))\cdot\mathrm{vol}(K_1(\mathfrak{p}^{\mathfrak{c}}))\cdot\left\langle\pi\left(\left[\begin{array}{ccc}
0 & 0 & \varpi^{-\mathfrak{c}}\cdot1_m\\
0 & 1_r & 0\\
\varpi^{\mathfrak{c}}\cdot 1_m & 0 & 0
\end{array}\right]\right)\phi,\phi\right\rangle.
\end{aligned}
\end{equation}
\end{prop}

\begin{proof}
Denote $d_u=\mathrm{diag}[1_{m+r},\epsilon\varpi^{\mathfrak{c}}u^{-1},1_{m+r},\epsilon\varpi^{-\mathfrak{c}}u^{\ast}]$,
then
\[
(w,w)^{-1}d_u\delta d_u^{-1}(w,w)=\left[\begin{array}{cccccc}
1_{r} & 0 & 0 & 0 & 0 & 0\\
0 & 1_m & 0 & 0 & 0 & \frac{u}{\varpi^{\mathfrak{c}}}\\
0 & 0 & 1_m & 0 & -\frac{\epsilon u^{\ast}}{\varpi^{\mathfrak{c}}} & 0\\
0 & 0 & 0 & 1_r & 0 & 0\\
0 & 0 & 0 & 0 & 1_m & 0\\
0 & 0 & 0 & 0 & 0 & 1_m
\end{array}\right]
\]
Changing variables $g\mapsto w^{-1}g\left[\begin{array}{ccc}
\epsilon\varpi^{-\mathfrak{c}}u & 0 & 0\\
0 & 1_r & 0\\
0 & 0 & \epsilon\varpi^{\mathfrak{c}}\hat{u}
\end{array}\right]w$, we need to calculate
\[
\begin{aligned}
&\mathcal{Z}(s;\phi,\pi(w)\phi,f_s^{\ddagger,\mathfrak{c}})\\
=&\int_{G(F)}\sum_{u\in\mathrm{GL}_m(\mathcal{O})/\varpi^{\mathfrak{c}}\mathrm{GL}_m(\mathcal{O})}\chi^{-1}(\nu(u))\left\langle\pi\left(g\left[\begin{array}{ccc}
\epsilon\varpi^{-\mathfrak{c}}u & 0 & 0\\
0 & 1_r & 0\\
0 & 0 & \epsilon\varpi^{\mathfrak{c}}\hat{u}
\end{array}\right]w\right)\phi,\phi\right\rangle\\
\times&f_s^{\dagger,0}\left(\delta(g,1)\delta d_u^{-1}(w,w)\right)dg.
\end{aligned}
\]
Write $g=\left[\begin{array}{ccc}
a & f & b\\
h & e & j\\
c & k & d
\end{array}\right]$ with $a,d$ of size $m\times m$, $e$ of size $r\times r$ and compute that
\[
\begin{aligned}
&\delta(g,1)\delta d_u^{-1}(w,w)\\
=&\left[\begin{array}{cccccc}
\frac{e+1}{2} & 0 & \frac{\epsilon \varpi^{\mathfrak{c}}j\hat{u}}{2} & \frac{\epsilon(e-1)\theta}{4} & -\frac{j}{2} & \frac{\varpi^{-\mathfrak{c}}hu}{2}\\
0 & 0 & 0 & 0 & -1_m & 0\\
\epsilon f & 0 & \varpi^{\mathfrak{c}}b\hat{u} & \frac{f}{2} & -\epsilon b & \epsilon\varpi^{-\mathfrak{c}} au\\
\epsilon\theta^{-1}(e-1) & 0 & \varpi^{\mathfrak{c}}\theta^{-1}j\hat{u} & \theta^{-1}\frac{e+1}{2}\theta & -\epsilon\theta^{-1}j & \epsilon\varpi^{-\mathfrak{c}}\theta^{-1}hu\\
-\epsilon f & -\epsilon\cdot 1_m & -\varpi^{\mathfrak{c}}b\hat{u} & -\frac{f\theta}{2} & \epsilon b & -\epsilon\varpi^{-\mathfrak{c}}(a+1)u\\
\epsilon k & 0 & \varpi^{\mathfrak{c}}d\hat{u} & \frac{k\theta}{2} & -\epsilon(d+1) & \epsilon \varpi^{-\mathfrak{c}}cu
\end{array}\right].
\end{aligned}
\]
Suppose it is an element in $P(F)J_nN'(\mathfrak{p}^{0})$, then as in the proof of Proposition \ref{proposition 4.2} it can be written in the form 
\[
\left[\begin{array}{cccc}
\ast & \ast & \ast & \ast\\
\ast & \ast & \ast & \ast\\
\epsilon D_1 & \epsilon D_2 & \epsilon D_1x+D_2y^{\ast} & \epsilon D_1y+\epsilon D_2z\\
\epsilon D_3 & \epsilon D_4 & \epsilon D_3x+D_4y^{\ast} & \epsilon D_3y+\epsilon D_4z
\end{array}\right]
\]
with $x\in S_r(\mathfrak{p}\mathcal{O}),y\in\mathrm{Mat}_{r,2m}(\mathcal{O}),z\in S_{2m}(\mathcal{O})$. We need
\[
D_1=\theta^{-1}(e-1),\,D_2=\left[\begin{array}{cc}
0 & \epsilon\varpi^{\mathfrak{c}}\theta^{-1}j\hat{u}
\end{array}\right],\,D_3=\left[\begin{array}{c}
-f\\
k
\end{array}\right],\,D_4=\left[\begin{array}{cc}
-1_m & -\epsilon \varpi^{\mathfrak{c}}b\hat{u}\\
0 & \epsilon \varpi^{\mathfrak{c}}d\hat{u}
\end{array}\right],
\]
and
\[
\begin{aligned}
\epsilon D_1x+D_2y^{\ast}&=\theta^{-1}\frac{e+1}{2}\theta,&D_1y+D_2z&=\left[\begin{array}{cc}
-\theta^{-1}j & \theta^{-1}\varpi^{-\mathfrak{c}}hu
\end{array}\right],\\
\epsilon D_3x+D_4y^{\ast}&=\left[\begin{array}{c}
-\frac{f\theta}{2}\\
\frac{k\theta}{2}
\end{array}\right],&D_3y+D_4z&=\left[\begin{array}{cc}
b & -\varpi^{-\mathfrak{c}}(a+1)u\\
-(d+1) & \varpi^{-\mathfrak{c}}cu
\end{array}\right].
\end{aligned}
\]
By the same arguments as in the proof of Proposition \ref{proposition 4.2}, the conditions on $x,y,z$ force\\
(1) $d$ is invertible and $1+d^{-1}\in\mathrm{Mat}_m(\mathfrak{p^c}\mathcal{O})$,\\
(2) $d^{-1}c$ have entries in $\mathfrak{p}^{2\mathfrak{c}}$, $d^{-1}k, jd^{-1}$ has entries in $\mathfrak{p^c}$ and $bd^{-1}$ has entries in $\mathcal{O}$,\\
(3) $e-jd^{-1}k\in-1+\mathrm{Mat}_r(\mathfrak{p}\mathcal{O})$.\\
These implies $-g\in K_1(\mathfrak{p^c})$ and we have
\[
\begin{aligned}
&\mathcal{Z}(s;\phi,\pi(w)\phi,f_s^{\ddagger,\mathfrak{c}})\\
=&\int_{K_1(\mathfrak{p^c})}\sum_{u\in\mathrm{GL}_m(\mathcal{O})/\varpi^{\mathfrak{c}}\mathrm{GL}_m(\mathcal{O})}\chi^{-1}(\nu(u))\left\langle\pi\left(g\left[\begin{array}{ccc}
\epsilon\varpi^{-\mathfrak{c}}u & 0 & 0\\
0 & 1_r & 0\\
0 & 0 & \epsilon\varpi^{\mathfrak{c}}\hat{u}
\end{array}\right]w\right)\phi,\phi\right\rangle dg\\
=&\mathrm{vol}(\mathrm{GL}_m(\mathcal{O})/\varpi^{\mathfrak{c}}\mathrm{GL}_m(\mathcal{O}))\cdot\mathrm{vol}(K_1(\mathfrak{p}^{\mathfrak{c}}))\cdot\left\langle\pi\left(\left[\begin{array}{ccc}
0 & 0 & \varpi^{-\mathfrak{c}}\cdot1_m\\
0 & 1_r & 0\\
\varpi^{\mathfrak{c}}\cdot 1_m & 0 & 0
\end{array}\right]\right)\phi,\phi\right\rangle.
\end{aligned}
\]
\end{proof}

\subsection{The $p$-adic section $f_s^p$}
\label{section 4.5}

Assume $\chi$ is unramified, $\phi\in\pi$ is fixed by $K(\mathfrak{q}^2)$ and is an eigenvector for the Hecke algebra $\mathcal{H}(K(\mathfrak{q}^2),\mathfrak{X})$. We construct yet another section $f_s^p$ as a twist of $f_s^{\dagger,0}$ which represent the $p$-adic modification factor in the construction of the $p$-adic $L$-functions. Again, here we are inspired by the idea of \cite[p.1392 and p.1400]{BS}.

For each $0\leq i\leq m$, denote $T_i$ for the Hecke operator given by the double coset
\begin{equation}
\label{4.5.1}
[K(\mathfrak{p})\xi_iK(\mathfrak{p})],\,\xi_i=\mathrm{diag}[u_i,1_r,\hat{u}_i],\,u_i=\left[\begin{array}{cc}
1_{m-i} & 0\\
0 & \widetilde{\varpi}\cdot 1_i
\end{array}\right].
\end{equation}
Suppose there is a double coset decomposition
\begin{equation}
\label{4.5.2}
\mathrm{GL}_m(\mathcal{O})u_i\mathrm{GL}_m(\mathcal{O})=\coprod_j\delta_{ij}\mathrm{GL}_m(\mathcal{O}).
\end{equation}
We define a local section $f_s^{p,i}$ by
\begin{equation}
\label{4.5.3}
\begin{aligned}
f_{s}^{p,i}(h)=\sum_j&\sum_{u\in\widetilde{\varpi}\mathrm{Mat}_m(\mathcal{O})\delta^{-1}_{ij}/\widetilde{\varpi}\mathrm{Mat}_m(\mathcal{O})}\\
&f_s^{\dagger,0}\left(h\left[\begin{array}{cccccc}
1_r & 0 & 0 & 0 & 0 & 0\\
0 & 1_m & 0 & 0 & 0 & \frac{u}{\widetilde{\varpi}}\\
0 & 0 & 1_m & 0 & -\frac{\epsilon u^{\ast}}{\widetilde{\varpi}} & 0\\
0 & 0 & 0 & 1_r & 0 & 0\\
0 & 0 & 0 & 0 & 1_m & 0\\
0 & 0 & 0 & 0 & 0 &1_m
\end{array}\right]\right).
\end{aligned}
\end{equation}

\begin{lem}
\label{lemma 4.5}
Let $\lambda_i$ be the eigenvalues of $\phi$ under $T_i$, i.e. $\phi|T_i=\lambda_i\phi$. Set $\phi_1=\phi,\phi_2=\pi(w)\phi$. Then for each $0\leq i\leq m,$
\begin{equation}
\label{4.5.4}
\begin{aligned}
&\mathcal{Z}(s;\phi,\pi(w)\phi,f_s^{p,i})\\
=&\chi(\widetilde{\varpi})^{\mathbf{d}_3(m-i)}q^{-(m-i)\mathbf{d}_4(s+\kappa)}\lambda_{m-i}L\left(s+\frac{1}{2},\phi\times\chi\right)\\
\times&\left\langle\pi\left(\left[\begin{array}{ccc}
0& 0 & \widetilde{\varpi}^{-1}\cdot 1_m\\
0 & 1_r & 0\\
\widetilde{\varpi}\cdot 1_m & 0 & 0
\end{array}\right]\right)\phi,\phi\right\rangle,
\end{aligned}
\end{equation}
with 
\begin{equation}
\label{4.5.5}
\mathbf{d}_3=\left\{\begin{array}{cc}
1 & \text{ Case I, II, V Ramified,}\\
2 & \text{ Case III, IV, V Inert.}
\end{array}\right.,\qquad \mathbf{d}_4=\left\{\begin{array}{cc}
2 & \text{ Case V Inert,}\\
1 & \text{ otherwise,}
\end{array}\right.
\end{equation}
\end{lem}

\begin{proof}
Denote $d_{\widetilde{\varpi}}=\mathrm{diag}[1_{m+r},\epsilon\widetilde{\varpi},1_{m+r},\widetilde{\varpi}^{-1}]$ and
\[
\delta_u=\left[\begin{array}{cccccc}
1_r & 0 & 0 & 0 & 0 & 0\\
0 & 1_m & 0 & 0 & 0 & 0\\
0 & 0 & 1_m & 0 & 0 & 0\\
0 & 0 & 0 & 1_r & 0 & 0\\
0 & 0 & -u & 0 & 1_m & 0\\
0 & \epsilon u^{\ast} & 0 & 0 & 0 & 1_m
\end{array}\right].
\]
Then
\[
(w,w)^{-1}d_{\widetilde{\varpi}}\delta_ud_{\widetilde{\varpi}}^{-1}(w,w)=\left[\begin{array}{cccccc}
1_r & 0 & 0 & 0 & 0 & 0\\
0 & 1_m & 0 & 0 & 0 & \frac{u}{\widetilde{\varpi}}\\
0 & 0 & 1_m & 0 & -\frac{\epsilon u^{\ast}}{\widetilde{\varpi}} & 0\\
0 & 0 & 0 & 1_r & 0 & 0\\
0 & 0 & 0 & 0 & 1_m & 0\\
0 & 0 & 0 & 0 & 0 & 1_m
\end{array}\right].
\]
Changing variables $g\mapsto w^{-1}g\left[\begin{array}{ccc}
\epsilon\widetilde{\varpi}^{-1}\cdot 1_m & 0 & 0\\
0 & 1_r & 0\\
0 & 0 & \epsilon\widetilde{\varpi}\cdot 1_m
\end{array}\right]w$, we need to calculate
\[
\begin{aligned}
&\mathcal{Z}(s;\phi_1,\phi_2,f_s^{p})\\
=&\int_{G(F)}\sum_j\sum_{u\in\widetilde{\varpi}\mathrm{Mat}_m(\mathcal{O})\delta^{-1}_{ij}/\widetilde{\varpi}\mathrm{Mat}_m(\mathcal{O})}\left\langle\pi\left(g\left[\begin{array}{ccc}
\epsilon\widetilde{\varpi}^{-1}\cdot 1_m& 0 & 0\\
0 & 1_r & 0\\
0 & 0 & \epsilon\widetilde{\varpi}\cdot 1_m
\end{array}\right]w\right)\phi,\phi\right\rangle\\
\times&f_s^{\dagger,0}\left(\delta(g,1)\delta_u d_{\widetilde{\varpi}}^{-1}(w,w)\right)dg.
\end{aligned}
\]
Write $g=\left[\begin{array}{ccc}
a & f & b\\
h & e & j\\
c & k & d
\end{array}\right]$ and compute that
\[
\begin{aligned}
&\delta(g,1)\delta_ud_{\widetilde{\varpi}}^{-1}(w,w)\left[\begin{array}{cccc}
1_r & 0 & 0 & 0\\
0 & \widetilde{\varpi}^{-1}\cdot 1_{2m} & 0 & 0\\
0 & 0 & 1_r & 0\\
0 & 0 & 0 &\widetilde{\varpi}\cdot 1_{2m}
\end{array}\right]\\
=&\left[\begin{array}{cccccc}
\frac{e+1}{2} & 0 & \frac{\widetilde{\varpi}\epsilon j}{2} & \frac{\epsilon(e-1)\theta}{4} & -\frac{ j}{2} & \frac{\widetilde{\varpi}^{-1}h}{2}\\
0 & 0 & 0 & 0 & -1_m & 0\\
\epsilon f & 0 & \widetilde{\varpi}b & \frac{f}{2} & -\epsilon b & \epsilon \widetilde{\varpi}^{-1}a\\
\epsilon\theta^{-1}(e-1) & 0 & \theta^{-1}\widetilde{\varpi}j & \theta^{-1}\frac{e+1}{2}\theta & -\epsilon\theta^{-1}j & \epsilon\widetilde{\varpi}^{-1}\theta^{-1}h\\
-\epsilon f & -\epsilon\cdot 1_m & -\widetilde{\varpi}b & -\frac{f\theta}{2} & \epsilon b & -\epsilon\widetilde{\varpi}^{-1}(a+u)\\
\epsilon k & 0 &\widetilde{\varpi}d & \frac{k\theta}{2} & -\epsilon(du^{\ast}+1) & \epsilon \widetilde{\varpi}^{-1}c
\end{array}\right].
\end{aligned}
\]
By the same arguments as in the proof of Proposition \ref{proposition 4.2}, \ref{proposition 4.4}, this is an element in $P(F)J_nN'(\mathfrak{p}^0)$ if and only if\\
(1) $d$ is invertible and $d^{-1}+u^{\ast}\in\mathrm{Mat}_m(\mathfrak{q}\mathcal{O})$,\\
(2) $d^{-1}c$ have entries in $\mathfrak{q}^{2}$, $d^{-1}k, jd^{-1}$ has entries in $\mathfrak{q}$ and $bd^{-1}$ has entries in $\mathcal{O}$,\\
(3) $e-jd^{-1}k\in-1+\mathrm{Mat}_r(\mathfrak{q}\mathcal{O})$.\\

Since $d$ is invertible, we can write
\[
g=\left[\begin{array}{ccc}
1_m & -\hat{d}j^{\ast}\theta^{-1} & bd^{-1}\\
0 & 1_r & jd^{-1}\\
0 & 0 & 1_m
\end{array}\right]\left[\begin{array}{ccc}
\hat{d} & 0 & 0\\
0 & e-jd^{-1}k & 0\\
0 & 0 & d
\end{array}\right]\left[\begin{array}{ccc}
1_m & 0 & 0\\
-\theta k^{\ast}\hat{d} & 1_r & 0\\
d^{-1}c & d^{-1}k & 1_m
\end{array}\right].
\]
Note that there is a permutation $j\mapsto j'$ such that $\hat{d}$ runs through $\mathrm{GL}_m(D)\cap\widetilde{\varpi}\mathrm{Mat}_m(\mathcal{O})\delta^{-1}_{ij'}$ for fixed $j$. Hence, when $\delta_{ij}$ running through the right coset $\mathrm{GL}_m(\mathcal{O})u_i\mathrm{GL}_m(\mathcal{O})/\mathrm{GL}_m(\mathcal{O})$, all such $d$ run through $\mathrm{GL}_m(D)\cap\widetilde{\varpi}\mathrm{Mat}_m(\mathcal{O})u^{-1}_i$. Therefore,
\[
\begin{aligned}
&\mathcal{Z}(s;\phi,\pi(w)\phi,f_s^{p,i})\\
=&
\sum_{\hat{d}\in\mathrm{GL}_m(D)\cap\widetilde{\varpi}\mathrm{Mat}_m(\mathcal{O})u^{-1}_i}\sum_{g\in K(\mathfrak{q}^2)\mathrm{diag}[\hat{d},1_r,d]K(\mathfrak{q}^2)}\chi(\nu(\widetilde{\varpi}d))|\mathrm{N}_{E/F}(\nu(\widetilde{\varpi} d))|^{s+\kappa}\\
\times&\left\langle\pi\left(g\left[\begin{array}{ccc}
\epsilon\widetilde{\varpi}^{-1}\cdot 1_m& 0 & 0\\
0 & 1_r & 0\\
0 & 0 & \epsilon\widetilde{\varpi}\cdot 1_m
\end{array}\right]w\right)\phi,\phi\right\rangle.
\end{aligned}
\]

Note that when $\hat{d}$ runs through $\mathrm{GL}_m(D)\cap\varpi\mathrm{Mat}_m(\mathcal{O})u^{-1}_i$, we are taking a sum over
\[
K(\mathfrak{q}^2)\mathfrak{Q}K(\mathfrak{q}^2)\cdot K(\mathfrak{q}^2)\mathrm{diag}[\widetilde{\varpi}^{-1} u_i,1_r,\widetilde{\varpi} u_i^{-1}]K(\mathfrak{q}^2).
\]
We thus obtain
\[
\begin{aligned}
&\mathcal{Z}(s;\phi,\pi(w)\phi,f_s^{p,i})\\
=&\chi(\widetilde{\varpi})^{\mathbf{d}_3(m-i)}q^{-(m-i)\mathbf{d}_4(s+\kappa)}\lambda_{m-i}L\left(s+\frac{1}{2},\phi\times\chi\right)\\
\times&\left\langle\pi\left(\left[\begin{array}{ccc}
0& 0 & \widetilde{\varpi}^{-1}\cdot 1_m\\
0 & 1_r & 0\\
\widetilde{\varpi}\cdot 1_m & 0 & 0
\end{array}\right]\right)\phi,\phi\right\rangle.
\end{aligned}
\]
as desired.
\end{proof}

Gluing all these $0\leq i\leq m$ together, we define the local section $f_s^p$ by
\begin{equation}
\label{4.5.6}
f_s^p(h)=\sum_{i=0}^m(-1)^iq^{\mathbf{d}_3\left(\frac{i(i-1)}{2}-im\right)}f_s^{p,i}(h).
\end{equation}

\begin{prop}
\label{proposition 4.6}
Assume $\chi$ is unramified, $\phi\in\pi$ is fixed by $K(\mathfrak{q}^2)$ and is an eigenvector for the Hecke algebra $\mathcal{H}(K(\mathfrak{q}^2),\mathfrak{X})$. Assume $2\in\mathcal{O}^{\times},\theta\in\mathrm{GL}_r(\mathcal{O})$. Set $\phi_1=\phi,\phi_2=\pi(w)\phi$ and denote $\beta_i$ for the Satake parameters of $\phi$. Then
\begin{equation}
\label{4.5.7}
\begin{aligned}
\mathcal{Z}(s;\phi,\pi(w)\phi,f_s^p)=&(-1)^mq^{-\mathbf{d}_3\frac{m^2+m}{2}}L\left(s+\frac{1}{2},\phi\times\chi\right)M\left(s+\frac{1}{2},\phi\times\chi\right)\\
\times&\left\langle\pi\left(\left[\begin{array}{ccc}
0& 0 & \widetilde{\varpi}^{-1}\cdot 1_m\\
0 & 1_r & 0\\
\widetilde{\varpi}\cdot 1_m & 0 & 0
\end{array}\right]\right)\phi,\phi\right\rangle,
\end{aligned}
\end{equation}
where $M(s,\phi\times\chi)$ is the modification factors given in the following list.\\
(Case I, Orthogonal)
\[
\prod_{i=1}^m\left(1-\chi(\varpi)\beta_iq^{\frac{r}{2}-s}\right),
\]
(Case II, Symplectic)
\[
\prod_{i=1}^m\left(1-\chi(\varpi)\beta_iq^{-s+1}\right),
\]
(Case III, Quaternionic Orthogonal Nonsplit)
\[
\prod_{i=1}^m\left(1-\chi(\varpi)\beta_iq^{1+r-s}\right),
\]
(Case IV, Quaternionic Unitary Nonsplit)
\[
\prod_{i=1}^m\left(1-\chi(\varpi)\beta_iq^{r-s}\right),
\]
(Case V, Unitary Inert) $E/F$ is inert,
\[
\prod_{i=1}^m\left(1-\chi(\varpi)\beta_i q^{1+r-s}\right),
\]
(Case V, Unitary Ramified) $E/F$ is ramified,
\[
\prod_{i=1}^m\left(1-\chi(\widetilde{\varpi})\beta_iq^{\frac{r+1}{2}-s}\right).
\]
\end{prop}

\begin{proof}
By the above lemma, we have
\[
\begin{aligned}
\mathcal{Z}(s;\phi_1,\phi_2,f_s^{p})&=\sum_{i=1}^m(-1)^iq^{\mathbf{d}_3\left(\frac{i(i-1)}{2}-im\right)}\chi(\widetilde{\varpi})^{\mathbf{d}_3(m-i)}q^{-(m-i)(s+\kappa)}\lambda_{m-i}\\
&\times L\left(s+\frac{1}{2},\phi\times\chi\right)\left\langle\pi\left(\left[\begin{array}{ccc}
0& 0 &\widetilde{\varpi}^{-1}\cdot 1_m\\
0 & 1_r & 0\\
\widetilde{\varpi}\cdot 1_m & 0 & 0
\end{array}\right]\right)\phi,\phi\right\rangle.
\end{aligned}
\]
It suffices to compute
\[
\begin{aligned}
&\sum_{i=1}^m(-1)^iq^{\mathbf{d}_3\left(\frac{i(i-1)}{2}-im\right)}\chi(\widetilde{\varpi})^{\mathbf{d}_3(m-i)}q^{-(m-i)(s+\kappa)}\lambda_{m-i}\\
=&(-1)^mq^{-\mathbf{d}_3\frac{m^2+m}{2}}\sum_{i=1}^m(-1)^iq^{\mathbf{d}_3\frac{i(i-1)}{2}}\chi(\widetilde{\varpi})^{\mathbf{d}_3i}q^{-i(s+\kappa-\mathbf{d}_3)}\lambda_i.
\end{aligned}
\]
This equals to $M(s+\frac{1}{2},\phi\times\chi)$ in the above lists. Indeed, using \cite[Lemma 19.13]{Sh00} and the explicit description of the Satake map in the proof of Proposition \ref{proposition 3.3}, one can show that
\[
\sum_{i=0}^m(-1)^iq^{\mathbf{d}_3\frac{i(i-1)}{2}}\lambda_i\chi(\widetilde{\varpi})^{\mathbf{d}_3i}q^{-i(s+\kappa-\frac{1}{2})}
\]
is the Euler factor in Proposition \ref{proposition 3.3} and the proposition easily follows.
\end{proof}

\section{The Archimedean Theory}
\label{section 5}

In the rest of the paper, we restrict ourselves to the following global setting as in Section \ref{section 2.2}. Let $F$ be a totally real field of degree $[F:\Q]=d$ and consider tuples $(D,\rho,\epsilon)$ of following four cases: 

\begin{tabular}{ll}
(Case II, Symplectic) & $(D,\rho)$ of type (a) with $\epsilon=-1$,\\
(Case III, Quaternionic Orthogonal) & $(D,\rho)$ of type (b) with $\epsilon=1$ and\\
& $D_v=\mathrm{Mat}_2(\R)$ for any archimedean places $v$,\\
(Case IV, Quaternionic Unitary) & $(D,\rho)$ of type (b) with $\epsilon=-1$,\\
& $D_v=\mathbb{H}$ for any archimedean places $v$,\\
(Case V, Unitary) & $(D,\rho)$ of type (c) with $\epsilon=-1$,\\
& $D=E$ is an imaginary quadratic extension of $F$.
\end{tabular}

Here $\mathbb{H}$ is the Hamilton quaternion algebra for which we fix an embedding into $\mathrm{Mat}_2(\R)$. The global group $G$ is defined as
\[
G:=G(F):=\{g\in\mathrm{GL}_n(D):g\Phi g^{\ast}=\Phi\},\Phi=\left[\begin{array}{ccc}
0 & 0 & 1_m\\
0 & \theta & 0\\
\epsilon\cdot 1_m & 0 & 0
\end{array}\right],
\]
with $n=2m+r$ and $\theta^{\ast}=\epsilon\theta\in\mathrm{GL}_r(D)$ is anisotropic (over $F$). We may also write it as $G_{m,r}$ to emphasize the index. In Case V, we assume $i\theta_v>0$ for all archimedean place $v$ for simplicity.

\begin{rem}
\label{remark 5.1}
Clearly in Case II, we always have $r=0$. Using the well known Hasse principle for quadratic forms or the one for quaternionic skew-hermitian forms proved in \cite{Hi63}, together with the local computations of quadratic forms in \cite{Sh04} and quaternionic skew-hermitian forms in \cite{Ts61} one can show that $r\leq 3$ in Case III and $r\leq 1$ in Case IV. This is especially because the Witt index of $G(F_v)$ for any archimedean place $v$ is less than $1$ in Case III and is $0$ in Case IV while, in Case V, $r$ is unbounded since the Witt index of $G(F_v)$ for any archimedean place $v$ is unbounded. For example, $i\cdot 1_r$ is always an anisotropic matrix in $G(F_v)$ in Case V for any $r$.
\end{rem}

The main reason for our above global assumption is that the symmetric space associated to $G(F_v)$ at any archimedean place $v$ is hermitian and admits a Shimura variety. In this case, the algebraic modular forms are well studied. The reader can refer to \cite{Helgason, Satake, PS69} for more details on hermitian symmetric spaces and \cite{LKW, Milne} for general theory of Shimura varieties. We consider the local archimedean groups and compute the local archimedean integrals in Section \ref{section 5.1}, \ref{section 5.2} and in Section \ref{section 5.3} we summarize the definition for algebraic modular forms in the adelic setting.

\subsection{Modular forms on $G(F_{\infty})$}
\label{section 5.1}

Let $v$ be an archimedean place of $F$ and $G_v=G(F_v)$ the localization of $G$ at $v$. Fix a maximal compact subgroup $K$ of $G_v$. Then $G_v/K$ is a hermitian symmetric space whose realizations can be found in \cite[Section 6,7]{Sh97} and \cite[Section 3,5]{Sh00} for symplectic and unitary cases; in \cite{JYB1} for quaternionic unitary case. Especially, in \cite[Section 2.2]{JYB1} we explain how one can construct certain realizations and the quaternionic orthogonal case can be treated in a same way. 

We fix an unbounded realization $\mathfrak{Z}_{m,r}$ as in above references with action of $G_v$ on $\mathfrak{Z}_{m,r}$ denoted by $g.z$ and $j(g,z)$ the automorphy factor for $g\in G_v,z\in\mathfrak{Z}_{m,r}$. Let $z_0$ be the central point of $\mathfrak{Z}_{m,r}$ so that $K=\{g\in G_v:g.z_0=z_0\}$. When $r=0$, the symmetric space $\mathcal{H}_m:=\mathfrak{Z}_{m,0}$ can be simply expressed as
\begin{equation}
\label{5.1.1}
\mathcal{H}_m:=\{z=x+iy\in S_{m,v}\otimes\C:x\in S_{m,v},y\in S_{m,v}^+\}
\end{equation}
where $S_m$ is an additive algebraic group with
\begin{equation}
\label{5.1.2}
S_m(F)=\{\beta\in\mathrm{GL}_m(D):\beta^{\ast}=-\epsilon\beta\}
\end{equation}
and $S_{m,v}^+$ the subgroup of $S_{m,v}$ containing positive definite matrices. The central point is chosen as $z_0=i\cdot 1_m$. The action of $G_{m,0}$ on $\mathcal{H}_m$ is defined as
\begin{equation}
\label{5.1.3}
g.z=(az+b)(cz+d)^{-1}\text{ for }z\in\mathcal{H}_m,g=\left[\begin{array}{cc}
a & b\\
c & d
\end{array}\right]\in G_{m,0}
\end{equation}
and $j(g,z)=\nu(cz+d)$. Recall that for archimedean places $v$, $D_v=\mathrm{Mat}_2(\R)$ in Case III and $D_v=\mathbb{H}$ can be embedded into $\mathrm{Mat}_2(\R)$ in Case IV. Then above space $\mathcal{H}_m$ can be embedded into $\mathrm{Mat}_{2m}(\C)$ and in particular one can show that the space $\mathcal{H}_m$ in Case III is same as the one for Case II with index $2m$.

For $\gamma\in G(F)$, we naturally view it as an element of $G(\mathbb{A})$ and its action on $\mathfrak{Z}_{m,r}^d$ is given by $\gamma.z:=(\gamma_v.z_v)_{v|\infty}$ for $z=(z_v)_{v|\infty}\in\mathfrak{Z}_{m,r}^d$. Fix a weight $\boldsymbol{l}=(l_v)_{v|\infty}$ with $l_v\in\mathbb{N}$ and denote
\begin{equation}
j(\gamma,z)^{\boldsymbol{l}}=\prod_{v|\infty}j(\gamma_v,z_v)^{l_v}.
\end{equation}

\begin{defn}
\label{definition 5.2}
A holomorphic function $\varphi:\mathfrak{Z}^d_{m,r}\to\C$ is called a modular form for a congruence subgroup $\Gamma\subset G(F)$ and weight $\boldsymbol{l}$ if for all $\gamma\in\Gamma$,
\begin{equation}
\label{5.1.4}
\varphi(\gamma.z)=j(\gamma,z)^{\boldsymbol{l}}\varphi(z),\qquad z=(z_v)_{v|\infty}.
\end{equation}
\end{defn}

\begin{rem}
\label{remark 5.3}
$\text{ }$\\
(1) In this paper, we use the term `modular form' as an analogue for the modular forms of $\mathrm{GL}_2$ so in particular we only consider the holomorphic functions.\\
(2) When $m\leq 1$ we need further assume $\varphi$ satisfies the cusp condition which is not necessary for $m\geq 2$ due the the Koecher principle \cite[Lemma 1.5]{Krieg} and \cite[Proposition 5.7]{Sh00}.\\
(3) We are restrict ourselves to certain scalar weight modular forms. Also in unitary case, with notations in \cite[Section 5]{Sh00}, there are indeed two automorphy factors $\lambda(g,z),\mu(g,z)$ and
\[
j(g,z)^l=\nu(\mu(g,z))^{l_1}\nu(\lambda(g,z))^{l_2},\text{ with }l=(l_1,l_2).
\]
Here for simplicity we only consider $l_1=l,l_2=0$ to make our discussions consistent in all cases.
\end{rem}

Denote $F_{\infty}=F\otimes_{\Q}\R\cong\R^d$. We rephrase above definitions for functions $\phi:G(F_{\infty})\to\C$. Set 
\begin{equation}
\phi(g)=j(g_v,z_0)^{-\boldsymbol{l}}\varphi((g_v.z_0)_{v|\infty}).
\end{equation}
Then clearly $\phi(gk)=\prod_{v|\infty}j(k_v,z_0)^{-l_v}\phi(g)$ for $k\in K$. We call a function $\phi:G(F_{\infty})\to\C$ a modular form of weight $\boldsymbol{l}$ if it is obtained from some weight $\boldsymbol{l}$ modular form $\varphi$ through this way. We call $\phi$ a cusp form if
\begin{equation}
\int_{U(\R)}\phi(ug)du=0
\end{equation}
for every unipotent radical $U$ of all proper parabolic subgroup of $G$.

\subsection{The archimedean integrals}
\label{section 5.2}

Recall that
\begin{equation}
\label{5.2.1}
H:=H(F):=\{h\in\mathrm{GL}_{2n}(D):hJ_n h^{\ast}=J_n\},
\end{equation}
with a doubling embedding $G\times G\to H$ defined in \eqref{2.1.7}. For an archimedean place $v$ of $F$, denote $H_v=H(F_v)$ for the localization at $v$ and $\mathcal{H}_n:=\mathfrak{Z}_{n,0}$ the symmetric space associated to $H_v$. We also write $J(h,z)$ for the automorphy factor of $H$ to distinguish the one for $G$. There is a doubling embedding (see for example \cite[Section 6,7]{Sh97} and \cite[Section 2.3]{JYB1}) 
\begin{equation}
\label{5.2.2}
\begin{aligned}
\mathfrak{Z}_{m,r}\times\mathfrak{Z}_{m,r}&\to\mathcal{H}_n,\\
z_1,z_2&\mapsto[z_1,z_2],
\end{aligned}
\end{equation}
compatible with the action, i.e. $(g_1,g_2)\cdot[z_0,z_0]=[g_1z_0,g_2z_0]$. In particular, we can normalize $z_0$ and the embedding such that $[z_0,z_0]=i\cdot 1_n\in\mathcal{H}_n$. We simply write $i:=i\cdot 1_n$ if it is clear from the context.

Let $\phi$ be a cusp form of weight $\boldsymbol{l}$ and set $\phi_1=\phi_2=\phi$. Assume $\chi$ is a Hecke character of infinity type $\boldsymbol{l}$. That is $\chi_v(x)=x^{l_v}|x|^{-l_v}$ for any $v|\infty$. Define a section $f_s^{\infty}\in\mathrm{Ind}_{P(F_{\infty})}^{H(F_{\infty})}(\chi|\cdot|^s)$ by $f_s^{\infty}=\prod_{v|\infty}f_{s,v}^{\infty}$ with
\begin{equation}
\label{5.2.3}
\begin{aligned}
f_{s,v}^{\infty}(h)=J(h_v,i)^{-l_v}|J(h_v,i)|^{l_v-s-\kappa} &\qquad \text{ Case II, III, IV,}\\
f_{s,v}^{\infty}(h)=J(h_v,i)^{-l_v}|J(h_v,i)|^{l_v-2s-2\kappa} &\qquad \text{ Case  V}.
\end{aligned}
\end{equation}
and consider the archimedean integral
\begin{equation}
\label{5.2.4}
\mathcal{Z}(s;\phi_1,\phi_2,f_s^{\infty})=\int_{G(F_{\infty})}f_s(\delta(g,1))\langle\pi(g)\phi_1,\phi_2\rangle dg.
\end{equation}

\begin{prop}
\label{proposition 5.4}
Let $\phi$ be a cusp form of weight $\boldsymbol{l}$ and $\mathrm{Re}(s)+l_v>\kappa$ for all $v$. we have
\begin{equation}
\label{5.2.5}
\mathcal{Z}(s;\phi,\phi,f_s^{\infty})=C(s)\cdot \prod_{v|\infty}c_{l_v}(s)\cdot\langle\phi,\pi(w)\phi\rangle.
\end{equation}
Here $w$ is the Weyl element as in \eqref{2.3.4}, $C(s)$ is a constant of power in $2$ depending on $s$ and $c_{l_v}(s)$ is given by the following list:\\
(Case II)
\[
\pi^{\frac{m(m+1)}{2}}\prod_{i=0}^{m-1}\frac{\Gamma\left(\frac{1}{2}\left(s+l_v-\frac{1}{2}\right)-\frac{i}{2}\right)}{\Gamma\left(\frac{1}{2}\left(s+l_v+\frac{2m+1}{2}\right)-\frac{i}{2}\right)},
\]
(Case III)
\[
\pi^{\frac{n(n+1)}{2}}\prod_{i=0}^{n-1}\frac{\Gamma\left(\frac{1}{2}\left(s+l_v-\frac{1}{2}\right)-\frac{i}{2}\right)}{\Gamma\left(\frac{1}{2}\left(s+l_v+\frac{2n+1}{2}\right)-\frac{i}{2}\right)},
\]
(Case IV)
\[
\pi^{\frac{n(n-1)}{2}}\prod_{i=0}^{\left\lfloor \frac{n}{2}\right\rfloor-1}\frac{\Gamma\left(s+l_v+\frac{1}{2}-2i\right)}{\Gamma\left(s+l_v+\frac{2n-1}{2}-2i\right)},
\]
(Case V)
\[
\pi^{m(m+r)}\prod_{i=0}^{m-1}\frac{\Gamma\left(s+\frac{l_v}{2}-i\right)}{\Gamma\left(s+\frac{l_v}{2}+\frac{n}{2}-i\right)}
\]
\end{prop}

\begin{proof}
This is well known (see \cite{Sh00} for symplectic and unitary case, \cite{JYB1} for quaternionic unitary case). Indeed, it suffices to calculate
\[
\begin{aligned}
&\int_{G(F_{\infty})}f_s^{\infty}(\delta(g,1))\overline{\phi(g'g)}dg.
\end{aligned}
\]
Note that $J(\delta(g,1),[z_0,z_0])=j(\delta,[gz_0,z_0])j(g,z_0)$ and rewrite above integral (for Case II, III, IV) as
\[
\int_{G(F_{\infty})}j(\delta,[gz_0,z_0])^{-\boldsymbol{l}}|j(\delta,[gz_0,z_0])|^{\boldsymbol{l}-s-\kappa}|j(g,z_0)|^{-\boldsymbol{l}-s-\kappa}\overline{\phi(g'g)}dg.
\]
This kind of integral is calculated in \cite[Appendix A.2]{Sh97} for symplectic and unitary case. The symmetric space for quaternionic orthogonal group is isomorphic to the one for symplectic group. All these cases including the quaternionic unitary group are treated in \cite[Theorem 2.2.1, 2.3.1, 2.4.1]{Hua}.
\end{proof}

\subsection{Algebraic modular forms}
\label{section 5.3}

For later use, we summarize the definition and properties of algebraic modular forms here. For holomorphic modular forms considered in this paper, there are mainly four characterizations of algebraic modular forms:\\
(1) In \cite{H85, H86, M90} automorphic forms are interpreted as sections of certain automorphic vector bundles. The canonical model of automorphic vector bundles then defines a subspace of algebraic automorphic forms.\\
(2) In \cite{JYB1, G77, G84, Sh00} algebraic modular forms are defined via CM points.\\
(3) In \cite{G81, G83, G84}, a characterization using Fourier-Jacobi expansion is given. In particular, in the special case $r=0$ the modular forms have a Fourier expansion and the algebraic modular form is defined to be the one have algebraic Fourier coefficients. This generalizes the classical definition of algebraic modular forms of $\mathrm{GL}_2$.\\
(4) In \cite{G84}, there is yet another characterization using the pullback to classical modular forms over $\mathrm{GL}_2$. Moreover, three definitions (2,3,4) are also proved to be equivalent there.

We first review the adelic definition of modular forms. Denote $K_{\infty}$ be the maximal compact subgroup of $G(F_{\infty})$ and $K$ be any open compact subgroup of $\prod_{v\nmid\infty}G(F_v)$. Fix a weight $\boldsymbol{l}=(l_v)_{v|\infty}$ with $l_v\in\mathbb{N}$ as before.

Recall the following weak approximation of $G$
\begin{equation}
\label{weakapproximation}
G(\mathbb{A})=\coprod_{i}G(F)t_iKG(F_{\infty}).
\end{equation}
For a function $\boldsymbol{f}:G(\mathbb{A})\to\C$, we can associate a series of functions $\phi_i$ on $G(F_{\infty})$ for each $i$ defined by
\begin{equation}
\phi_i(g_{\infty})=\boldsymbol{f}(t_ig_{\infty})\qquad g_{\infty}\in G(F_{\infty}).
\end{equation}

\begin{defn}
\label{definition 5.5}
 The space of weight $\boldsymbol{l}$ and level $K$ modular forms $\mathcal{M}_{\boldsymbol{l}}(K)$ contain functions $\boldsymbol{f}:G(\mathbb{A})\to\C$ satisfying:\\
(1) $\boldsymbol{f}$ is left invariant under $G(F)$ and right invariant under $K$, i.e. 
\begin{equation}
\label{5.3.1}
\boldsymbol{f}(\gamma gk)=\boldsymbol{f}(g)\text{ for }\gamma\in G(F),k\in K,
\end{equation}
(2) The functions $\phi_i$ associated to $\boldsymbol{f}$ defined as above are weight $\boldsymbol{l}$ defined as in Section \ref{section 5.1}. Especially,
\begin{equation}
\label{5.3.2}
\boldsymbol{f}(gk_{\infty})=\prod_{v|\infty}j(k_v,z_0)^{l_v}\boldsymbol{f}(g)\text{ for }k_{\infty}=(k_v)_{v|\infty}\in K_{\infty}.
\end{equation}
Furthermore, the subspace $\mathcal{S}_{\boldsymbol{l}}(K)$ of cusp forms consisting functions $\boldsymbol{f}\in\mathcal{M}_{\boldsymbol{l}}(K)$ satisfying
 \begin{equation}
\int_{U(\Q)\backslash U(\mathbb{A})}\boldsymbol{f}(ug)du=0,
\end{equation}
for all unipotent radicals $U$ of all proper parabolic subgroups of $G$. The subspace of cusp forms is denoted as $\mathcal{S}_{\boldsymbol{l}}(K)$. We may write $\mathcal{M}^{m,r}_{\boldsymbol{l}}(K)$ and $\mathcal{S}^{m,r}_{\boldsymbol{l}}(K)$ if we want to emphasize the index $m,r$.
\end{defn}

For two modular forms $\boldsymbol{f}_1,\boldsymbol{f}_2\in\mathcal{M}_{\boldsymbol{l}}(K)$, we define the Petersson inner product
\begin{equation}
\langle\boldsymbol{f}_1,\boldsymbol{f}_2\rangle=\int_{G(F)\backslash G(\mathbb{A})/KK_{\infty}}\boldsymbol{f}_1(g)\overline{\boldsymbol{f}_2(g)}dg,
\end{equation}
whenever the integral converges. For example, this is well defined when one of $\boldsymbol{f}_1$ is a cusp form. Here, by convention, the measure $\mathbf{d}g$ is normalized such that the volume of $K$ is $1$.

We define the algebraic modular forms via CM points as \cite{JYB1, G77, G84, Sh00}. That is a modular form $\varphi:\mathfrak{Z}_{m,r}\to\C$ is called algebraic if for all CM points $z\in\mathfrak{Z}_{m,r}$, $\varphi(z)$ is algebraic up to certain period $\mathcal{P}(z)$ (defined in above references) depending only on $z$ and the weight $\boldsymbol{l}$. To rephrase it in adelic language, we call $g\in G(\mathbb{A})$ a CM point if $g_v\cdot z_0\in\mathfrak{Z}_{m,r}$ is a CM point for any archimedean place $v$ of $F$ and set $\mathcal{P}(g)=\prod_{v|\infty}\mathcal{P}(g_v.z_0)$.

\begin{defn}
\label{definition 5.6}
The subspace 
\begin{equation}
\label{5.3.3}
\mathcal{M}_{\boldsymbol{l}}(K,\overline{\Q})\subset\mathcal{M}_{\boldsymbol{l}}(K),\text{ resp. }\mathcal{S}_{\boldsymbol{l}}(K,\overline{\Q})\subset\mathcal{S}_{\boldsymbol{l}}(K),
\end{equation}
of algebraic modular forms (resp. algebraic cusp forms) consisting functions $\boldsymbol{f}:G(\mathbb{A})\to\C$ such that $\boldsymbol{f}(g)\in\mathcal{P}(g)\overline{\Q}$ for any CM points $g$.
\end{defn}

The properties we need for algebraic modular forms are collected in the following proposition.

\begin{prop}
\label{proposition 5.7}
$\text{ }$\\
(1) There is a basis of $\mathcal{M}_{\boldsymbol{l}}(K)$ consisting of algebraic modular forms. That is $\mathcal{M}_{\boldsymbol{l}}(K)=\mathcal{M}_{\boldsymbol{l}}(K,\overline{\Q})\otimes_{\overline{\Q}}\C$.\\
(2) Let $\sigma\in\mathrm{Aut}(\C/\overline{\Q})$ act on $\mathcal{M}_{\boldsymbol{l}}(K)$ by acting on $\C$. Then $\mathcal{S}_{\boldsymbol{l}}(K,\overline{\Q})^{\sigma}=\mathcal{S}_{\boldsymbol{l}}(K,\overline{\Q})$ and $\mathcal{S}_{\boldsymbol{l}}(K)=\mathcal{S}_{\boldsymbol{l}}(K,\overline{\Q})\otimes_{\overline{\Q}}\C$.
\end{prop}

\begin{rem}
\label{remark 5.8}
When $r=0$, the definition of algebraicity has a simple characterization using Fourier expansions. Let $e_{\mathbb{A}}=\prod_ve_v$ be the standard additive character of $\mathbb{A}$. That is $e_v(x)=e^{2\pi ix}$ for archimedean places $v$ and $e_v(\varpi_v)=\prod_{v|\infty}e_v(-q_v^{-1})$ with $\varpi_v$ the uniformizer of $F_v$ and $|\varpi_v|_v=q_v^{-1}$. Also set $e_{\infty}=\prod_{v|\infty}e_v$. Let $\boldsymbol{f}:G(\mathbb{A})\to\C$ be a  modular form in $\mathcal{M}_{\boldsymbol{l}}(K)$. Then for all $y\in\mathrm{GL}_n(\mathbb{A}_{D})$ and $x\in S_n(\mathbb{A})$ we have a Fourier expansion of the form
\begin{equation}
\label{adelicfourier}
\boldsymbol{f}\left(\left[\begin{array}{cc}
y & x\hat{y}\\
0 & \hat{y}
\end{array}\right]\right)=\prod_{v|\infty}\nu(y_v^{\ast})^{l_v}\cdot\sum_{\beta\in S_m(F)}\mathbf{c}(\beta;\boldsymbol{f},y)e_{\infty}(\tau(iy^{\ast}\beta y))e_{\mathbb{A}}(\tau(\beta x)).
\end{equation}
We call $\mathbf{c}(\beta;\boldsymbol{f},y)$ the Fourier coefficients of $\boldsymbol{f}$ and they have following properties (see for example \cite[Proposition 2.4]{B21}):\\
(1) $\mathbf{c}(\beta;\boldsymbol{f},y)=0$ unless $\beta$ is non-negative and $\prod_{v\nmid\infty}e_v(\tau(y^{\ast}\beta yx))=1$ for any $x\in S_m(\mathcal{O})$,\\
(2) $\mathbf{c}(\beta;\boldsymbol{f},y)=\mathbf{c}(\beta;\boldsymbol{f},\prod_{v\nmid\infty}y_v)$,\\
(3) $\mathbf{c}(b^{\ast}\beta b;\boldsymbol{f},y)=\nu(b^{\ast})^l\mathbf{c}(\beta;\boldsymbol{f},by)$ for any $b\in\mathrm{GL}_n(D)$,\\
(4) $\mathbf{c}(\beta;\boldsymbol{f},yk)=\mathbf{c}(\beta;\boldsymbol{f},y)$ for any $k\in\prod_{v\nmid\infty}\mathrm{GL}_n(\mathcal{O}_v)$,\\
(5) $\boldsymbol{f}$ is a cusp form if and only if for all $y$, $\mathbf{c}(\beta;\boldsymbol{f},y)=0$ unless $\beta$ is positive definite.

Then the action of $\sigma\in\mathrm{Aut}(\C)$ on $\boldsymbol{f}$ characterized by
\begin{equation}
\boldsymbol{f}^{\sigma}\left(\left[\begin{array}{cc}
y & x\hat{y}\\
0 & \hat{y}
\end{array}\right]\right)=\prod_{v|\infty}\nu(y_v^{\ast})^{l_v}\cdot\sum_{\beta\in S_m(F)}\mathbf{c}(\beta;\boldsymbol{f},y)^{\sigma}e_{\infty}(\tau(iy^{\ast}\beta y))e_{\mathbb{A}}(\tau(\beta x))
\end{equation}
is well defined. In particular, we have $\boldsymbol{f}\in\mathcal{M}_{\boldsymbol{l}}(K,\overline{\Q})$ if and only if $\mathbf{c}(\beta;\boldsymbol{f},y)\in\overline{\Q}$ for all $y$.

\end{rem}

\section{Reformulating the Integral Representations}
\label{section 6}

Keep the assumption of our global group as the beginning of Section \ref{section 5}. We conclude our integral representation from Theorem \ref{theorem 2.2} together with archimedean computations in Proposition \ref{proposition 5.4}. We also reformulate our integral representation for our later study of algebraic and $p$-adic properties.

Let $\boldsymbol{l}=(l_v)_{v|\infty}$ be a tuple of positive integers indexed by archimedean places of $F$. Fix a specific prime $\boldsymbol{p}$ of $\mathfrak{o}$ and an integral ideal $\mathfrak{n}=\mathfrak{n}_1\mathfrak{n}_2=\prod_{v}\mathfrak{p}_v^{\mathfrak{c}_v}$ with $\mathfrak{n}_1,\mathfrak{n}_2,\boldsymbol{p}$ coprime. Denote $\boldsymbol{\varpi}$ for the uniformizer of $\boldsymbol{p}$. Let $\boldsymbol{q}$ be the prime ideal of $\mathcal{O}$ above $\boldsymbol{p}$ and $\boldsymbol{\widetilde{\varpi}}$ the uniformizer of $\boldsymbol{q}$. We make the following assumptions:\\
(1) $2\in\mathcal{O}_v^{\times}$ and $\theta\in\mathrm{GL}_r(\mathcal{O}_v)$ for all $v|\mathfrak{n}\boldsymbol{p}$.\\
(2) $\boldsymbol{f}\in\mathcal{S}_{\boldsymbol{l}}(K(\mathfrak{n}\boldsymbol{p}))$ is an eigenform for the Hecke algebra $\mathcal{H}(K(\mathfrak{n}\boldsymbol{p}),\mathfrak{X})$ as in Section \ref{section 3.4}.\\
(3) $\boldsymbol{f}$ is an eigenform for the $U(\boldsymbol{p})$ operator with eigenvalue $\alpha(\boldsymbol{p})\neq 0$.\\
(4) $\chi=\chi_1\boldsymbol{\chi}$ with $\chi_1$ has conductor $\mathfrak{n}_2$ and $\boldsymbol{\chi}$ has conductor $\boldsymbol{p}^{\boldsymbol{c}}$ for some integer $\boldsymbol{c}\geq 0$. We assume $\chi$ has infinity type $\boldsymbol{l}$. That is, $\chi_v(x)=x^{l_v}|x|^{-l_v}$ for all $v|\infty$.\\
(5) In Case V, all places $v|\mathfrak{n}\boldsymbol{p}$ are nonsplit in $\mathcal{O}$.

Denote $\eta_1,\eta_2\in G(\mathbb{A})$ such that
\begin{equation}
\label{6.1}
(\eta_1)_v=\left\{\begin{array}{cc}
w & v|\mathfrak{n}_1\\
1 & \text{otherwise},
\end{array}\right.,\quad(\eta_2)_v=\left\{\begin{array}{cc}
w & v|\mathfrak{n}_2\boldsymbol{p}\\
1 & \text{otherwise},
\end{array}\right.,
\end{equation}
where
\[
w=\left[\begin{array}{ccc}
0 & 0 & 1_m\\
0 & 1_r & 0\\
\epsilon\cdot 1_m & 0 & 0
\end{array}\right]
\]
is an Weyl element.

Denote $E(h;f_s)$ be the Eisenstein series on $H(\mathbb{A})$ associated to $f_s$. Taking $\phi_1$ to be $\pi(\eta_1)\boldsymbol{f}$ and $\phi_2$ to be $\pi(\eta_2)\boldsymbol{f}$, our global integral \eqref{2.2.7} can be written as
\begin{equation}
\label{6.2}
\begin{aligned}
&\mathcal{Z}(s;\boldsymbol{f},f_s)\\
=&\int_{G(F)\times G(F)\backslash G(\mathbb{A})\times G(\mathbb{A})}E((g_1,g_2);f_s)\overline{\boldsymbol{f}(g_1\eta_1)}\boldsymbol{f}(g_2\eta_2)\chi(\nu(g_2))^{-1}dg_1dg_2.
\end{aligned}
\end{equation}

The integral representation is summarized in the following theorem.

\begin{thm}
\label{theorem 6.1}
Take the section $f_s$ to be
\begin{equation}
\label{6.3}
\begin{aligned}
f_s&=\prod_{v\nmid\mathfrak{n}\boldsymbol{p}\infty}f_{s,v}^0\cdot\prod_{v|\mathfrak{n}_1}f_{s,v}^{\dagger,\mathfrak{c}_v}\cdot\prod_{v|\mathfrak{n}_2}f_{s,v}^{\ddagger,\mathfrak{c}_v}\cdot f_{s,\boldsymbol{p}}^{\ddagger,\boldsymbol{c}}\cdot f_{s}^{\infty},\qquad&\boldsymbol{c}>0,\\
f_s&=\prod_{v\nmid\mathfrak{n}\boldsymbol{p}\infty}f_{s,v}^0\cdot\prod_{v|\mathfrak{n}_1}f_{s,v}^{\dagger,\mathfrak{c}_v}\cdot\prod_{v|\mathfrak{n}_2}f_{s,v}^{\ddagger,\mathfrak{c}_v}\cdot f_{s,\boldsymbol{p}}^p\cdot f_{s}^{\infty},\qquad&\boldsymbol{c}=0.
\end{aligned}
\end{equation}
with $f_{s,v}^0,f_{s,v}^{\dagger,\mathfrak{c}_v},f_{s,v}^{\ddagger,\mathfrak{c}_v},f_{s,v}^p$ are local sections defined in \eqref{4.2.1}, \eqref{4.3.2}, \eqref{4.4.1}, \eqref{4.5.6} and $f_{s}^{\infty}$ the archimedean section defined in \eqref{5.2.3}. Then 
\begin{equation}
\label{6.4}
\begin{aligned}
\mathcal{Z}(s;\boldsymbol{f},f_s)&=C'\cdot\prod_{v|\infty}c_{l_{v}}(s)\cdot L\left(s+\frac{1}{2},\boldsymbol{f}\times\chi\right)\cdot\langle\pi(\eta)\boldsymbol{f}|U'(\mathfrak{n_1}),\boldsymbol{f}\rangle,\quad &\boldsymbol{c}>0,\\
\mathcal{Z}(s;\boldsymbol{f},f_s)&=C''\cdot\prod_{v|\infty}c_{l_{v}}(s)\cdot L\left(s+\frac{1}{2},\boldsymbol{f}\times\chi\right)\cdot\langle\pi(\eta)\boldsymbol{f}|U'(\mathfrak{n_1}),\boldsymbol{f}\rangle &\\
&\times M\left(s+\frac{1}{2},\boldsymbol{f}\times\chi\right),\quad&\boldsymbol{c}=0.
\end{aligned}
\end{equation} 
Here:\\
(a) $M(s,\boldsymbol{f}\times\chi)$ is the modification factor given in Proposition \ref{proposition 4.6},\\
(b) $U'(\mathfrak{n}_1)$ is the Hecke operator defined by \eqref{3.2.8},\\
(c)
\begin{equation}
\label{6.5}
\begin{aligned}
\eta&=\left[\begin{array}{ccc}
0 & 0 & \boldsymbol{\varpi}^{-\boldsymbol{c}}\cdot 1_m\\
0 & 1_r & 0\\
\boldsymbol{\varpi}^{\boldsymbol{c}}\cdot 1_m & 0 & 0
\end{array}\right]\cdot\prod_{v|\mathfrak{n}_2}\left[\begin{array}{ccc}
0 & 0 & \varpi_v^{-\mathfrak{c}_v}\cdot 1_m\\
0 & 1_r & 0\\
\varpi_v^{\mathfrak{c}_v}\cdot 1_m & 0 & 0
\end{array}\right],\quad&\boldsymbol{c}>0,\\
\eta&=\left[\begin{array}{ccc}
0 & 0 & \boldsymbol{\widetilde{\varpi}}^{-1}\cdot 1_m\\
0 & 1_r & 0\\
\boldsymbol{\widetilde{\varpi}}\cdot 1_m & 0 & 0
\end{array}\right]\cdot\prod_{v|\mathfrak{n}_2}\left[\begin{array}{ccc}
0 & 0 & \varpi_v^{-\mathfrak{c}_v}\cdot 1_m\\
0 & 1_r & 0\\
\varpi_v^{\mathfrak{c}_v}\cdot 1_m & 0 & 0
\end{array}\right],\quad&\boldsymbol{c}=0,
\end{aligned}
\end{equation}
(d) The constants $c_{l_v}(s)$ are given in \eqref{5.2.5}. Up to a power of $2$ depending on $s$,
\begin{equation}
\label{6.6}
\begin{aligned}
C'=&\chi(\mathfrak{n}_1)^{m\mathbf{d}_1}|\mathfrak{n}_1|^{m\mathbf{d}_2(s+\kappa)}\mathrm{vol}(\mathrm{GL}_m(\mathcal{O})/\mathrm{GL}_m(\mathfrak{n}_2\boldsymbol{p}^{\boldsymbol{c}}\mathcal{O})),
\end{aligned}
\end{equation}
and
\begin{equation}
\label{6.7}
\begin{aligned}
C''=&(-1)^m|\boldsymbol{\varpi}|_{\boldsymbol{p}}^{\mathbf{d}_3\frac{m^2+m}{2}}\chi(\mathfrak{n}_1)^{m\mathbf{d}_1}|\mathfrak{n}_1|^{m\mathbf{d}_2(s+\kappa)}\mathrm{vol}(\mathrm{GL}_m(\mathcal{O})/\mathrm{GL}_m(\mathfrak{n}_2\mathcal{O})),
\end{aligned}
\end{equation}
with
\begin{equation}
\label{6.8}
\begin{aligned}
\mathbf{d}_1&=\left\{\begin{array}{cc}
1 & \text{ Case I, II, V},\\
2 & \text{ Case III, IV},
\end{array}\right.\qquad\mathbf{d}_2=\left\{\begin{array}{cc}
1 & \text{ Case I, II},\\
2 & \text{ Case III, IV, V},
\end{array}\right.\\
\mathbf{d}_3&=\left\{\begin{array}{cc}
1 & \text{ Case I, II, V Ramified,}\\
2 & \text{ Case III, IV, V Inert.}
\end{array}\right..
\end{aligned}
\end{equation}
\end{thm}

For any integer $\boldsymbol{n}\geq 0$, let $\mathcal{K}(\boldsymbol{p}^{\boldsymbol{n}})$ (resp. $\mathcal{K}'(\boldsymbol{p}^{\boldsymbol{n}})$) be an open compact subgroup of $G(\mathbb{A})$ defined by $\mathcal{K}(\boldsymbol{p}^{\boldsymbol{n}})=K'(\mathfrak{n}_1)K(\mathfrak{n}_2)K(\boldsymbol{p}^{\boldsymbol{n}})$ (resp. $\mathcal{K}'(\boldsymbol{p}^{\boldsymbol{n}})=K(\mathfrak{n}_1)K'(\mathfrak{n}_2)K'(\boldsymbol{p}^{\boldsymbol{n}})$). Denote $w_{\infty}\in G(\mathbb{A})$ be an element such that $w_v=w$ for any archimedean place $v$ and $w_v=1$ for all non-archimedean places. Then for $f_s$ as above, we have
\begin{equation}
\label{6.9}
\mathcal{E}(g_1,g_2;f_s):=\chi(\nu(g_2))^{-1}E((g_1,g_2^{\iota});f_s)\in\mathcal{M}_{\boldsymbol{l}}(\mathcal{K}(\boldsymbol{p}^{2\boldsymbol{n}}))\otimes\mathcal{M}_{\boldsymbol{l}}(\mathcal{K}'(\boldsymbol{p}^{2\boldsymbol{n}})),
\end{equation}
with 
\begin{equation}
\label{6.10}
g^{\iota}=\left[\begin{array}{ccc}
-1_m & 0 &0\\
0 & 1_r & 0\\
0 & 0 & 1_m
\end{array}\right]_{\infty}\cdot g\cdot \left[\begin{array}{ccc}
-1_m & 0 & 0\\
0 & -1_r & 0\\
0 & 0 & 1_m
\end{array}\right]_{\infty}
\end{equation}
where the matrix with subscript $\infty$ indicates an element in $G(F_{\infty})$. Here, in this section only, we abuse the notation by writing $\mathcal{M}_{\boldsymbol{l}}(\mathcal{K}(\boldsymbol{p}^{2\boldsymbol{n}}))$ for the space of functions transforming as a modular form (i.e. satisfying \eqref{5.3.2} but may not necessary holomorphic). That is, $\mathcal{E}(g_1,g_2;f_s)$ transforms as a modular form in $\mathcal{M}_{\boldsymbol{l}}(\mathcal{K}(\boldsymbol{p}^{2\boldsymbol{n}}))$ for the first variable and $\mathcal{M}_{\boldsymbol{l}}(\mathcal{K}'(\boldsymbol{p}^{2\boldsymbol{n}}))$ for the second variable. Indeed, later in Section \ref{section 7} and \ref{section 8}, we will specialize to the special points $s=s_0$ (as in \eqref{8.2}) in which case $\mathcal{E}(g_1,g_2;f_s)$ is holomorphic in both variables (follows from the Fourier expansion). Note that $\boldsymbol{n}$ can be took as any integer such that $\boldsymbol{n}\geq\boldsymbol{c}$ if $\boldsymbol{c}>0$ and $\boldsymbol{n}\geq1$ if $\boldsymbol{c}=0$. We also remark that the involution $\iota$ is included in the second variable since our doubling embedding of the symmetric space \eqref{5.2.2} is holomorphic in the first variable and anti-holomorphic in the second variable. To compare all these integral representations when varying the character $\boldsymbol{\chi}$ of different conductors, we further descend the level of Eisenstein series such that it is independent of $\boldsymbol{c}$. Our approach is an analogue of \cite[Section 4]{BS}.

\begin{rem}
\label{remark 6.2}
In the following we actually assume $\boldsymbol{p}$ is nonsplit in $D$. As our argument is local, it directly extended to the split cases of Case III and IV by identifying the local group $G(F_{\boldsymbol{p}})$ with the group in Case I or Case II as in Section \ref{section 3.3}.
\end{rem}

We will use the following general lemma to descend the level.

\begin{lem}
\label{lemma 6.3}
The Hecke operator $U(\boldsymbol{p}^{n-1})$ defined by \eqref{3.2.8} maps $\mathcal{M}_{\boldsymbol{l}}(\mathcal{K}(\boldsymbol{p}^{2\boldsymbol{n}}))$ to $\mathcal{M}_{\boldsymbol{l}}(\mathcal{K}(\boldsymbol{p}^{2}))$.
\end{lem}

\begin{proof}
We define a map $\mathcal{M}_{\boldsymbol{l}}(\mathcal{K}(\boldsymbol{p}^{2\boldsymbol{n}}))\to\mathcal{M}_{\boldsymbol{l}}(\mathcal{K}(\boldsymbol{p}^{2}))$ in following steps. Let $f\in\mathcal{M}_{\boldsymbol{l}}(\mathcal{K}(\boldsymbol{p}^{2\boldsymbol{n}}))$ and first set 
\[
f_1(g):=f\left(g\left[\begin{array}{ccc}
\boldsymbol{\varpi}^{\boldsymbol{n}}\cdot 1_m & 0 & 0\\
0 & 1_r & 0\\
0 & 0 & \boldsymbol{\varpi}^{-\boldsymbol{n}}\cdot 1_m
\end{array}\right]\right).
\]
Then $f_1$ is fixed by
\[
K''(\boldsymbol{p}^{\boldsymbol{n}}):=
G(\mathfrak{o}_{\boldsymbol{p}})\cap\left[\begin{array}{ccc}
\mathrm{Mat}_{m}(\mathcal{O}_{\boldsymbol{p}}) & \mathrm{Mat}_{m,r}(\boldsymbol{p}^{\boldsymbol{n}}\mathcal{O}_{\boldsymbol{p}}) & \mathrm{Mat}_{m}(\boldsymbol{p}^{2\boldsymbol{n}}\mathcal{O}_{\boldsymbol{p}})\\
\mathrm{Mat}_{r,m}(\boldsymbol{p}^{\boldsymbol{n}}\mathcal{O}_{\boldsymbol{p})} &1+\mathrm{Mat}_{r}(\boldsymbol{p}\mathcal{O}_{\boldsymbol{p}}) & \mathrm{Mat}_{r,m}(\boldsymbol{p}^{\boldsymbol{n}}\mathcal{O}_{\boldsymbol{p}})\\
\mathrm{Mat}_{m}(\mathcal{O}_{\boldsymbol{p}}) & \mathrm{Mat}_{m,r}(\boldsymbol{p}^{\boldsymbol{n}}\mathcal{O}_{\boldsymbol{p}}) & \mathrm{Mat}_{m}(\mathcal{O}_{\boldsymbol{p}})
\end{array}\right],
\]
Secondly we define
\[
f_2(g):=\sum_{\gamma\in K''(\boldsymbol{p})/K''(\boldsymbol{p}^{\boldsymbol{n}})}f_1(g\gamma),
\]
where the representatives of $K''(\boldsymbol{p})/K''(\boldsymbol{p}^{\boldsymbol{n}})$ can be taken as
\[
\left[\begin{array}{ccc}
1_m & -\boldsymbol{\varpi}b^{\ast}\theta^{-1} & \boldsymbol{\varpi}^2 c\\
0 & 1_r & \boldsymbol{\varpi}b\\
0 & 0 & 1_r
\end{array}\right],
\]
with $b\in\mathrm{Mat}_{m,r}(\mathcal{O}_{\boldsymbol{p}}/\boldsymbol{p}^{\boldsymbol{n}-1}\mathcal{O}_{\boldsymbol{p}})$ and $c\in\mathrm{Mat}_{m}(\mathcal{O}_{\boldsymbol{p}}/\boldsymbol{p}^{2\boldsymbol{n}-2}\mathcal{O}_{\boldsymbol{p}})$ satisfying $\epsilon c+b^{\ast}\hat{\theta}b+c^{\ast}=0$. Then $f_2\in\mathcal{M}_{\boldsymbol{l}}(K''(\boldsymbol{p}))$. Finally we put
\[
f_3(g):=f_2\left(g\left[\begin{array}{ccc}
\boldsymbol{\varpi}^{-1}\cdot 1_m & 0 & 0\\
0 & 1_r & 0\\
0 & 0 & \boldsymbol{\varpi}\cdot 1_m
\end{array}\right]\right).
\]
to obtain $f_3\in\mathcal{M}_{\boldsymbol{l}}(\mathcal{K}(\boldsymbol{p}^{2}))$. Combining these three steps together, $f\mapsto f_3$ defines a map 
\[
\begin{aligned}
\mathrm{Tr}:\mathcal{M}_{\boldsymbol{l}}(\mathcal{K}(\boldsymbol{p}^{2\boldsymbol{n}}))&\to\mathcal{M}_{\boldsymbol{l}}(\mathcal{K}(\boldsymbol{p}^{2}))\\
f&\mapsto f|\mathrm{Tr}(g):=\sum_{\gamma}f(g\gamma)
\end{aligned}
\]
where $\gamma$ runs through elements of the form
\[
\left[\begin{array}{ccc}
\boldsymbol{\varpi}^{\boldsymbol{n}-1}\cdot 1_m & -b^{\ast}\theta^{-1} & \epsilon\boldsymbol{\varpi}^{1-\boldsymbol{n}}c^{\ast}\\
0 & 1_r & \boldsymbol{\varpi}^{1-\boldsymbol{n}}b\\
0 & 0 & \boldsymbol{\varpi}^{1-\boldsymbol{n}}\cdot 1_m 
\end{array}\right]
\]
with $b\in\mathrm{Mat}_{m,r}(\mathcal{O}_{\boldsymbol{p}}/\boldsymbol{p}^{\boldsymbol{n}-1}\mathcal{O}_{\boldsymbol{p}})$ and $c\in\mathrm{Mat}_{m}(\mathcal{O}_{\boldsymbol{p}}/\boldsymbol{p}^{2\boldsymbol{n}-2}\mathcal{O}_{\boldsymbol{p}})$ satisfying $\epsilon c+b^{\ast}\hat{\theta}b+c^{\ast}=0$. Comparing above matrix with the one in \eqref{3.2.6} for $U(\boldsymbol{p}^{\boldsymbol{n}-1})$ operator we obtain the lemma.
\end{proof}

We apply above process for both variables and define
\begin{equation}
\label{6.11}
\mathbb{E}(h;f_s):=E(h;f_s)|\boldsymbol{U}(\boldsymbol{p}^{\boldsymbol{n}-1}):=\sum_{\gamma}E(h\gamma;f_s),
\end{equation}
where $\gamma$ runs through elements of the form
\begin{equation}
\label{6.12}
\left[\begin{array}{cccccc}
1_r & 0 & 0 & 0 & \frac{\epsilon \boldsymbol{\varpi}^{1-\boldsymbol{n}}b_2}{2} & -\frac{\epsilon \boldsymbol{\varpi}^{1-\boldsymbol{n}}b_1}{2}\\
\epsilon b_2^{\ast}\theta^{-1} & \boldsymbol{\varpi}^{\boldsymbol{n}-1}\cdot 1_m & 0 & -\frac{b_2^{\ast}}{2} & -\boldsymbol{\varpi}^{1-\boldsymbol{n}}c_2 & 0\\
-\epsilon b_1^{\ast}\theta^{-1} & 0 & \boldsymbol{\varpi}^{\boldsymbol{n}-1}\cdot 1_m & -\frac{b_1^{\ast}}{2} & 0 & \boldsymbol{\varpi}^{1-\boldsymbol{n}}c_1\\
0 & 0 & 0 & 1_r & -\boldsymbol{\varpi}^{1-\boldsymbol{n}}\theta^{-1}b_2 & \boldsymbol{\varpi}^{1-\boldsymbol{n}}\theta^{-1}b_1\\
0 & 0 & 0 & 0 & \boldsymbol{\varpi}^{1-\boldsymbol{n}}\cdot 1_m & 0\\
0 & 0 & 0 & 0 & 0 & \boldsymbol{\varpi}^{1-\boldsymbol{n}}\cdot 1_m 
\end{array}\right]
\end{equation}
with $b_1,b_2\in\mathrm{Mat}_{m,r}(\mathcal{O}_{\boldsymbol{p}}/\boldsymbol{p}^{\boldsymbol{n}-1}\mathcal{O}_{\boldsymbol{p}})$ and $c_1,c_2\in\mathrm{Mat}_{m}(\mathcal{O}_{\boldsymbol{p}}/\boldsymbol{p}^{2\boldsymbol{n}-2}\mathcal{O}_{\boldsymbol{p}})$ satisfying $\epsilon c_1+b_1^{\ast}\hat{\theta}b_1+c_1^{\ast}=0,\epsilon c_2+b_2^{\ast}\hat{\theta}b+c_2^{\ast}=0$. 

Take $\mathbb{E}(g_1,g_2;f_s)$ as with $f_s$ as in \eqref{6.3}. Then
\begin{equation}
\label{6.13}
\boldsymbol{E}(g_1,g_2;f_s):=\chi(\nu(g_2))^{-1}\mathbb{E}((g_1,g_2^{\iota});f_s)\in\mathcal{M}_{\boldsymbol{l}}(\mathcal{K}(\boldsymbol{p}^{2}))\otimes\mathcal{M}_{\boldsymbol{l}}(\mathcal{K}'(\boldsymbol{p}^{2})).
\end{equation}
We may also denote $\mathbb{E}(h;f_s,\chi,\boldsymbol{n})$ and $\boldsymbol{E}(g_1,g_2;f_s,\chi,\boldsymbol{n})$ to emphasize their dependence on $\chi,\boldsymbol{n}$. Consider the global integral
\begin{equation}
\label{6.14}
\begin{aligned}
&\boldsymbol{Z}(s;\boldsymbol{f},f_s)\\
:=&\int_{G(F)\times G(F)\backslash G(\mathbb{A})\times G(\mathbb{A})}\boldsymbol{E}((g_1,g_2^{\iota});f_s)\overline{\boldsymbol{f}(g_1\eta_1\eta_{\boldsymbol{p}})}\boldsymbol{f}(g_2\eta_2)dg_1dg_2.
\end{aligned}
\end{equation}
with $\eta_1,\eta_2$ as \eqref{6.1} and 
\begin{equation}
\label{6.15}
\eta_{\boldsymbol{p}}=\left[\begin{array}{ccc}
0 & 0 & \boldsymbol{\varpi}^{-1}\cdot 1_m\\
0 & 1_r & 0\\
\boldsymbol{\varpi}\cdot 1_m & 0 & 0
\end{array}\right]\in G(F_{\boldsymbol{p}}).
\end{equation}
Again we may also denote $\boldsymbol{Z}(s;\boldsymbol{f},f_s,\chi,\boldsymbol{n})$ to emphasize its dependence on $\chi,\boldsymbol{n}$. By simply changing variables, we reformulate Theorem \ref{theorem 6.1} in the following corollary. We remark that it is essential to assume that the eigenvalue $\alpha(\boldsymbol{p})\neq 0$ otherwise the integral will be identically zero.

\begin{cor}
\label{corollary 6.4}
For $\boldsymbol{c}>0$ we have
\begin{equation}
\label{6.16}
\boldsymbol{Z}(s;\boldsymbol{f},f_s)=\alpha(\boldsymbol{p})^{2\boldsymbol{n}-2}C'\prod_{v|\infty}c_{l_{v}}(s)\cdot L\left(s+\frac{1}{2},\boldsymbol{f}\times\chi\right)\cdot\langle\pi(\eta)\boldsymbol{f}|U'(\mathfrak{n_1}),\boldsymbol{f}\rangle,\\
\end{equation}
and for $\boldsymbol{c}=0$ we have
\begin{equation}
\label{6.17}
\begin{aligned}
\boldsymbol{Z}(s;\boldsymbol{f},f_s)&=\alpha(\boldsymbol{p})^{2\boldsymbol{n}-2}C''\prod_{v|\infty}c_{l_{v}}(s)\cdot L\left(s+\frac{1}{2},\boldsymbol{f}\times\chi\right)\cdot\langle\pi(\eta)\boldsymbol{f}|U'(\mathfrak{n_1}),\boldsymbol{f}\rangle \\
&\times M\left(s+\frac{1}{2},\boldsymbol{f}\times\chi\right),
\end{aligned}
\end{equation} 
Here the notations are same as Theorem \ref{theorem 6.1} except 
\begin{equation}
\label{eta}
\eta=\prod_{v|\mathfrak{n}_2}\left[\begin{array}{ccc}
0 & 0 & \varpi_v^{-\mathfrak{c}_v}\cdot 1_m\\
0 & 1_r & 0\\
\varpi_v^{\mathfrak{c}_v}\cdot 1_m & 0 & 0
\end{array}\right].
\end{equation}
\end{cor}

\begin{rem}
If we denote $\boldsymbol{f}^1(g):=\boldsymbol{f}(g\eta_1\eta_{\boldsymbol{p}}), \boldsymbol{f}^2(g):=\boldsymbol{f}(g\eta_2)$ and let $\boldsymbol{V}$ be the operator defined by $\boldsymbol{f}^2|\boldsymbol{V}:=\pi(\eta)\boldsymbol{f}^2|U'(\mathfrak{n}_1)$, then our computations in Section \ref{section 4} also show that
\begin{equation}
\label{6.19}
\left\langle\boldsymbol{E}(g_1,g_2;f_s),\boldsymbol{f}^1(g_1)\right\rangle=\frac{\boldsymbol{Z}(s;\boldsymbol{f},f_s)}{\langle\boldsymbol{f}^2|\boldsymbol{V},\boldsymbol{f}^2\rangle}\cdot\overline{\boldsymbol{f}^2|\boldsymbol{V}(g_2^{\iota})},
\end{equation}
where the left hand side is the Petersson inner product respect to $g_1$. The integral \eqref{6.19} is the adelic version of the integral representation obtained in \cite{BS} and \cite{Sh97, Sh00}. One can also further reformulate the integral in a classical setting as there (see also \cite[Section 4]{JYB2}). Indeed, recall that by the weak approximation \eqref{weakapproximation} of $G$ there is finite number $h$ such that
\[
G(\mathbb{A})=\coprod_{1\leq i\leq h}G(F)t_iKG(F_{\infty}).
\]
For $1\leq i,j\leq h$ and $z=g_{\infty}z_0\in\mathfrak{Z}_{m,r},w=g_{\infty}'z_0\in\mathfrak{Z}_{m,r}$, we set 
\[
f^1_i(z)=\boldsymbol{f}^1(t_ig_{\infty}),\quad f^2_j(w)=\boldsymbol{f}^2|\boldsymbol{V}(t_jg'_{\infty}),\quad E_{ij}(z,w)=\boldsymbol{E}(t_ig_{\infty},t_jg'_{\infty};f_s).
\]
Then the integral \eqref{6.14} can be written as
\begin{equation}
\sum_{i,j}\left\langle\left\langle E_{ij}(z,-\overline{w}), f_i^1(z)\right\rangle, f_j^2(w)\right\rangle
\end{equation}
and \eqref{6.19} can be rewritten as
\begin{equation}
\sum_i\left\langle E_{ij}(z,w),f_i^1(z)\right\rangle=\frac{\boldsymbol{Z}(s;\boldsymbol{f},f_s)}{\langle\boldsymbol{f}^2|\boldsymbol{V},\boldsymbol{f}^2\rangle}\cdot \overline{f_j^2(-\overline{w})},
\end{equation}
which is the pullback formula obtained in \cite{Sh97,Sh00}.
\end{rem}

\section{Fourier Expansion of Eisenstein Series}
\label{section 7}

We calculate the Fourier expansion of the Eisenstein series in this section. Then the properties of Eisenstein series can be directly obtained from the properties of Fourier coefficients.

\subsection{Generalities}
\label{section 7.1}

Let $e_{\mathbb{A}}=\prod_ve_v$ be the standard additive character of $\mathbb{A}$. That is, $e_v(x)=e^{2\pi ix}$ for an archimedean place $v$ and $e_v(\varpi_v)=\prod_{v|\infty}e_v(-|\varpi_v|_v^{-1})$ with $\varpi_v$ the uniformizer of $F_v$. Denote $S_n$ to be the additive algebraic group such that
\begin{equation}
\label{7.1.1}
S_n(F)=\{\beta\in\mathrm{Mat}_n(D):\beta^{\ast}=-\epsilon\beta\}.
\end{equation}
The Eisenstein series $E(h;f_s)$ on $H(\mathbb{A})$ has a Fourier expansion of the form
\begin{equation}
\label{7.1.2}
\begin{aligned}
E(h;f_s)&=\sum_{\beta\in S_n(F)}E_{\beta}(h;f_s),\\
E_{\beta}(h;f_s)&=\int_{S_n(F)\backslash S_n(\mathbb{A})}E\left(\left[\begin{array}{cc}
1_n & S\\
0 & 1_n
\end{array}\right]h;f_s\right)e_{\mathbb{A}}(-\tau(\beta S))dS.
\end{aligned}
\end{equation}
By the Iwasawa decomposition, it suffices to calculate the Fourier coefficients for parabolic element $q\in P(\mathbb{A})$. In particular, we can take $q_v=\mathrm{diag}[y,\hat{y}]$ for non-archimedean places $v$ and for all but finitely many $v$ we can assume $q_v=1$. For an archimedean place $v$, we can take $q_v=\left[\begin{array}{cc}
y_v & x_v\hat{y}_v\\
0 & \hat{y}_v
\end{array}\right]$ with $z_v=x_v+iy_vy_v^{\ast}\in\mathcal{H}_n$. We shall also denote such $q$ as $q_z$ to indicate its dependence on $z=(z_v)_{v|\infty}$.

Since our $f_s$ is chosen such that, for at least one place $v$, the support of $f_{s,v}$ is in the big cell $P(F)J_nP(F)$, the Fourier coefficient $E_{\beta}(q;f_s)$ at parabolic element $q\in P(\mathbb{A})$ is factorizable. That is
\begin{equation}
\label{7.1.3}
\begin{aligned}
E_{\beta}(q;f_s)&=\prod_vE_{\beta,v}(q;f_s),\\
E_{\beta,v}(q;f_s)&=\int_{S_n(F_v)}f_{s,v}\left(J_n\left[\begin{array}{cc}
1_n & S\\
0 & 1_n
\end{array}\right]q\right)e_v(-\tau(\beta S))dS.
\end{aligned}
\end{equation}

For local sections $f_s=f_s^0, f_s^{\dagger,\mathfrak{c}}, f_s^{\ddagger,\mathfrak{c}}, f_s^p, f_s^{\infty}$ defined in Section \ref{section 4}, we calculate the local Fourier coefficients $E_{\beta,v}(q;f_s)$ place by place in next two subsections.

\subsection{Non-archimedean computations}

\label{section 7.2}

Let $F$ be a non-archimedean local field and $\mathfrak{o}$ its ring of integers with the maximal ideal $\mathfrak{p}$. Fix uniformizer $\varpi$ and the absolute value $|\cdot|$ on $F$ normalized so that $|\varpi|=q^{-1}$ with $q$ the cardinality of the residue field. We also fix a maximal order $\mathcal{O}$ of $D$ such that $D=\mathcal{O}\otimes_{\mathfrak{o}}F$. Let $\mathfrak{q}$ be a prime in $\mathcal{O}$ above $\mathfrak{p}$ and fix $\widetilde{\varpi}$ a uniformizer of $\mathfrak{q}$. 

We are going to calculate local Fourier coefficients
\begin{equation}
\label{7.2.1}
E_{\beta}(q;f_s)=\int_{S_n(F)}f_{s}\left(J_n\left[\begin{array}{cc}
1_n & S\\
0 & 1_n
\end{array}\right]q\right)e(-\tau(\beta S))dS
\end{equation}
for various local sections $f_s^0,f_s^{\dagger,\mathfrak{c}},f_s^{\ddagger,\mathfrak{c}},f_s^p$ defined in \eqref{4.2.1}, \eqref{4.3.2}, \eqref{4.4.1}, \eqref{4.5.6}.

\subsubsection{The unramified case}

We first consider the local section $f_s^0$. Denote
\begin{equation}
\label{7.2.1.1}
S_n(\mathfrak{o})^{\ast}=\{\beta\in S_{n}(F):\tau(\beta S)\in\mathfrak{o}\text{ for any }S\in S_{n}(\mathfrak{o})\}.
\end{equation}

\begin{prop}
\label{proposition 7.1}
Set $t=\mathrm{rank}(\beta)$ and $\transpose{b}\beta b=\mathrm{diag}[\beta',0]$ with $b\in\mathrm{GL}_n(\mathcal{O})$ and $\beta'\in S_t(F)$. Let $q=\mathrm{diag}[a,\hat{a}]$, then $E_{\beta}(q;f_s^0)$ is nonzero only if $a^{\ast}\beta a\in S_t(\mathfrak{o})^{\ast}$. In this case, up to the term $\chi(\nu(a))|N_{E/F}(\nu(a))|^{s+\kappa}$ and a power of the discriminant of $D$, $E_{\beta}(q;f_s^0)$ is given by the following list. \\
(Case I, Orthogonal) This case occurs as quaternionic unitary split case.
\[
\prod_{i=1}^{\left\lfloor\frac{n-t}{2}\right\rfloor}L\left(2s-n+t+2i,\chi^2\right)\cdot P_{a^{\ast}\beta a}(\chi(q)q^{-s-\kappa}),
\]
(Case II, Symplectic Even) Assume $t$ is even. Let $\lambda_{\beta}$ be the quadratic character associated to the quadratic field $F((-1)^{\frac{t}{2}}\nu(2\beta))$ over $F$. 
\[
L\left(s-\frac{n-1}{2}+\frac{t}{2},\chi\lambda_{\beta}\right)\cdot \prod_{i=1}^{\left\lfloor\frac{n-t}{2}\right\rfloor}L(2s-n+t+2i,\chi^2)\cdot P_{a^{\ast}\beta a}(\chi(q)q^{-s-\kappa}),
\]
(Case II, Symplectic Odd) Assume $t$ is odd. This case only occurs as quaternionic orthogonal split case.
\[
\prod_{i=1}^{\left\lfloor\frac{n-t+1}{2}\right\rfloor}L(2s-n+t-1+2i,\chi^2)\cdot P_{a^{\ast}\beta a}(\chi(q)q^{-s-\kappa}),
\]
(Case III, Quaternionic Orthogonal Nonsplit Even) Assume $t$ is even. Let $\lambda_{\beta}$ be the quadratic character associated to the quadratic field $F((-1)^{\frac{t}{2}}\nu(2\beta))$ over $F$. 
\[
L\left(s-\frac{2n-1}{2}+t,\chi\lambda_{\beta}\right)\cdot\prod_{i=1}^{n-t}L\left(2s-2n+2t+2i,\chi^2\right)\cdot P_{a^{\ast}\beta a}(\chi(q)q^{-s-\kappa}),
\]
(Case III, Quaternionic Orthogonal Nonsplit Odd) Assume $t$ is odd. 
\[
\prod_{i=1}^{n-t}L\left(2s-2n+2t+2i,\chi^2\right)\cdot P_{a^{\ast}\beta a}(\chi(q)q^{-s-\kappa}),
\]
(Case IV, Quaternionic Unitary Nonsplit)
\[
\prod_{i=1}^{n-t}L\left(2s-2n+2t+2i,\chi^2\right)\cdot P_{a^{\ast}\beta a}(\chi(q)q^{-s-\kappa})
\]
(Case V, Unitary)
\[
\prod_{i=1}^{n-t}L(2s-i+1,\chi^0\chi_{E/F}^{n+i})\cdot P_{a^{\ast}\beta a}(\chi^0(q)q^{-2s-2\kappa}).
\]
Here $L(s,\chi)$ means the local $L$-factor of Hecke $L$-functions and $P_{\alpha^{\ast}\beta\alpha}(X)\in\Z[X]$ is a polynomial with coefficients in $\Z$ whose constant term is $1$. 
\end{prop}

\begin{proof}
Conjugate $q$ to the left, we obtain
\[
E_{\beta}(q;f_s^0)=\chi(\nu(a))|N_{E/F}(\nu(a))|^{s+\kappa}\int_{S_n(F)}f_{s}\left(J_n\left[\begin{array}{cc}
1_n & a^{-1}S\hat{a}\\
0 & 1_n
\end{array}\right]q\right)e(-\tau(\beta S))dS.
\]
The above integral is the Siegel series $\alpha$ studied in \cite[Chapter III]{Sh97} (see also \cite{F89, F94}). The orthogonal, symplectic and unitary case are listed in \cite[Theorem 13.6]{Sh97}. (We remind the reader that we have already normalized the local section $f_s^0$ by $b(s,\chi)$). Two quaternionic cases can also be calculated in the same way as \cite[Section 13, 14, 15]{Sh97} (see also \cite[Proposition 3.5]{Sh99} and \cite{Y17}).
\end{proof}

\subsubsection{The ramified case}

We now assume $q=1$ and consider the local section $f_s^{\dagger,\mathfrak{c}},f_s^{\ddagger,\mathfrak{c}},f_s^p$. Denote 
\begin{equation}
\label{7.2.2.1}
\begin{aligned}
S_n(\mathfrak{o})^{\mathfrak{c}}&=S_n(\mathfrak{o})\cap\left[\begin{array}{cc}
\mathrm{Mat}_r(\mathfrak{p}\mathcal{O}) & \mathrm{Mat}_{r,2m}(\mathfrak{p^c}\mathcal{O})\\
\mathrm{Mat}_{2m,r}(\mathfrak{p^c}\mathcal{O}) & \mathrm{Mat}_{2m}(\mathfrak{p^c}\mathcal{O})
\end{array}\right],\\
S_n(\mathfrak{o})^{\ast,\mathfrak{c}}&=\{\beta\in S_n(F):\tau(\beta S)\in\mathfrak{o}\text{ for any }S\in S_n(\mathfrak{o})^{\mathfrak{c}}\}.
\end{aligned}
\end{equation}

\begin{prop}
\label{proposition 7.2}
Assume $\nu(\beta)\neq 0$, then
\begin{equation}
\label{7.2.2.2}
E_{\beta}(1;f_s^{\dagger,\mathfrak{c}})=\left\{\begin{array}{cc}
1 & \beta\in S_n(\mathfrak{o})^{\ast,\mathfrak{c}},\\
0 & \text{otherwise}.
\end{array}\right.
\end{equation} 
\end{prop}

\begin{proof}
Consider
\[
E_{\beta}(1;f^{\dagger,\mathfrak{c}}_s)=\int_{S_n(F)}f^{\dagger,\mathfrak{c}}_s\left(J_n\left[\begin{array}{cc}
1_n & S\\
0 & 1_n
\end{array}\right]\right)e(-\tau(\beta S))dS.
\]
By the definition of $f_s^{\dagger,\mathfrak{c}}$, the integrand vanishes unless $S\in S_n(\mathfrak{o})^{\mathfrak{c}}$ and the proposition easily follows.
\end{proof}

For a character $\chi$ of $F$ with conductor $\mathfrak{p^c}$, the local Gauss sum of $\chi$ is defined as
\begin{equation}
\label{7.2.2.3}
G(\chi)=\sum_{u\in\mathcal{O}/\mathfrak{p}^{\mathfrak{c}}\mathcal{O}}\chi(\nu(u))e\left(\frac{\tau(u)}{\varpi^{\mathfrak{c}}}\right).
\end{equation}

\begin{lem}
\label{lemma 7.3}
Consider
\begin{equation}
\label{7.2.2.4}
G(\chi;\beta,m)=q^{\mathfrak{c}\mathbf{d}_1\frac{m(m-1)}{2}}\sum_{u\in\mathrm{Mat}_m(\mathcal{O}/\mathfrak{p}^{\mathfrak{c}}\mathcal{O})}\chi(\nu(u))e\left(\frac{\tau(\beta u)}{\varpi^{\mathfrak{c}}}\right)
\end{equation}
for a matrix $\beta\in\mathrm{Mat}_m(\mathcal{O})$. Then 
\begin{equation}
\label{7.2.2.5}
G(\chi;\beta,m)=\left\{\begin{array}{cc}
\chi(\nu(\beta))G(\chi)^m & \beta\in\mathrm{GL}_m(\mathcal{O}),\\
0 & \text{otherwise.}
\end{array}\right..
\end{equation}
\end{lem}

\begin{proof}
This is an analogue of the computations in \cite[Section 6]{BS}. Multiply by some matrix of $\mathrm{GL}_m(\mathcal{O}_v)$ on the left and right of $\beta$, it suffices to prove the lemma for diagonal $\beta=\mathrm{diag}[\beta_1,...,\beta_m]$. In this case, we calculate that
\[
G(\chi;\beta,m)=\nu(\varpi^{\mathfrak{c}})^{\frac{m(m-1)}{2}}\prod_{i=1}^m\left(\sum_{u\in\mathcal{O}/\mathfrak{p}^{\mathfrak{c}}\mathcal{O}}\chi(\nu(u))e\left(\frac{\tau(\beta_iu)}{\varpi^{\mathfrak{c}}}\right)\right).
\] 
By the property of Gauss sums, the sum in the bracket is nonzero if and only if $\beta_i\in\mathcal{O}^{\times}$ and in this case it equals $\chi(\nu(\beta_i))G(\chi)$.
\end{proof}

We write $\beta\in S_n(F)$ as 
\begin{equation}
\label{7.2.2.6}
\beta=\left[\begin{array}{ccc}
\beta_1 & -\epsilon \beta^{\ast}_2 & -\epsilon \beta^{\ast}_3\\
\beta_2 & \beta_4 & -\epsilon \beta^{\ast}_5\\
\beta_3 & \beta_5 & \beta_6
\end{array}\right]
\end{equation}
with $\beta_1\in S_r(F), \beta_4,\beta_6\in S_m(F)$.

\begin{prop}
\label{proposition 7.4}
$E_{\beta}(1;f_s^{\ddagger,\mathfrak{c}})=0$ unless $\beta_5\in\mathrm{GL}_m(\mathcal{O})$. In this case, if we further assume $\nu(\beta)\neq 0$, then
\begin{equation}
\label{7.2.2.7}
E_{\beta}(1;f_s^{\ddagger,\mathfrak{c}})=\left\{\begin{array}{cc}
q^{\mathfrak{c}\mathbf{d}_1\frac{m(m-1)}{2}}\chi(\nu(\beta_5))G(\chi)^m & \beta\in S_n(\mathfrak{o})^{\ast,0},\\
0 & \text{ otherwise}.
\end{array}\right.
\end{equation}
\end{prop}

\begin{proof}
We need to compute
\[
\begin{aligned}
E_{\beta}(1;f_s^{\ddagger,\mathfrak{c}})&=\int_{S_n(F)}\sum_{u\in\mathrm{GL}_m(\mathcal{O})/\varpi^{\mathfrak{c}}\mathrm{GL}_m(\mathcal{O})}\chi^{-1}(\nu(u))e(-\tau(\beta S))\\
\times& f_s^{\dagger,0}\left(J_n\left[\begin{array}{cc}
1_n & S\\
0 & 1_n
\end{array}\right]\left[\begin{array}{cccccc}
1_r & 0 & 0 & 0 & 0 & 0\\
0 & 1_m & 0 & 0 & 0 & \frac{u}{\varpi^{\mathfrak{c}}}\\
0 & 0 & 1_m & 0 & -\frac{\epsilon u^{\ast}}{\varpi^{\mathfrak{c}}} & 0\\
0 & 0 & 0 & 1_r & 0 & 0\\
0 & 0 & 0 & 0 & 1_m & 0\\
0 & 0 & 0 & 0 & 0 & 1_m
\end{array}\right]\right)dS.
\end{aligned}
\]
Changing variables
\[
S\mapsto S-\left[\begin{array}{ccc}
0 & 0 & 0 \\
0 & 0 & \frac{u}{\varpi^{\mathfrak{c}}}\\
0 & -\frac{\epsilon u^{\ast}}{\varpi^{\mathfrak{c}}} & 0
\end{array}\right],
\]
we obtain
\[
\begin{aligned}
&\int_{S_n(F)}f_s^{\dagger,0}\left(J_n\left[\begin{array}{cc}
1_n & S\\
0 & 1_n
\end{array}\right]\right)e(-\tau(\beta S))dS\\
\times&\sum_{u\in\mathrm{GL}_m(\mathcal{O})/\varpi^{\mathfrak{c}}\mathrm{GL}_m(\mathcal{O})}\chi^{-1}(\nu(u))e\left(\frac{2\tau(\beta_5 u)}{\varpi^{\mathfrak{c}}}\right).
\end{aligned}
\]
The second line implies $\beta_5\in\mathrm{GL}_m(\mathcal{O})$ by Lemma \ref{lemma 7.3}. In this case and $\nu(\beta)\neq 0$ the integral in the first line can be simply calculated by the definition of $f_s^{\dagger,0}$.
\end{proof}

\begin{prop}
\label{proposition 7.5}
$E_{\beta}(1;f_s^{p})=0$ unless $\beta_5\in\mathrm{GL}_m(\mathcal{O})$. In this case, if we further assume $\nu(\beta)\neq 0$, then
\begin{equation}
\label{7.2.2.8}
E_{\beta}(1;f_s^p)=\left\{\begin{array}{cc}
1 & \beta\in S_n(\mathfrak{o})^{\ast,0},\\
0 &\text{ otherwise.}
\end{array}\right.
\end{equation}
\end{prop}

\begin{proof}
This is same as Proposition \ref{proposition 7.5} except, rather than a Gauss sum, we obtain a term
\[
\sum_{i=0}^m(-1)^iq^{\mathbf{d}_3\left(\frac{i(i-1)}{2}-im\right)}\sum_j\sum_{u\in\widetilde{\varpi}\mathrm{Mat}_m(\mathcal{O})\delta_{ij}^{-1}/\widetilde{\varpi}\mathrm{Mat}_m(\mathcal{O})}e(2\tau(\beta_5 u)),
\]
which is nonzero unless $\beta_5\in\mathrm{GL}_m(\mathcal{O})$. This can be shown by the property of exponential sums as in \cite[page 1412]{BS}.
\end{proof}

\subsection{Archimedean computations}
\label{section 7.3}

We now turn to the archimedean setting. We fix an archimedean place $v$ and omit it from the notation where we also abuse the notation by simply denoting $f_s^{\infty}:=f_{s,v}^{\infty}$. Fix a positive integer $l$ be our weight. Let $z=x+iyy^{\ast}\in\mathcal{H}_n$ and consider the local Fourier coefficients
\begin{equation}
\label{7.3.1}
E_{\beta}(z;f_s^{\infty})=\int_{S_n(\R)}f_s^{\infty}\left(J_n\left[\begin{array}{cc}
1_n & S\\
0 & 1_n
\end{array}\right]\left[\begin{array}{cc}
y & x\hat{y}\\
0 & \hat{y}
\end{array}\right]\right)e_{\infty}(-\tau(\beta S))dS.
\end{equation}
For $y,\beta\in S_n(\R)$ and $s_1,s_2\in\C$ we define a function $\xi_n$ by
\begin{equation}
\label{7.3.2}
\xi_n(y,\beta;s_1,s_2)=\int_{S_n(\R)}\nu(s+iy)^{-s_1}\nu(s-iy)^{-s_2}e_{\infty}(\tau(\beta S))dS.
\end{equation}
By definition of $f_s^{\infty}$, we have (for Case II, III, IV)
\begin{equation}
\label{7.3.3}
\begin{aligned}
&E_{\beta}(z,f_s^{\infty})\\
=&\int_{S_n(\R)}\nu(yi+x\hat{y}+S\hat{y})^{-l}|\nu(yi+x\hat{y}+S\hat{y})|^{l-s-\kappa}e_{\infty}(-\tau(\beta S))dS\\
=&e_{\infty}(\tau(\beta x))\nu(y)^{s+\kappa}\int_{S_n(\R)}\nu(S+yy^{\ast}i)^{-\frac{s+\kappa+l}{2}}\nu(S-yy^{\ast}i)^{-\frac{s+\kappa-l}{2}}e_{\infty}(-\tau(\beta S))dS\\
=&e_{\infty}(\tau(\beta x))\nu(y)^{s+\kappa}\xi_n\left(yy^{\ast},\beta;\frac{s+\kappa+l}{2},\frac{s+\kappa-l}{2}\right).
\end{aligned}
\end{equation}

Similarly, for Case V we have
\begin{equation}
\label{7.3.4}
E_{\beta}(z,f_s^{\infty})=e_{\infty}(\tau(\beta x))\nu(y^{\ast})^{l}|\nu(y^{\ast})|^{2s+2\kappa-l}\xi_n\left(yy^{\ast},\beta;s+\kappa+\frac{l}{2},s+\kappa-\frac{l}{2}\right).
\end{equation}

Recall that in Case III, the symmetric space is same as the one for Case II. Denote $\xi_n^{\mathrm{II}},\xi_n^{\mathrm{III}}$ to indicate above functions $\xi_n$ in two cases. After identify $\beta,yy^{\ast}$ with their image $\beta',(yy^{\ast})'$ in $\{\beta'\in\mathrm{GL}_{2n}(\R):\transpose{\beta}=\beta\}$ we have
\begin{equation}
\label{7.3.5}
\xi^{\mathrm{III}}_n\left(yy^{\ast},\beta;\frac{s+\kappa+l}{2},\frac{s+\kappa-l}{2}\right)=\xi^{\mathrm{II}}_{2n}\left((yy^{\ast})',\beta';s+\kappa+l,s+\kappa-l\right).
\end{equation}
The function $\xi_n$ is the confluent hypergeometric function studied in \cite{Sh82}. We record some of its property in the following lemma.

\begin{lem}
\label{lemma 7.6}
Let $t=\mathrm{rank}(\beta)$ be the rank of $\beta$ and $t_+$ (resp. $t_-$) the number of positive (resp. negative) eigenvalues of $\beta$. Then
\begin{equation}
\label{7.3.6}
\xi(y,\beta;s_1,s_2)=\boldsymbol{\Gamma}(s_1,s_2)\times\omega(y,\beta;s_1,s_2)
\end{equation}
where $\omega(y,\beta;s_1,s_2)$ is a holomorphic function in $s_1,s_2$ and $\boldsymbol{\Gamma}(s_1,s_2)$ is given by the following list. \\
(Case II, Symplectic)
\[
\frac{\prod_{i=0}^{n-t-1}\Gamma\left(s_1+s_2-\frac{n+1+i}{2}\right)}{\prod_{i=0}^{n-t_--1}\Gamma\left(s_1-\frac{i}{2}\right)\prod_{i=0}^{n-t_+-1}\Gamma\left(s_2-\frac{i}{2}\right)},
\]
(Case IV, Quaternionic Unitary)
\[
\frac{\prod_{i=0}^{n-t-1}\Gamma\left(2s_1+2s_2-2n+1-2i\right)}{\prod_{i=0}^{n-t_--1}\Gamma\left(2s_1-2i\right)\prod_{i=0}^{n-t_+-1}\Gamma\left(2s_2-2i\right)},
\]
(Case V, Unitary)
\[
\frac{\prod_{i=0}^{n-t-1}\Gamma\left(s_1+s_2-n-i\right)}{\prod_{i=0}^{n-t_--1}\Gamma\left(s_1-i\right)\prod_{i=0}^{n-t_+-1}\Gamma\left(s_2-i\right)}.
\]
In particular, if $\beta>0$, then 
\begin{equation}
\label{7.3.7}
\omega(y,\beta;l,0)=2^ni^{-nl}\pi^{nl-\frac{\iota n(n-1)}{4}}\nu(2\beta)^{l-\kappa}e_{\infty}(i\tau(\beta y)),\text{ with }\iota=\left\{\begin{array}{cc}
1 & \text{ Case II},\\
4 & \text{ Case IV},\\
2 & \text{ Case V}. 
\end{array}\right.
\end{equation}
\end{lem}

We are interested in the special value 
\begin{equation}
\label{7.3.8}
s=s_0:=\left\{\begin{array}{cc}
l-\kappa & \text{ Case II, III, IV,}\\
\frac{l}{2}-\kappa & \text{ Case V}.
\end{array}\right.\text{ with }l\geq\left\{\begin{array}{cc}
m+1 & \text{ Case II,}\\
n+1 & \text{ Case III, IV, V.}\\
\end{array}\right..
\end{equation}
In this case we need to consider the special value of $\xi(y,\beta;s_1,s_2)$ at $s_1=l,s_2=0$. 

\begin{lem}
\label{lemma 7.7}
The function $\xi(y,\beta;s_1,s_2)$ does not have a zero at $s_1=l,s_2=0$ only if $\beta>0$.
\end{lem}

\begin{proof}
We prove for Case IV and omit the same proof for Case II, V. Consider
\[
\frac{\prod_{i=0}^{n-t-1}\Gamma\left(2l-2n+1-2i\right)}{\prod_{i=0}^{n-t_--1}\Gamma\left(2l-2i\right)\prod_{i=0}^{n-t_+-1}\Gamma\left(-2i\right)}.
\]
and calculate the contributions of poles for each terms. By our assumption on $l$, $\prod_{i=0}^{n-t_--1}\Gamma(2l-2i)$ does not have poles. The numerator at most contributes $n-t$ poles while the denominator always has $n-t_+$ poles thus we must have $t=t_+$.
\end{proof}
 
We summarize the archimedean Fourier coefficients in the following proposition.

\begin{prop}
\label{proposition 7.8}
As a function of $s$, $E_{\beta}(z;f_s^{\infty})$ does not have a zero at $s_0$ only if $\beta>0$. In this case, its value at $s=s_0$ is given by the following list.\\
(Case II, Symplectic)
\[
\frac{2^ni^{-nl}\pi^{nl-\frac{n(n-1)}{4}}}{\prod_{i=0}^{n-1}\Gamma\left(l-\frac{i}{2}\right)}\nu(2\beta)^{l-\frac{n+1}{2}}\nu(y)^le_{\infty}(\tau(\beta z)),
\]
(Case III, Quaternionic Orthogonal)
\[
\frac{2^{2n}(-1)^{-nl}\pi^{2nl-\frac{n(2n-1)}{2}}}{\prod_{i=0}^{2n-1}\Gamma\left(l-\frac{i}{2}\right)}\nu(2\beta)^{l-\frac{2n+1}{2}}\nu(y)^le_{\infty}(\tau(\beta z)),
\]
(Case IV, Quaternionic Unitary)
\[
\frac{2^ni^{-nl}\pi^{nl-n(n-1)}}{\prod_{i=0}^{n-1}\Gamma(2l-2i)}\nu(2\beta)^{l-\frac{2n-1}{2}}\nu(y)^le_{\infty}(\tau(\beta z)),
\]
(Case V, Unitary)
\[
\frac{2^ni^{-nl}\pi^{nl-\frac{n(n-1)}{2}}}{\prod_{i=0}^{n-1}\Gamma(l-i)}\nu(2\beta)^{l-\frac{n}{2}}\nu(y^{\ast})^le_{\infty}(\tau(\beta z)).
\]
\end{prop}

\subsection{The global Fourier expansion}
\label{section 7.4}

We now study the global Fourier expansion of $\mathbb{E}(h;f_s,\chi,\boldsymbol{n})$ defined in \eqref{6.11}. From now on, we assume the weight $\boldsymbol{l}=(l_v)_{v|\infty}$ is parallel. That is $l_v=l$ for all $v|\infty$. We also assume
\begin{equation}
\label{7.4.1}
\begin{aligned}
l&\geq\left\{\begin{array}{cc}
m+1 & \text{ Case II,}\\
n+1 & \text{ Case III, IV, V.}\\
\end{array}\right.\text{ when }F\neq\Q,\\
l&\geq\left\{\begin{array}{cc}
m+1 & \text{ Case II,}\\
n+r+1 & \text{ Case III, IV, V,}\\
\end{array}\right.\text{ when }F=\Q.
\end{aligned}
\end{equation}
For such $l$, we are interested in the special value at
\begin{equation}
\label{7.4.2}
s=s_0:=\left\{\begin{array}{cc}
l-\kappa & \text{ Case II, III, IV,}\\
\frac{l}{2}-\kappa & \text{ Case V}.
\end{array}\right.
\end{equation}
Let $q_z\in G(\mathbb{A})$ be an element such that
\begin{equation}
\label{7.4.3}
q_z=\left\{\begin{array}{cc}
\mathrm{diag}[a_v,\hat{a}_v] & v\nmid\mathfrak{n}\boldsymbol{p}\infty,\\
1 & v|\mathfrak{n}\boldsymbol{p},\\
\left[\begin{array}{cc}
y_v & x_v\hat{y}_v\\
0 & \hat{y}_v
\end{array}\right] & v|\infty,
\end{array}\right.
\end{equation}
for any $z=(z_v)_{v|\infty}$ with $z_v=x_v+y_vy_v^{\ast}i\in\mathcal{H}_n$. Denote
\begin{equation}
\label{7.4.4}
\boldsymbol{S}=S_n(F)\cap\prod_{v\nmid\mathfrak{n}\infty}\hat{a}S_n(\mathfrak{o}_v)^{\ast}a^{-1}\cdot\prod_{v|\mathfrak{n}_2\boldsymbol{p}}S_n(\mathfrak{o}_v)^{\ast,0}\cdot\prod_{v|\mathfrak{n}_1}S_n(\mathfrak{o}_v)^{\ast,\mathfrak{c}_v}
\end{equation}
and
\begin{equation}
\label{sp}
\boldsymbol{S}^{\boldsymbol{p}}=\left\{\beta=\left[\begin{array}{ccc}
\beta_1 & -\epsilon\beta_2^{\ast} & -\epsilon\beta_3^{\ast}\\
\beta_2 & \beta_4 & -\epsilon\beta_5^{\ast}\\
\beta_3 & \beta_5 & \beta_6
\end{array}\right]\in\boldsymbol{S}:\begin{array}{c}
\beta_2,\beta_3\in\mathrm{Mat}_{m,r}(\boldsymbol{p}^{\boldsymbol{n}-1}\mathcal{O}_{\boldsymbol{p}}),\\
\beta_4,\beta_6\in\mathrm{Mat}_m(\boldsymbol{p}^{2\boldsymbol{n}-2}\mathcal{O}_{\boldsymbol{p}}).
\end{array}
\right\}.
\end{equation}

For a Hecke character $\chi=\prod_v\chi_v$, we define the Gauss sum $G(\chi)=\prod_vG(\chi_v)$ with $G(\chi_v)$ the local Gauss sum defined in \eqref{7.2.2.3}. We may write $G^D(\chi)$ to indicate that the Gauss sum is defined for $D$. In Case II, IV, V, we always omit the superscript `$D$' for simplicity as no confusion will occur. In Case III, we will need both $G^D(\chi)$ and $G^F(\chi)$ and we make the convention $G(\chi):=G^D(\chi)$.

Combining with the local computations in previous subsections, we summarize the Fourier expansion of $\mathbb{E}(h;f_s,\chi,\boldsymbol{n})$ in the following proposition.

\begin{prop}
\label{proposition 7.10}
At $s=s_0$, the Eisenstein series $\mathbb{E}(h;f_s,\chi,\boldsymbol{n})$ has a Fourier expansion of the form
\begin{equation}
\label{7.4.7}
\mathbb{E}(q_z;f_s,\chi,\boldsymbol{n})=\mathfrak{D}C(q_z)\nu(y'^{\ast})^l\sum_{0<\beta\in\boldsymbol{S}^{\boldsymbol{p}}}C(\beta,\chi)\cdot\mathbb{E}(\beta;\chi)e_{\infty}(\tau(\beta z')).
\end{equation}
Here:\\
(1) $\mathfrak{D}$ is a power of the discriminant of $D$, \\
(2) $C(q_z)=\prod_{v\nmid\mathfrak{n}\infty}\chi(\nu(a_v))|N_{E/F}(\nu(a_v))|^{s_0+\kappa}$,\\
(3) $z'=x+iy'y'^{\ast}$ with $y'=\mathrm{diag}[1_r,\boldsymbol{\varpi}^{1-\boldsymbol{n}}\cdot 1_m,\boldsymbol{\varpi}^{1-\boldsymbol{n}}\cdot 1_m]y.$\\
(4) The constant $C(\beta,\chi)$ is given by
\begin{equation}
C(\beta,\chi)=\prod_{v|\mathfrak{n}_2}|\varpi_v|^{-\mathfrak{c}\mathbf{d}_1\frac{m(m-1)}{2}}\cdot|\boldsymbol{\varpi}|^{-\boldsymbol{c}\mathbf{d_1}\frac{m(m-1)}{2}}G(\chi)^m\chi(\nu(\beta_5))\nu(2\beta)^{l-\kappa}
\end{equation}
with $\mathbf{d}_1$ as in \eqref{4.3.5} and $\chi(\nu(\beta_5))$ is understood as zero if $\nu(\beta_5)\notin\mathcal{O}^{\times}_{\mathfrak{n}_2\boldsymbol{p}}$,\\
(5) $\mathbb{E}(\beta;\chi)$ is given by the following list with $d(F)=[F:\Q]$\\
(Case II, Symplectic)
\[
\left(\frac{2^ni^{-nl}\pi^{nl-\frac{n(n-1)}{4}}}{\prod_{i=0}^{n-1}\Gamma\left(l-\frac{i}{2}\right)}\right)^{d(F)}\prod_{v\nmid\mathfrak{n}}L_{v}\left(s_0+\frac{1}{2},\chi\lambda_{\beta}\right)\cdot\prod_{v\nmid\mathfrak{n}}P_{a_v^{\ast}\beta a_v}(\chi(q_v)q_v^{-l}),
\]
(Case III, Quaternionic Orthogonal with $r=0$)
\[
\left(\frac{2^{2n}(-1)^{-nl}\pi^{2nl-\frac{n(2n-1)}{2}}}{\prod_{i=0}^{2n-1}\Gamma\left(l-\frac{i}{2}\right)}\right)^{d(F)}\prod_{v\nmid\mathfrak{n}}L_v\left(s_0+\frac{1}{2},\chi\lambda_{\beta}\right)\cdot\prod_{v\nmid\mathfrak{n}}P_{a_v^{\ast}\beta a_v}(\chi(q_v)q_v^{-l}),
\]
(Case III, Quaternionic Orthogonal with $r=1$)
\[
\left(\frac{2^{2n}(-1)^{-nl}\pi^{2nl-\frac{n(2n-1)}{2}}}{\prod_{i=0}^{2n-1}\Gamma\left(l-\frac{i}{2}\right)}\right)^{d(F)}\prod_{\substack{v\nmid\mathfrak{n}\\v\text{split}}}L_v\left(s_0+\frac{1}{2},\chi\lambda_{\beta}\right)\prod_{v\nmid\mathfrak{n}}P_{a_v^{\ast}\beta a_v}(\chi(q_v)q_v^{-l}),
\]
(Case IV, Quaternionic Unitary)
\[
\left(\frac{2^ni^{-nl}\pi^{nl-n(n-1)}}{\prod_{i=0}^{n-1}\Gamma(2l-2i)}\right)^{d(F)}\prod_{v\nmid\mathfrak{n}}P_{a_v^{\ast}\beta a_v}(\chi(q_v)q_v^{-l}),
\]
(Case V, Unitary)
\[
\left(\frac{2^ni^{-nl}\pi^{nl-\frac{n(n-1)}{2}}}{\prod_{i=0}^{n-1}\Gamma(l-i)}\right)^{d(F)}\prod_{v\nmid\mathfrak{n}}P_{a_v^{\ast}\beta a_v}(\chi^0(q_v)q_v^{-l}).
\]
\end{prop}

\begin{proof}
We first show that only $\beta\in\boldsymbol{S}^{\boldsymbol{p}}$ can contribute a nonzero Fourier coefficient. Recall that the Eisenstein series $\mathbb{E}(h;f_s)$ is defined as
\[
\mathbb{E}(h;f_s):=E(h;f_s)|\boldsymbol{U}(\boldsymbol{p}^{\boldsymbol{n}-1}):=\sum_{\gamma}E(h\gamma;f_s)
\]
with $\gamma$ running through elements of the form in \eqref{6.12}. To ease the notation, we write
\[
A=\left[\begin{array}{ccc}
1_r & 0 & 0\\
\epsilon b_2^{\ast}\theta^{-1} & \boldsymbol{\varpi}^{\boldsymbol{n}-1}\cdot 1_m & 0\\
-\epsilon b_1^{\ast}\theta^{-1} & 0 & \boldsymbol{\varpi}^{\boldsymbol{n}-1}\cdot 1_m
\end{array}\right],\,B=\left[\begin{array}{ccc}
0 & \frac{\epsilon\boldsymbol{\varpi}^{1-\boldsymbol{n}}b_2}{2} & -\frac{\epsilon\boldsymbol{\varpi}^{1-\boldsymbol{n}}b_1}{2}\\
-\frac{b_2^{\ast}}{2} & -\boldsymbol{\varpi}^{1-\boldsymbol{n}}c_2 & 0\\
-\frac{b_1^{\ast}}{2} & 0 & \boldsymbol{\varpi}^{1-\boldsymbol{n}}c_1
\end{array}\right].
\]
In the following, we write $\gamma_{\boldsymbol{p}},A_{\boldsymbol{p}},B_{\boldsymbol{p}}$ to indicate they are matrices with entries in $F_{\boldsymbol{p}}$. By straightforward computations,
\[
\begin{aligned}
\mathbb{E}_{\beta}(q_z;f_s)&=\int_{S_n(F)\backslash S_n(\mathbb{A})}\sum_{\gamma_{\boldsymbol{p}}}E\left(\left[\begin{array}{cc}
1_n & S\\
0 & 1_n
\end{array}\right]q_z\gamma_{\boldsymbol{p}};f_s\right)e_{\mathbb{A}}(-\tau(\beta S))dS\\
&=\int_{S_n(F)\backslash S_n(\mathbb{A})}E\left(\left[\begin{array}{cc}
1_n & S\\
0 & 1_n
\end{array}\right]q_{z'};f_s\right)e_{\mathbb{A}}(-\tau(\beta S))dS\cdot\sum_{\gamma_{\boldsymbol{p}}}e_{\mathbb{A}}(\tau(\beta A_{\boldsymbol{p}}^{-1}B_{\boldsymbol{p}})).
\end{aligned}
\]
The integral is the Fourier coefficients of $E(h;f_s)$ calculated in previous subsection. Hence, by Proposition \ref{proposition 7.1}, Proposition \ref{proposition 7.2}, Proposition \ref{proposition 7.4}, Proposition \ref{proposition 7.5}, this integral is nonzero unless $\beta\in\boldsymbol{S}$. The exponential sum is nonzero unless further $\beta\in\boldsymbol{S}^{\boldsymbol{p}}$.

Secondly, we show that only $\beta>0$ can contribute a nonzero Fourier coefficient. Note that the condition $\beta\in\boldsymbol{S}^{\boldsymbol{p}}$ and Proposition \ref{proposition 7.4}, Proposition \ref{proposition 7.5} imply that only such $\beta$ with $\mathrm{rank}(\beta)\geq 2m$ can contribute a nonzero term. Under our assumptions on $l$, the $L$-functions occurring form Proposition \ref{proposition 7.1} does not provide any poles. Therefore, the Fourier coefficients are nonvanishing if and only if the confluent hypergeometric function in archimedean computations does not have zeros. Then by Lemma \ref{lemma 7.7}, when specializing to $s=s_0$, $\mathbb{E}_{\beta}(q_z;f_s)\neq 0$ unless $\beta>0$. 

The proposition then follows from the explicit formulas in Proposition \ref{proposition 7.1}, Proposition \ref{proposition 7.2}, Proposition \ref{proposition 7.4}, Proposition \ref{proposition 7.5} and Proposition \ref{proposition 7.8}.
\end{proof}

\begin{rem}
\label{remark 7.11}
We emphasize that we indeed obtain a better bound on $l$ due the action of $\boldsymbol{U}(\boldsymbol{p})$ operator. Without such process, we have to consider $\beta$ of all rank so that one need to assume $l\geq 2m+1$ in Case II and $2n+1$ in Case III, IV, V to avoid the occurrence of the poles in $L$-functions from unramified computations. Also note that when $r=0$ or when $\boldsymbol{p}$ splits in Case III, we always have $\nu(\beta)$ is a square mod $\boldsymbol{p}$. This property is essential in the construction of $p$-adic $L$-functions for Case II, III (see also \cite[Section 3.5]{LZ20}). 
\end{rem}

From these explicit formulas of the Fourier coefficients, we immediately have

\begin{cor}
\label{corollary 7.12}
Up to a power of $\pi$, $\mathbb{E}(h;f_s,\chi,\boldsymbol{n})$ is an algebraic modular form on $H(\mathbb{A})$ at $s=s_0$. More precisely,  we have
\begin{equation}
\label{7.4.8}
\frac{\mathbb{E}(\beta,\chi)}{\pi^{d(F)d(\pi)}}\in\overline{\Q}\qquad\text{ with }d(\pi)=\left\{\begin{array}{cc}
nl-\frac{n(n-1)}{4}+l-\frac{n}{2} & \text{ Case II,}\\
2nl-\frac{n(2n-1)}{2}+l-n & \text{ Case III,}\\
nl-n(n-1) & \text{ Case IV,}\\
nl-\frac{n(n-1)}{2} & \text{ Case V.}
\end{array}\right.
\end{equation}
Furthermore, for any $\sigma\in\mathrm{Aut}(\C/F)$, we have\\
(Case II, Symplectic)
\[
\left(\frac{\mathbb{E}(\beta;\chi)}{\pi^{d(F)\left(nl-\frac{n(n-1)}{4}\right)}(\pi i)^{d(F)\left(l-\frac{n}{2}\right)}G(\chi)}\right)^{\sigma}=\frac{\mathbb{E}(\beta;\chi^{\sigma})}{\pi^{d(F)\left(nl-\frac{n(n-1)}{4}\right)}(\pi i)^{d(F)\left(l-\frac{n}{2}\right)}G(\chi^{\sigma})},
\]
(Case III, Quaternionic Orthognoal) In this case we denote $G^F(\chi)$ to indicate that the Gauss sum is defined for $F$.
\[
\left(\frac{\mathbb{E}(\beta;\chi)}{\pi^{d(F)\left(2nl-\frac{n(2n-1)}{2}\right)}(\pi i)^{d(F)(l-n)}G^F(\chi)}\right)^{\sigma}=\frac{\mathbb{E}(\beta;\chi^{\sigma})}{\pi^{d(F)\left(2nl-\frac{n(2n-1)}{2}\right)}(\pi i)^{d(F)(l-n)}G^F(\chi^{\sigma})},
\]
(Case IV, Quaternionic Unitary)
\[
\left(\frac{\mathbb{E}(\beta;\chi)}{\pi^{d(F)\left(nl-n(n-1)\right)}}\right)^{\sigma}=\frac{\mathbb{E}(\beta;\chi^{\sigma})}{\pi^{d(F)\left(nl-n(n-1)\right)}},
\]
(Case V, Unitary)
\[
\left(\frac{\mathbb{E}(\beta;\chi)}{\pi^{d(F)\left(nl-\frac{n(n-1)}{2}\right)}}\right)^{\sigma}=\frac{\mathbb{E}(\beta;\chi^{\sigma})}{\pi^{d(F)\left(nl-\frac{n(n-1)}{2}\right)}}.
\]
\end{cor}

\section{Algebraicity and the $p$-adic $L$-function}
\label{section 8}

In this final section, we apply our computations to study the properties of special $L$-values. Let $\boldsymbol{l}=(l,...,l)$ be a parallel weight satisfying
\begin{equation}
\label{8.1}
\begin{aligned}
l&\geq\left\{\begin{array}{cc}
m+1 & \text{ Case II,}\\
n+1 & \text{ Case III, IV, V.}\\
\end{array}\right.\text{ when }F\neq\Q,\\
l&\geq\left\{\begin{array}{cc}
m+1 & \text{ Case II,}\\
n+r+1 & \text{ Case III, IV, V,}\\
\end{array}\right.\text{ when }F=\Q.
\end{aligned}
\end{equation}
Fix a specific prime $\boldsymbol{p}$ of $\mathfrak{o}$ and an integral ideal $\mathfrak{n}=\mathfrak{n}_1\mathfrak{n}_2=\prod_{v}\mathfrak{p}_v^{\mathfrak{c}_v}$ with $\mathfrak{n}_1,\mathfrak{n}_2,\boldsymbol{p}$ coprime. Denote $\boldsymbol{\varpi}$ for the uniformizer of $\boldsymbol{p}$. We make the following assumptions:\\
(1) $2\in\mathcal{O}_v^{\times}$ and $\theta\in\mathrm{GL}_r(\mathcal{O}_v)$ for all $v|\mathfrak{n}\boldsymbol{p}$.\\
(2) $\boldsymbol{f}\in\mathcal{S}_{\boldsymbol{l}}(K(\mathfrak{n}\boldsymbol{p}),\overline{\Q})$ is an algebraic eigenform for the Hecke algebra $\mathcal{H}(K(\mathfrak{n}\boldsymbol{p}),\mathfrak{X})$ as in Section \ref{section 3.4}.\\
(3) $\boldsymbol{f}$ is an eigenform for the $U(\boldsymbol{p})$ operator with eigenvalue $\alpha(\boldsymbol{p})\neq 0$.\\
(4) $\chi=\chi_1\boldsymbol{\chi}$ with $\chi_1$ has conductor $\mathfrak{n}_2$ and $\boldsymbol{\chi}$ has conductor $\boldsymbol{p}^{\boldsymbol{c}}$ for some integer $\boldsymbol{c}\geq 0$. We assume $\chi$ has infinity type $\boldsymbol{l}$. That is, $\chi_v(x)=x^{l}|x|^{-l}$ for all $v|\infty$.\\
(5) In Case V, all places $v|\mathfrak{n}\boldsymbol{p}$ are nonsplit in $\mathcal{O}$.

We study the special values of $L$-functions $L(s,\boldsymbol{f}\times\chi)$ at 
\begin{equation}
\label{8.2}
s=s_0:=\left\{\begin{array}{cc}
l-\kappa & \text{ Case II, III, IV,}\\
\frac{l}{2}-\kappa & \text{ Case V}.
\end{array}\right.
\end{equation}

\subsection{The algebraicity of special $L$-values}

The following algebraic result is also studied in \cite{BS, Sh00} for symplectic and unitary groups and in \cite{JYB1} for quaternionic unitary groups. Comparing with our previous work in \cite{JYB1}, here we obtain a better bound $l\geq n+r+1$ rather than $l\geq 2n+1$.

\begin{thm}
\label{theorem 8.1}
Let $l$ and $s_0$ as in \eqref{8.1}, \eqref{8.2}. Then
\begin{equation}
\label{8.3}
\begin{aligned}
&\frac{L\left(s_0+\frac{1}{2},\boldsymbol{f}\times\chi\right)}{\pi^{d(F)\boldsymbol{d}(\pi)}\Omega\cdot\langle\boldsymbol{f},\boldsymbol{f}\rangle}\in\overline{\Q},\qquad & \text{ if }\boldsymbol{c}>0,\\
&\frac{L\left(s_0+\frac{1}{2},\boldsymbol{f}\times\chi\right)M\left(s_0+\frac{1}{2},\boldsymbol{f}\times\chi\right)}{\pi^{d(F)\boldsymbol{d}(\pi)}\Omega\cdot\langle\boldsymbol{f},\boldsymbol{f}\rangle}\in\overline{\Q},\qquad & \text{ if }\boldsymbol{c}=0.
\end{aligned}
\end{equation}
with
\begin{equation}
\label{8.4}
\boldsymbol{d}(\pi)=\left\{\begin{array}{cc}
nl-\frac{3m^2}{2}+l-\frac{n}{2} & \text{ Case II,}\\
2nl-\frac{3n^2}{2}+l-n & \text{ Case III,}\\
nl-\frac{3}{2}n(n-1) & \text{ Case IV,}\\
nl-\frac{n(n-1)}{2}-m(m+r) & \text{ Case V.}
\end{array}\right.
\end{equation}
Here $\Omega=1$ in Case II, III, IV and in Case V, $\Omega\in\C^{\times}$ is the following CM period
\begin{equation}
\label{cmperiod}
\Omega=p_E\left(l\tau,r\tau_v\right),
\end{equation}
where $(E,\tau)$ is a fixed CM type and $p_E$ is the period notation given in \cite[Section 11.3]{Sh00}.
\end{thm}

\begin{proof}
The proof is similar to \cite[Appendix]{BS} and \cite{JYB1, Sh00}. We remark that in \cite{JYB1, Sh00}, one needs to use the fact that the space of algebraic modular forms is a direct sum of space of algebraic cusp forms and algebraic Eisenstein series. This result is proved in \cite{H84} when the Eisenstein series is absolutely convergent at $s=s_0$ which forces $l\geq n$ in Case I and $l\geq 2n+1$ in Case II, III, IV. This result is not necessary and not used in \cite[Appendix]{BS} so that the special value below the absolutely convergence bound can be considered. However, the proof there need the assumption that the eigenvalue $\alpha(\boldsymbol{p})$ of the $U(\boldsymbol{p})$ operator for $\boldsymbol{f}$ is nonzero as we made here. We sketch the proof following \cite[Appendix]{BS}.

Let $\{\boldsymbol{f}_i\}$ be an orthogonal basis of $\mathcal{S}_{\boldsymbol{l}}(K(\mathfrak{n}\boldsymbol{p}))$ consisting of eigenforms of the Hecke algebra $\mathcal{H}(K(\mathfrak{n}\boldsymbol{p}),\mathfrak{X})$, which without losing generality we assume $\boldsymbol{f}_1=\boldsymbol{f}$. Take $\{\boldsymbol{h}_i\}$ be a basis of the orthogonal complement of $\mathcal{S}_{\boldsymbol{l}}(K(\mathfrak{n}\boldsymbol{p}))$ in $\mathcal{M}_{\boldsymbol{l}}(K(\mathfrak{n}\boldsymbol{p}))$. Denote $\boldsymbol{f}_i^1,\boldsymbol{f}_i^2$ (resp. $\boldsymbol{h}_i^1,\boldsymbol{h}_i^2$) such that $\boldsymbol{f}_i^1(g)=\boldsymbol{f}_i(g\eta_1\eta_{\boldsymbol{p}})$ (resp. $\boldsymbol{h}_i^1(g)=\boldsymbol{h}_i(g\eta_1\eta_{\boldsymbol{p}})$)  and $\boldsymbol{f}_i^2(g)=\boldsymbol{f}_i(g\eta_2)$ (resp. $\boldsymbol{h}_i^2(g)=\boldsymbol{h}_i(g\eta_2)$) with $\eta_1,\eta_2$ in \eqref{6.1} and $\eta_{\boldsymbol{p}}$ in \eqref{6.15}. Let $\boldsymbol{V}$ be the operator defined by $\boldsymbol{f}|\boldsymbol{V}:=\pi(\eta)\boldsymbol{f}|U'(\mathfrak{n}_1)$ and we use the superscript $c$ to mean $\boldsymbol{f}^c(g):=\overline{\boldsymbol{f}(g^{\iota})}$. 

We can write the Eisenstein series $\boldsymbol{E}(g_1,g_2;f_s)$ as
\begin{equation}
\label{I}
\begin{aligned}
\frac{\boldsymbol{E}(g_1,g_2;f_s)}{\pi^{d(F)d(\pi)}}&=\sum_{i,j}a_{ij}\boldsymbol{f}_i^1(g_1)\boldsymbol{f}_j^{2,c}|\boldsymbol{V}(g_2)+\sum_{i,j}b_{ij}\boldsymbol{h}_i^1(g_1)\boldsymbol{h}_j^{2,c}(g_2)\\
&+\sum_{i,j}c_{ij}\boldsymbol{f}_i^1(g_1)\boldsymbol{h}_j^{2,c}(g_2)+\sum_{i,j}d_{ij}\boldsymbol{h}_i^1(g_1)\boldsymbol{f}_j^{2,c}|\boldsymbol{V}(g_2).
\end{aligned}
\end{equation}

We take the Petersson inner product on both sides of \eqref{I} with $\boldsymbol{f}_i^1$ for the first variable. Then the integral representation \eqref{6.19} shows that
\[
\begin{aligned}
&\frac{\boldsymbol{Z}(s;\boldsymbol{f}_i,f_s)}{\pi^{d(F)d(\pi)}\langle\boldsymbol{f}^2_i|\boldsymbol{V},\boldsymbol{f}^2_i\rangle}\boldsymbol{f}_i^{2,c}|\boldsymbol{V}(g_2)\\
=&\sum_{j}a_{ij}\langle\boldsymbol{f}_i,\boldsymbol{f}_i\rangle\boldsymbol{f}_j^{2,c}|\boldsymbol{V}(g_2)+\sum_{j}c_{ij}\langle\boldsymbol{f}_i,\boldsymbol{f}_i\rangle\boldsymbol{h}_j^{2,c}(g_2).
\end{aligned}
\]
Clearly, we have $a_{ij}=0$ if $j\neq i$ and $c_{ij}=0$ for all $j$. Similarly taking the Petersson inner product on both sides of \eqref{I} with $\boldsymbol{f}_j^{2,c}|\boldsymbol{V}$ for the second variable, we conclude that $c_{ij}=d_{ij}=0$ for all $i,j$ and $a_{ij}\neq0$ unless $i=j$ in which case
\[
a_{ii}=\frac{\boldsymbol{Z}(s;\boldsymbol{f}_i,f_s)}{\pi^{d(F)d(\pi)}\langle\boldsymbol{f}_i,\boldsymbol{f}_i\rangle\langle\boldsymbol{f}^2_i|\boldsymbol{V},\boldsymbol{f}^2_i\rangle}.
\] 
Hence we can write
\begin{equation}
\label{Ia}
\frac{\boldsymbol{E}(g_1,g_2;f_s)}{\pi^{d(F)d(\pi)}}=\sum_ia_{ii}\boldsymbol{f}_i^1(g_1)\boldsymbol{f}_i^{2,c}|\boldsymbol{V}(g_2)+\sum_{i,j}b_{ij}\boldsymbol{h}_i^1(g_1)\boldsymbol{h}_j^{2,c}(g_2).
\end{equation}

Applying $\sigma\in\mathrm{Aut}(\C/\overline{\Q})$ on both sides of \eqref{Ia}, we have
\begin{equation}
\label{II}
\begin{aligned}
\left(\frac{\boldsymbol{E}(g_1,g_2;f_s)}{\pi^{d(F)d(\pi)}}\right)^{\sigma}&=\sum_ia_{ii}^{\sigma}\boldsymbol{f}_i^{1,\sigma}(g_1)(\boldsymbol{f}_i^{2,c}|\boldsymbol{V})^{\sigma}(g_2)+\sum_{i,j}b^{\sigma}_{ij}\boldsymbol{h}_i^{1,\sigma}(g_1)\boldsymbol{h}_i^{2,c,\sigma}(g_2).
\end{aligned}
\end{equation}

We now take the Petersson inner product on both sides of \eqref{II} with $\boldsymbol{f}_1^{1,\sigma}$ for the first variable $g_1$. For the left hand side, by Corollary \ref{corollary 7.12}, we have
\[
\left(\frac{\boldsymbol{E}(g_1,g_2;f_s)}{\pi^{d(F)d(\pi)}}\right)^{\sigma}=\frac{\boldsymbol{E}(g_1,g_2;f_s)}{\pi^{d(F)d(\pi)}},
\]
and the integral representation \eqref{6.19} shows that
\begin{equation}
\label{III}
\left\langle\frac{\boldsymbol{E}(g_1,g_2;f_s)}{\pi^{d(F)d(\pi)}},\boldsymbol{f}_1^{1,\sigma}(g_1)\right\rangle=\frac{\boldsymbol{Z}(s,\boldsymbol{f}^{\sigma},f_s)}{\pi^{d(F)d(\pi)}\langle\boldsymbol{f}_1^{2,\sigma}|\boldsymbol{V},\boldsymbol{f}_1^{2,\sigma}\rangle}\cdot\left(\boldsymbol{f}_1^{2,\sigma}|\boldsymbol{V}\right)^c(g_2).
\end{equation}
For the right hand side, we obtain
\begin{equation}
\label{IV}
\begin{aligned}
&\left(\frac{\boldsymbol{Z}(s;\boldsymbol{f},f_s)}{\pi^{d(F)d(\pi)}\langle\boldsymbol{f}^{\sigma},\boldsymbol{f}^{\sigma}\rangle\langle\boldsymbol{f}_1^{2,\sigma}|\boldsymbol{V},\boldsymbol{f}_1^{2,\sigma}\rangle}\right)^{\sigma}\langle\boldsymbol{f}^{\sigma},\boldsymbol{f}^{\sigma}\rangle\left(\boldsymbol{f}_1^{2,c}|\boldsymbol{V}\right)^{\sigma}(g_2)\\
+&\sum_{i,j}b^{\sigma}_{ij}\langle\boldsymbol{h}_i^{\sigma},\boldsymbol{f}_1^{\sigma}\rangle\boldsymbol{h}_i^{2,c,\sigma}(g_2).
\end{aligned}
\end{equation}
Our assumption on the algebraicity of $\boldsymbol{f}$ implies 
\[
\left(\boldsymbol{f}_1^{2,\sigma}\right)^{c}=\boldsymbol{f}_1^{2,c}\text{ and }\left(\Omega\cdot\boldsymbol{f}_{1}^{2,c}\right)^{\sigma}=\Omega\cdot\boldsymbol{f}_{1}^{2,c}.
\]
Comparing \eqref{III} and \eqref{IV} we conclude that
\[
\frac{\boldsymbol{Z}(s,\boldsymbol{f},f_s)}{\Omega\cdot\pi^{d(F)d(\pi)}\langle\boldsymbol{f},\boldsymbol{f}\rangle\langle\boldsymbol{f}^2_1|\boldsymbol{V},\boldsymbol{f}^2_1\rangle}=\left(\frac{\boldsymbol{Z}(s,\boldsymbol{f},f_s)}{\Omega\cdot\pi^{d(F)d(\pi)}\langle\boldsymbol{f},\boldsymbol{f}\rangle\langle\boldsymbol{f}^2_1|\boldsymbol{V},\boldsymbol{f}^2_1\rangle}\right)^{\sigma}.
\]

We finally conclude the theorem by the integral representation in Corollary \ref{corollary 6.4}. Assume $\boldsymbol{c}>0$ (the case $\boldsymbol{c}=0$ is similar), the term $\alpha(\boldsymbol{p})^{2\boldsymbol{n}-2}$ and the constant $C'$ are algebraic so that
\[
\begin{aligned}
&\frac{c_l(s)^{d(F)}L\left(s+\frac{1}{2},\boldsymbol{f}\times\chi\right)}{\Omega\cdot\pi^{d(F)d(\pi)}\langle \boldsymbol{f},\boldsymbol{f}\rangle}=\left(\frac{c_l(s)^{d(F)}L\left(s+\frac{1}{2},\boldsymbol{f}\times\chi\right)}{\Omega\cdot\pi^{d(F)d(\pi)}\langle\boldsymbol{f},\boldsymbol{f}\rangle}\right)^{\sigma}.
\end{aligned}
\]
The theorem then follows by the easy calculation of the power of $\pi$.
\end{proof}

When $r=0$, we can define the action of $\sigma\in\mathrm{Gal}(\overline{\Q}/F)$ on $\boldsymbol{f}\in\mathcal{S}_{\boldsymbol{l}}(K(\mathfrak{n}\boldsymbol{p}),\overline{\Q})$ on the Fourier coefficients of $\boldsymbol{f}$. In this case we have the following refined version of above theorem.

\begin{thm}
\label{theorem 8.2}
Assume $r=0$. Let $l$ and $s_0$ as in \eqref{8.1}, \eqref{8.2}. For $\boldsymbol{c}>0$ and $\sigma\in\mathrm{Gal}(\overline{\Q}/F)$we have\\
(Case II, Symplectic)
\[
\left(\frac{\chi(\mathfrak{n}_1)^{m\boldsymbol{d}_1}L\left(s_0+\frac{1}{2},\boldsymbol{f}\times\chi\right)}{\pi^{d(F)\boldsymbol{d}(\pi)}i^mG(\chi)^{m+1}\langle\boldsymbol{f},\boldsymbol{f}\rangle}\right)^{\sigma}=\frac{\chi^{\sigma}(\mathfrak{n}_1)^{m\boldsymbol{d}_1}L\left(s_0+\frac{1}{2},\boldsymbol{f}^{\sigma}\times\chi^{\sigma}\right)}{\pi^{d(F)\boldsymbol{d}(\pi)}i^mG(\chi^{\sigma})^{m+1}\langle\boldsymbol{f}^{\sigma},\boldsymbol{f}^{c\sigma c}\rangle},
\]
(Case III, Quaternionic Orthogonal)
\[
\left(\frac{\chi(\mathfrak{n}_1)^{m\boldsymbol{d}_1}L\left(s_0+\frac{1}{2},\boldsymbol{f}\times\chi\right)}{\pi^{d(F)\boldsymbol{d}(\pi)}G^F(\chi)G(\chi)^m\langle\boldsymbol{f},\boldsymbol{f}\rangle}\right)^{\sigma}=\frac{\chi^{\sigma}(\mathfrak{n}_1)^{m\boldsymbol{d}_1}L\left(s_0+\frac{1}{2},\boldsymbol{f}^{\sigma}\times\chi^{\sigma}\right)}{\pi^{d(F)\boldsymbol{d}(\pi)}G^F(\chi^{\sigma})G(\chi^{\sigma})^m\langle\boldsymbol{f}^{\sigma},\boldsymbol{f}^{c\sigma c}\rangle},
\]
(Case IV, Quaternionic Unitary)
\[
\left(\frac{\chi(\mathfrak{n}_1)^{m\boldsymbol{d}_1}L\left(s_0+\frac{1}{2},\boldsymbol{f}\times\chi\right)}{\pi^{d(F)\boldsymbol{d}(\pi)}G(\chi)^m\langle\boldsymbol{f},\boldsymbol{f}\rangle}\right)^{\sigma}=\frac{\chi^{\sigma}(\mathfrak{n}_1)^{m\boldsymbol{d}_1}L\left(s_0+\frac{1}{2},\boldsymbol{f}^{\sigma}\times\chi^{\sigma}\right)}{\pi^{d(F)\boldsymbol{d}(\pi)}G(\chi^{\sigma})^m\langle\boldsymbol{f}^{\sigma},\boldsymbol{f}^{c\sigma c}\rangle},
\]
(Case V, Unitary)
\[
\left(\frac{\chi(\mathfrak{n}_1)^{m\boldsymbol{d}_1}L\left(s_0+\frac{1}{2},\boldsymbol{f}\times\chi\right)}{\pi^{d(F)\boldsymbol{d}(\pi)}G(\chi)^m\Omega\cdot\langle\boldsymbol{f},\boldsymbol{f}\rangle}\right)^{\sigma}=\frac{\chi^{\sigma}(\mathfrak{n}_1)^{m\boldsymbol{d}_1}L\left(s_0+\frac{1}{2},\boldsymbol{f}^{\sigma}\times\chi^{\sigma}\right)}{\pi^{d(F)\boldsymbol{d}(\pi)}G(\chi^{\sigma})^m\Omega\cdot\langle\boldsymbol{f}^{\sigma},\boldsymbol{f}^{c\sigma c}\rangle}.
\]
Here we use the superscript $c$ to mean $\boldsymbol{f}^c(g)=\overline{\boldsymbol{f}(g^{\iota})}$. When $\boldsymbol{c}=0$, one replace $L\left(s_0+\frac{1}{2},\boldsymbol{f}\times\chi\right)$ by $L\left(s_0+\frac{1}{2},\boldsymbol{f}\times\chi\right)M\left(s_0+\frac{1}{2},\boldsymbol{f}\times\chi\right)$ in above formulas.
\end{thm}

\begin{proof}
We omit as it can be proved by the similar argument as in Theorem \ref{theorem 8.1} (see also \cite[Appendix]{BS}).
\end{proof}

\begin{rem}
\label{remark 8.3}
We do not obtain above theorem in general for any $r$ because we do not have a well defined action of $\sigma\in\mathrm{Gal}(\overline{\Q}/F)$. If for a field $\Psi\subset\overline{\Q}$, one can define the meaning of $\mathcal{M}_{\boldsymbol{l}}(K,\Psi)\subset\mathcal{M}_{\boldsymbol{l}}(K,\overline{\Q})$ properly such that $\mathcal{M}_{\boldsymbol{l}}(K,\overline{\Q})=\mathcal{M}_{\boldsymbol{l}}(K,\Psi)\otimes_{\Psi}\overline{\Q}$, then one can further define the action $\mathrm{Gal}(\overline{\Q}/\Psi)$ on $\mathcal{M}_{\boldsymbol{l}}(K,\overline{\Q})$ by acting on $\overline{\Q}$. If this action preserves the subspace $\mathcal{S}_{\boldsymbol{l}}(K,\Psi)$, then one can refine Theorem \ref{theorem 8.1} to obtain similar formulas as in Theorem \ref{theorem 8.2} for $\sigma\in\mathrm{Gal}(\overline{\Q}/\Psi)$.
\end{rem}

\subsection{Preliminary on $p$-adic $L$-functions }
\label{section 8.2}

We now turn to the $p$-adic interpolation of the special value $L(s_0+\frac{1}{2},\boldsymbol{f}\times\chi)$.  For our specified prime $\boldsymbol{p}$, let $p$ be the prime number under $\boldsymbol{p}$ and $\C_p=\widehat{\overline{\Q}}_p$ the completion of $\overline{\Q}_p$. Fix an embedding $\overline{\Q}\to\C_p$. The $p$-adic absolute value $|\cdot|_p$ naturally extends to $\C_p$ and we denote
\begin{equation}
\label{8.2.1}
\mathcal{O}_{\C_p}=\{x\in\C_p:|x|_p\leq 1\}.
\end{equation}

Consider the $p$-adic analytic group
\begin{equation}
\label{8.2.2}
\mathrm{Cl}_E^+(\boldsymbol{p}^{\infty})=E^{\times}\backslash\mathbb{A}_E^{\times}/U(\boldsymbol{p}^{\infty})E^+_{\infty}
\end{equation}
where $U(\boldsymbol{p}^{\infty})$ is the group of elements of $\hat{\mathfrak{o}}^{\times}$ that are congruent to $1$ mod $\boldsymbol{p}^n$ for all integers $n$ with $\hat{\mathfrak{o}}$ the completion of $\mathfrak{o}$ and $E^+_{\infty}$ the connected component of the identity in $E_{\infty}=E\otimes_{\Q}\R$. We refer the reader to \cite[Section 10.2]{CW} for more details on the geometry of this space and the locally analytic functions on this space. We denote by $\mathcal{A}(\mathrm{Cl}_E^{+}(\boldsymbol{p}^{\infty}),\C_p)$ the space of locally analytic functions on $\mathrm{Cl}_E^+(\boldsymbol{p}^{\infty})$ and the space of $p$-adic distributions $\mathcal{D}(\mathrm{Cl}_E^+(\boldsymbol{p}^{\infty}),\C_p)$ are defined as the topological dual of $\mathcal{A}(\mathrm{Cl}_E^{+}(\boldsymbol{p}^{\infty}),\C_p)$. Clearly there is a natural pairing
\begin{equation}
\label{8.2.3}
\begin{aligned}
\mathcal{A}(\mathrm{Cl}_E^{+}(\boldsymbol{p}^{\infty}),\C_p)\times\mathcal{D}(\mathrm{Cl}_E^+(\boldsymbol{p}^{\infty}),\C_p)&\to\C_p,\\
(f,\mu)&\mapsto\mu(f)=:\int_{\mathrm{Cl}_E^{+}(\boldsymbol{p}^{\infty})}fd\mu.
\end{aligned}
\end{equation}
A $p$-adic distribution is called a $p$-adic measure if it is bounded.

The Hecke character $\boldsymbol{\chi}$ of conductor $\boldsymbol{p}^{\boldsymbol{c}}$ defines a locally analytic function on $\mathrm{Cl}_E^{+}(\boldsymbol{p}^{\infty})$. Since they forms a dense subspace of $\mathcal{A}(\mathrm{Cl}_E^{+}(\boldsymbol{p}^{\infty}),\C_p)$, a $p$-adic distribution is uniquely determined by its value at all these Hecke characters. In following two subsections, we define the $p$-adic distribution interpolating the special value $L(s_0+\frac{1}{2},\boldsymbol{f}\times\chi)$ for $p$-ordinary $\boldsymbol{f}$ and prove that the distribution we constructed is indeed a $p$-adic measure. Note that Case II, III and Case IV, V are treated separately because of the occurrence of the Hecke $L$-function for the Fourier expansion in Proposition \ref{proposition 7.10} for Case II, III.

We end up this subsection by the following important preliminary lemma.

\begin{lem}
\label{lemma 8.4}
Let $F(h)$ be a modular form on $H(\mathbb{A})$ with a Fourier expansion of the form
\begin{equation}
\label{8.2.4}
F(q_z)=\nu(y^{\ast})^l\sum_{\beta\in S}C(\beta)e_{\infty}(\tau(\beta z))
\end{equation}
with $q_z$ as in \ref{7.4.3}. We further assume that $F(g_1,g_2)\in\mathcal{M}_{\boldsymbol{l}}(\mathcal{K}(\boldsymbol{p}^2),\overline{\Q})\otimes\mathcal{M}_{\boldsymbol{l}}(\mathcal{K}'(\boldsymbol{p}^2),\overline{\Q})$ with notation as in Section \ref{section 6}. \\
(1) Let $\{\boldsymbol{f}_i\}$ be a basis of $\mathcal{M}_{\boldsymbol{l}}(K(\mathfrak{n}\boldsymbol{p}^2),\overline{\Q})$ and denote $\boldsymbol{f}_i^1$ (resp. $\boldsymbol{f}_i^2$) such that $\boldsymbol{f}_i^1(g)=\boldsymbol{f}_i(g\eta_1\eta_{\boldsymbol{p}})$ (resp. $\boldsymbol{f}_i^2(g)=\boldsymbol{f}_i(g\eta_2)$) with $\eta_1,\eta_2$ in \eqref{6.1} and $\eta_{\boldsymbol{p}}$ in \eqref{6.15}. Then there exists some constants $a_{ij}$ such that
\begin{equation}
\label{8.2.5}
F(g_1,g_2)=\sum_{i,j}a_{ij}\boldsymbol{f}_i^1(g_1)\boldsymbol{f}_j^2(g_2).
\end{equation}
(2) There exist a constant $\Omega_{\boldsymbol{p}}\in\overline{\Q}^{\times}$ independent of $F$ such that if $C(\beta)\in\mathcal{O}_{\C_p}$ then $a_{ij}\in\Omega\cdot\Omega_{\boldsymbol{p}}\cdot\mathcal{O}_{\C_p}$ where $\Omega=1$ in Case II, III, IV and $\Omega$ is the CM period \eqref{cmperiod} in Case V.
\end{lem}

\begin{proof}
This lemma is a $p$-adic analogue of \cite[Lemma 24.11, Lemma 26.12]{Sh00}. The first part is already proved there and it is also shown that if $C(\beta)\in\overline{\Q}$ then $a_{ij}\in\Omega\cdot\overline{\Q}$. We descent the argument there to $\mathcal{O}_{\C_p}$.

Let $\{\boldsymbol{h}_i\}$ be a basis of $\mathcal{M}_{\boldsymbol{l}}^H(\mathcal{K},\overline{\Q})$ (i.e. space of algebraic modular forms over $H$) where $\mathcal{K}$ is the image of $\mathcal{K}(\boldsymbol{p}^2)\times\mathcal{K}'(\boldsymbol{p}^2)$ under doubling map. We can write
\[
F(h)=\sum_{i}A_i\cdot\boldsymbol{h}_i(h)\text{ for }A_i\in\C
\]
and note that each $\boldsymbol{h}_i$ has a Fourier expansion of the form
\[
\boldsymbol{h}_i(q_z)=\nu(y^{\ast})^l\sum_{\beta\in S}c_i(\beta)e_{\infty}(\tau(\beta z)).
\]
There exist a system $\{\beta_i\}$ such that the matrix $[c_i(\beta_k)]_{ik}$ is of full rank. For each $i,k$, $c_i(\beta_k)\in\overline{\Q}$ by the algebraicity of $\boldsymbol{h}_i$ and we can pick a constant $\Omega_1$ depending on $\{\beta_i\}$ and $\{\boldsymbol{h}_i\}$ such that $c_i(\beta_k)\in\Omega_1\mathcal{O}_{\C_p}$. Then $C(\beta)\in\mathcal{O}_{\C_p}$ implies $A_i\in\Omega_1^{-1}\mathcal{O}_{\C_p}$.

Choose a system of CM points $\{g_i\}$ of $G$ such that the matrix $X=[\boldsymbol{f}_i(g_k)]_{ik}$ is of full rank. Note that for any $k$, $(g_k,g_k)$ is a CM point of $H$ so that $\boldsymbol{h}_i((g_k,g_k))\in\mathcal{P}((g_k,g_k))\overline{\Q}$ where $\mathcal{P}((g_k,g_k))$ is the period of CM points over $H$. There exist a constant $\Omega_2$ depending on $\{g_k\}$ and $\{\boldsymbol{h}_i\}$ such that 
\[
\boldsymbol{h}_i((g_k,g_k))\in\Omega_2\mathcal{P}((g_k,g_k))\mathcal{O}_{\C_p}.
\]
Then $C(\beta)\in\mathcal{O}_{\C_p}$ further implies
\[
F((g_k,g_k))\in\Omega_1^{-1}\Omega_2\mathcal{P}((g_k,g_k))\mathcal{O}_{\C_p}.
\]

By the algebraicity of $\boldsymbol{f}_i$, we have $\boldsymbol{f}_i(g_k)\in\mathcal{P}(g_k)\overline{\Q}$ where $\mathcal{P}(g_k)$ is the period of CM points over $G$. We can choose a constant $\Omega_3$ depending on $\{g_k\}$ and $\{\boldsymbol{f}_i\}$ such that
\[
\boldsymbol{f}_i(g_k)\in\Omega_3\mathcal{P}(g_k)\mathcal{O}_{\C_p}.
\]
Now write $F$ as in \eqref{8.2.5} and compare the period $\mathcal{P}(g_k,g_k),\mathcal{P}(g_k)$ as in the proof of \cite[Lemma 26.12]{Sh00}, we conclude that $C(\beta)\in\mathcal{O}_{\C_p}$ implies
\[
a_{ij}\in\Omega\cdot\Omega_1^{-1}\Omega_2\Omega_3^{-1}\cdot\mathcal{O}_{\C_p}.
\]
Take $\Omega_{\boldsymbol{p}}=\Omega_1^{-1}\Omega_2\Omega_3^{-1}$ which is clearly independent of $F$ by our above constructions and the lemma follows.
\end{proof}

\subsection{$p$-adic $L$-functions for unitary and quaternionic unitary groups}

Denote $\Omega=1$ for quaternionic unitary groups and $\Omega$ is the CM period \eqref{cmperiod} for unitary groups. Fix $\chi_1$ be a Hecke character of conductor $\mathfrak{n}_2$ and infinity type $\boldsymbol{l}$. We define a $p$-adic distribution $\mu(\boldsymbol{f})$ such that for any Hecke character $\boldsymbol{\chi}$ of conductor $\boldsymbol{p^c}$, 
\begin{equation}
\label{8.11}
\begin{aligned}
\int_{\mathrm{Cl}_E^+(\boldsymbol{p}^{\infty})}\boldsymbol{\chi}d\mu(\boldsymbol{f})&:=\alpha(\boldsymbol{p})^{2-2\boldsymbol{n}}C'^{-1}|\boldsymbol{\varpi}|^{\boldsymbol{c}\mathbf{d}_1\frac{m(m-1)}{2}}G(\chi)^{-m}\pi^{d(F)d(\pi)}\\
&\times\left(\prod_{i=0}^{n-1}\Gamma(\mathbf{d}_1(l-i))\right)^{d(F)}\cdot\frac{\boldsymbol{Z}(s_0;\boldsymbol{f},f_s,\chi,\boldsymbol{n})}{\Omega\cdot\langle\boldsymbol{f},\boldsymbol{f}\rangle^2}.
\end{aligned}
\end{equation}
Here we are again denoting $\chi=\boldsymbol{\chi}\chi_1$ when $\boldsymbol{\chi}$ varying. The right hand side in above formula is indeed independent of $\boldsymbol{n}$ and $\mu(\boldsymbol{f})$ is a well-defined $p$-adic distribution. 

Assume $\boldsymbol{\chi}$ is of finite order and $\chi$ has infinity type $\boldsymbol{l}$. By \eqref{6.16}, \eqref{6.17} we have for $\boldsymbol{c}>0$,
\begin{equation}
\label{8.12}
\begin{aligned}
\int_{\mathrm{Cl}_E^+(\boldsymbol{p}^{\infty})}\boldsymbol{\chi}d\mu(\boldsymbol{f})&=|\boldsymbol{\varpi}|^{\boldsymbol{c}\mathbf{d}_1\frac{m(m-1)}{2}}G(\chi)^{-m}\pi^{d(F)d(\pi)}\\
&\times\left(c_l(s_0)\prod_{i=0}^{n-1}\Gamma(\mathbf{d}_1(l-i))\right)^{d(F)}\cdot\frac{L\left(s_0+\frac{1}{2},\boldsymbol{f}\times\chi\right)}{\Omega\cdot\langle\boldsymbol{f},\boldsymbol{f}\rangle},
\end{aligned}
\end{equation}
and for $\boldsymbol{c}=0$,
\begin{equation}
\label{8.13}
\begin{aligned}
\int_{\mathrm{Cl}_E^+(\boldsymbol{p}^{\infty})}\boldsymbol{\chi}d\mu(\boldsymbol{f})&=G(\chi)^{-m}\pi^{d(F)d(\pi)}M\left(s_0+\frac{1}{2},\boldsymbol{f}\times\chi\right)\\
&\times\left(c_l(s_0)\prod_{i=0}^{n-1}\Gamma(\mathbf{d}_1(l-i))\right)^{d(F)}\cdot\frac{L\left(s_0+\frac{1}{2},\boldsymbol{f}\times\chi\right)}{\Omega\cdot\langle\boldsymbol{f},\boldsymbol{f}\rangle}.
\end{aligned}
\end{equation}

\begin{thm}
\label{theorem 8.5}
Assume $\boldsymbol{f}$ is $p$-ordinary in the sense that $\alpha(\boldsymbol{p})\in\mathcal{O}_{\C_p}^{\times}$. Then $\mu(\boldsymbol{f})$ defined above is a $p$-adic measure.
\end{thm}

\begin{proof}
The proof is similar to \cite[Section 9]{BS}. Indeed, by Lemma \ref{lemma 8.4}, the boundness of the distribution $\mu(\boldsymbol{f})$ defined above follows from the boundness of the Fourier coefficients of Eisenstein series which can be checked straightforwardly from explicit formulas in Proposition \ref{proposition 7.1}. One can also verify $\mu(\boldsymbol{f})$ is a $p$-adic measure by checking the Kummer congruences following \cite{CP96}. For more details see also \cite[Theorem 6.4]{JYB2}, in which we prove the Kummer congruences for totally isotropic quaternionic unitary groups when $F=\Q$.
\end{proof}

\subsection{$p$-adic $L$-functions for symplectic and quaternionic orthogonal groups}

As we have mentioned before, the Case II, III are different to Case IV, V because of the occurrence of the Hecke $L$-function for the Fourier expansion in Proposition \ref{proposition 7.10}. Therefore, we treat these two cases by the known $p$-adic interpolation of Hecke $L$-functions as in \cite[Section 8]{BS} and \cite{LZ20}. 

We first recall some fact about Hecke $L$-functions. Let $\psi:F^{\times}\backslash\mathbb{A}_F^{\times}\to\C^{\times}$ be any Hecke character trivial at infinity with conductor $c(\psi)$. We assume for simplicity that $l,m$ has the same parity, i.e. $l\equiv m\text{ mod }2$ in Case II so that $s_0+\frac{1}{2}$ is always even. In this case, there is a functional equation (\cite[Theorem 18.12]{Sh00})
\begin{equation}
\label{8.14}
L\left(s_0+\frac{1}{2},\psi\right)=\left(\frac{(2\pi i)^{s_0+\frac{1}{2}}}{2\Gamma\left(s_0+\frac{1}{2}\right)}\right)^{d(F)}\frac{\mathfrak{D}_F^{1/2}G^F(\psi)}{N_{F/\Q}(c(\psi))^{s_0-\frac{1}{2}}}\cdot L\left(\frac{1}{2}-s_0,\psi^{-1}\right)
\end{equation}

We denote
\begin{equation}
\label{8.18}
L_{\boldsymbol{p}}\left(s_0+\frac{1}{2},\psi\right)=\left\{\begin{array}{cc}
1 & \boldsymbol{p}|c(\psi),\\
\left(1-\psi(\boldsymbol{\varpi})|\boldsymbol\varpi|^{s_0+\frac{1}{2}}\right)^{-1} & \boldsymbol{p}\nmid c(\psi),
\end{array}\right.
\end{equation}
for the local $L$-factor at $\boldsymbol{p}$.

When $c(\psi)$ is coprime to $\boldsymbol{p}$, there is a $p$-adic measure $\mu(\psi)$ (see for example \cite{B78, CN79, DR80}) such that for all Hecke character $\boldsymbol{\chi}$ of conductor $\boldsymbol{p^c}$,
\begin{equation}
\label{8.15}
\int_{\mathrm{Cl}_F^+(\boldsymbol{p}^{\infty})}\boldsymbol{\chi}\mu(\psi)=L_{\boldsymbol{p}}\left(\frac{1}{2}-s_0,\psi^{-1}\boldsymbol{\chi}\right)\cdot L\left(\frac{1}{2}-s_0,\psi^{-1}\boldsymbol{\chi}\right).
\end{equation}
The existence of such measure is equivalent to the existence of Kummer congruences (\cite[Proposition 1.7]{CP96}). In particular, for some constant $C\in\C_p$ with $C\cdot\boldsymbol{\chi}(x)\in \mathcal{O}_{\C_p}$ for all $x\in\mathrm{Cl}_F^+(\boldsymbol{p}^{\infty})$, we have
\begin{equation}
\label{8.16}
C\cdot L_{\boldsymbol{p}}\left(\frac{1}{2}-s_0,\psi^{-1}\boldsymbol{\chi}\right)\cdot L\left(\frac{1}{2}-s_0,\psi^{-1}\boldsymbol{\chi}\right)\in\mathcal{O}_{\C_p}
\end{equation}

Fix $\chi_1$ be a Hecke character of conductor $\mathfrak{n}_2$ and infinity type $\boldsymbol{l}$. We define a $p$-adic distribution $\mu(\boldsymbol{f})$ such that for any Hecke character $\boldsymbol{\chi}$ of conductor $\boldsymbol{p^c}$, 
\begin{equation}
\label{8.17}
\begin{aligned}
\int_{\mathrm{Cl}_F^+(\boldsymbol{p}^{\infty})}\boldsymbol{\chi}d\mu(\boldsymbol{f})&:=\alpha(\boldsymbol{p})^{2-2\boldsymbol{n}}
C'^{-1}|\boldsymbol{\varpi}|^{\boldsymbol{c}\mathbf{d}_1\frac{m(m-1)}{2}}N_{F/\Q}(\boldsymbol{p})^{\boldsymbol{c}\left(s_0-\frac{1}{2}\right)}\pi^{d(F)d(\pi)}\\
&\times G(\chi)^{-m}G^F(\chi)^{-1}\left(\Gamma\left(s_0+\frac{1}{2}\right)\prod_{i=0}^{n\mathbf{d}_1-1}\Gamma\left(l-\frac{i}{2}\right)\right)^{d(F)}\\
&\times\frac{L_{\boldsymbol{p}}\left(s_0+\frac{1}{2},\chi\right)}{L_{\boldsymbol{p}}\left(\frac{1}{2}-s_0,\chi^{-1}\right)}\cdot\frac{\boldsymbol{Z}(s_0;\boldsymbol{f},f_s,\chi,\boldsymbol{n})}{\langle\boldsymbol{f},\boldsymbol{f}\rangle}.
\end{aligned}
\end{equation}

Assume $\boldsymbol{\chi}$ is of finite order and $\chi$ has infinity type $\boldsymbol{l}$. By \eqref{6.16}, \eqref{6.17} we have for $\boldsymbol{c}>0$,
\begin{equation}
\label{8.19}
\begin{aligned}
\int_{\mathrm{Cl}_F^+(\boldsymbol{p}^{\infty})}\boldsymbol{\chi}d\mu(\boldsymbol{f})&=|\boldsymbol{\varpi}|^{\boldsymbol{c}\mathbf{d}_1\frac{m(m-1)}{2}}N_{F/\Q}(\boldsymbol{p})^{\boldsymbol{c}\left(s_0-\frac{1}{2}\right)}G(\chi)^{-m}G^F(\chi)^{-1}\pi^{d(F)d(\pi)}\\
&\times\left(c_l(s_0)\Gamma\left(s_0+\frac{1}{2}\right)\prod_{i=0}^{n\mathbf{d}_1-1}\Gamma\left(l-\frac{i}{2}\right)\right)^{d(F)}\cdot\frac{1-\chi^{-1}(\boldsymbol{\varpi})|\boldsymbol{\varpi}|^{\frac{1}{2}-s_0}}{1-\chi(\boldsymbol{\varpi})|\boldsymbol{\varpi}|^{s_0+\frac{1}{2}}}\\
&\times\frac{L\left(s_0+\frac{1}{2},\boldsymbol{f}\times\chi\right)}{\langle\boldsymbol{f},\boldsymbol{f}\rangle},
\end{aligned}
\end{equation}
and for $\boldsymbol{c}=0$,
\begin{equation}
\label{8.20}
\begin{aligned}
&\int_{\mathrm{Cl}_F^+(\boldsymbol{p}^{\infty})}\boldsymbol{\chi}d\mu(\boldsymbol{f})\\
=&|\boldsymbol{\varpi}|^{\boldsymbol{c}\mathbf{d}_1\frac{m(m-1)}{2}}N_{F/\Q}(\boldsymbol{p})^{\boldsymbol{c}\left(s_0-\frac{1}{2}\right)}G(\chi)^{-m}G^F(\chi)^{-1}\pi^{d(F)d(\pi)}\\
\times&\left(c_l(s_0)\Gamma\left(s_0+\frac{1}{2}\right)\prod_{i=0}^{n\mathbf{d}_1-1}\Gamma\left(l-\frac{i}{2}\right)\right)^{d(F)}\\
\times&M\left(s_0+\frac{1}{2},\boldsymbol{f}\times\chi\right)\cdot\frac{L\left(s_0+\frac{1}{2},\boldsymbol{f}\times\chi\right)}{\langle\boldsymbol{f},\boldsymbol{f}\rangle}.
\end{aligned}
\end{equation}

\begin{thm}
\label{theorem 8.6}
Assume $l\equiv m\text{ mod }2$ in Case II and $\boldsymbol{p}$ splits in Case III when $r=1$. Assume $\boldsymbol{f}$ is $p$-ordinary in the sense that $\alpha(\boldsymbol{p})\in\mathcal{O}_{\C_p}^{\times}$. Then $\mu(\boldsymbol{f})$ defined above is a $p$-adic measure.
\end{thm}

\begin{proof}
The argument for checking $\mu(\boldsymbol{f})$ is again similar to \cite[Section 9]{BS} or \cite{CP96}. We have also proved the Kummer congruences for isotropic quaternionic orthogonal groups when $F=\Q$ in \cite[Theorem 6.5]{JYB2}. The main difference between Case II, III and Case IV, V is the occurrence of following term
\[
\prod_{v\nmid\mathfrak{n}}L_v\left(s_0+\frac{1}{2},\chi\lambda_{\beta}\right)
\]
for the Fourier expansion in Proposition \ref{proposition 7.10}. We now explain how to deal with this term. When $r=1$ in Case III, we take the product only for those $v$ splits in $D$, but this does not change our argument. Comparing \eqref{8.11} and \eqref{8.17}, notice that we have multiply a term
\[
N_{F/\Q}^{\boldsymbol{c}\left(s_0-\frac{1}{2}\right)}G^F(\chi)^{-1}\Gamma\left(s_0+\frac{1}{2}\right)^{d(F)}\frac{L_{\boldsymbol{p}}\left(s_0+\frac{1}{2},\chi\right)}{L_{\boldsymbol{p}}\left(\frac{1}{2}-s_0,\chi^{-1}\right)}.
\] 
By the functional equation \eqref{8.14}, after multiplying above term and cancel out the power of $\pi$ it remains to consider
\[
\frac{G^F(\chi\lambda_{\beta})}{G^F(\chi)}\cdot\prod_{v|\mathfrak{n}}L_v\left(s_0+\frac{1}{2},\chi\lambda_{\beta}\right)^{-1}\cdot\frac{L_{\boldsymbol{p}}\left(s_0+\frac{1}{2},\chi\right)}{L_{\boldsymbol{p}}\left(\frac{1}{2}-s_0,\chi^{-1}\right)}L\left(\frac{1}{2}-s_0,\chi^{-1}\lambda_{\beta}^{-1}\right).
\]
Under our assumption ($\boldsymbol{p}$ splits in Case III when $r=1$), $\nu(\beta)$ is always a square mod $\boldsymbol{p}$ (Remark \ref{remark 7.11}) so that $\lambda_{\beta}(\boldsymbol{\varpi})=1$. Using the $p$-adic interpolation of above Hecke $L$-function, especially \eqref{8.16}, one checks that above term is in $\mathcal{O}_{\C_p}$ up to a bounded constant. Then our theorem follows from the explicit formulas for Fourier expansion of Eisenstein series in Proposition \ref{proposition 7.1} and Lemma \ref{lemma 8.4}. 
\end{proof}

We give a final remark on what we have not done in this paper.

\begin{rem}
$\text{ }$\\
(1) In this paper we have only consider one critical point at $s_0$. Of course one may also discuss other critical points by the standard process of applying differential operators. There are two approaches for applying the differential operators. One is following \cite{CP96, LZ20, Sh00}, in which the differential operator studied in \cite{Sh94} is applied. This kind of differential operators are defined for all classical groups discussed here but one need to consider the nearly holomorphic Eisenstein series and apply the holomorphic projection. Another approach is following \cite{BS}, in which the holomorphic differential operator constructed in \cite{B85} is used and the holomorphic projection is avoided. The differential operator constructed there can also generalized to other groups with $r=0$ (see for example \cite{JYB2} for the quaternionic unitary case). However, we do not know whether one can construct such differential operators for general groups with $r>0$. \\
(2) For the $p$-adic $L$-functions, we have only computed the interpolation at $\chi$ as assumed at the beginning of Section \ref{section 8}. In particular, we have only considered $\chi$ of infinity type $\boldsymbol{l}$ coincided with the weight of modular forms $\boldsymbol{f}$. This is because in our integral representation, we need the weight of Eisenstein series (which equals the infinity type of $\chi$) coincide with the weight of modular forms. To consider Hecke characters of other infinity type, we will need applying the differential operators on the Eisenstein series to shift the weight.\\
(3) We have only considered the parallel weight in this section. Especially, our archimedean computations have only done for scalar weight. For the general weight, \cite{EL} and \cite{LZ21} have computed the archimedean integrals for unitary and symplectic groups. Also in \cite{Pi21,Pi22} the archimedean integrals are calculated for symplectic group in a different way. 
\end{rem}

\section*{Acknowledgement}

This paper is prepared as part of my thesis at Durham University. I would like to express my gratitude to my supervisor, Thanasis Bouganis, for his valuable suggestions. I thank my thesis examiners Tobias Berger and Herbert Gangl for providing several helpful comments. I also thank Alexei Panchishkin for inviting me to the Fourier Institute and the nice discussion there with him and Siegfried Böcherer.

\end{document}